\documentclass[11pt]{article}
\pdfoutput=1 
\usepackage[margin=1in]{geometry}
\usepackage{amsthm,amssymb,enumerate, bbm ,graphicx,color,caption,upgreek, float, tikz, subcaption,booktabs,longtable, appendix,graphics, pdfpages,rotating,mathtools, array, bm, blkarray,setspace,textcomp,ytableau, tabularx}
\usepackage{oubraces}

\usepackage[pdftex]{hyperref} 

\usetikzlibrary{positioning,chains,fit,shapes,calc}

\makeatletter 
\newcommand\mynobreakpar{\par\nobreak\@afterheading} 
\makeatother

\makeatletter
\let\@fnsymbol\@arabic
\makeatother
\newtheorem{theorem}{Theorem}[section]
\newtheorem{conjecture}[theorem]{Conjecture}
\newtheorem{proposition}[theorem]{Proposition}
\newtheorem{lemma}[theorem]{Lemma}
\newtheorem{corollary}[theorem]{Corollary}

\newtheorem{remark}[theorem]{Remark}
\newtheorem{example}[theorem]{Example}
\newtheorem{definition}[theorem]{Definition}

\let\OLDthebibliography\thebibliography
\renewcommand\thebibliography[1]{
  \OLDthebibliography{#1}
  \setlength{\parskip}{0pt}
  \setlength{\itemsep}{0pt plus 0.3ex}
}

\newcommand{\N}{\mathbb{N}}

\newcommand{\R}{\mathbb{R}}

\newcommand{\F}{\mathbb{F}}
 
\newcommand{\calP}{{\mathcal P}}
\newcommand{\olE}{\overline E}
\newcommand{\gbbi}{g}
\newcommand{\hbbi}{h}
\newcommand{\Diag}{\text{\rm Diag}}
\newcommand{\diag}{\text{\rm diag}}
\newcommand{\T}{^{\sf T}}
\newcommand{\rank}{\text{\rm rank}}
\newcommand{\rc}{\text{\rm rc}}

\newcommand{\alphabal}{\alpha_{\text{\rm bal}}}

\newtheorem{probl}{Problem}
\newcommand{\Tr}{\text{\rm Tr}}
\DeclarePairedDelimiter\ceil{\lceil}{\rceil}
\DeclarePairedDelimiter\floor{\lfloor}{\rfloor}

\newcommand\norm[1]{\lVert#1\rVert}

\newcommand{\gbal}{g_{\text{\rm bal}}}
\newcommand{\hbal}{h_{\text{\rm bal}}}
\newcommand{\olG}{\overline G}

\newcommand{\IGbal}{I_{G,\text{\rm bal}}}
\newcommand{\IGbaltwo}{I_{G,\text{\rm bal},2}}

\newcommand{\gbbir}{g_{r}}
\newcommand{\hbbir}{h_{r}}
\newcommand{\gbbione}{g_{1}}
\newcommand{\hbbione}{h_{1}}

\newcommand{\hbalr}{h_{\text{\rm bal},r}}
\newcommand{\gbalr}{g_{\text{\rm bal},r}}
\newcommand{\hbalone}{h_{\text{\rm bal},1}}
\newcommand{\gbalone}{g_{\text{\rm bal},1}}
\newcommand{\lasbalr}{\text{\rm las}_{\text{\rm bal},r}}

\newcommand{\gbalonehat}{\widehat{g_{\text{\rm bal}}}}

\newcommand{\lasbalone}{\text{\rm las}_{\text{\rm bal,1}}}
\newcommand{\thetabal}{{\vartheta}_{\text{\rm bal}}}
\newcommand{\thetabalhat}{\widehat{{\vartheta}_{\text{\rm bal}}}}
\newcommand{\lasbalhat}{\widehat{\text{\rm las}_{\text{\rm bal}}}}
\newcommand{\lasbaltilde}{\widetilde{\text{\rm las}_{\text{\rm bal}}}}

\newcommand{\gbi}{g_{\text{\rm bi}}}
\newcommand{\hbi}{h_{\text{\rm bi}}}
\newcommand{\gbc}{g_{\text{\rm bc}}}
\newcommand{\hbc}{h_{\text{\rm bc}}}
\newcommand{\las}{\text{\rm las}}

\definecolor{donkergroen2}{RGB}{0,95,0}

\usepackage[english]{babel}

\title{Semidefinite approximations for bicliques and biindependent pairs}
\author{Monique Laurent\thanks{Centrum Wiskunde \& Informatica (CWI) and Tilburg University. E-mail: \href{mailto:m.laurent@cwi.nl}{\texttt{m.laurent@cwi.nl}}.} \  \& Sven Polak\thanks{Centrum Wiskunde \& Informatica (CWI) and Tilburg University. E-mail: \href{mailto:s.c.polak@tilburguniversity.edu}{\texttt{s.c.polak@tilburguniversity.edu}}.} \ \& Luis Felipe Vargas\thanks{Centrum Wiskunde \& Informatica (CWI) and  Istituto Dalle Molle Studi sull'Intelligenza Artificiale (IDSIA),  USI - SUPSI. E-mail: \href{mailto:luis.vargas@idsia.ch}{\texttt{luis.vargas@idsia.ch}}.}}
\begin{document}
\maketitle
\setcounter{footnote}{1}

\begin{abstract}
We investigate some graph parameters dealing with biindependent pairs $(A,B)$ in a bipartite graph $G=(V_1\cup V_2,E)$, i.e., pairs $(A,B)$ where $A\subseteq V_1$, $B\subseteq V_2$ and $A\cup B$ is independent. These parameters also allow to study bicliques in general graphs. When maximizing the cardinality $|A\cup B|$ one finds the stability number $\alpha(G)$, well-known to be polynomial-time computable. When maximizing the product $|A|\cdot |B|$ one finds the parameter $g(G)$, shown to be NP-hard by Peeters (2003), and when maximizing the ratio $|A|\cdot |B|/|A\cup B|$ one finds $h(G)$, introduced by Vallentin (2020) for bounding product-free sets in finite groups. We show that $h(G)$ is an NP-hard parameter and, as a crucial ingredient, that it is NP-complete to decide whether a bipartite graph $G$ has a balanced maximum independent set. These hardness results motivate introducing semidefinite programming bounds for $g(G)$, $h(G)$, and $\alpha_{\text{bal}}(G)$ (the maximum cardinality of a balanced independent set). We show that these bounds can be seen as natural variations of the Lov\'{a}sz $\vartheta$-number, a well-known semidefinite bound on $\alpha(G)$. In addition we formulate closed-form eigenvalue bounds and we show relationships among them as well as with earlier spectral parameters by Hoffman, Haemers (2001) and Vallentin (2020).

\end{abstract}


\section{Introduction}

Given a  bipartite graph $G=(V_1 \cup V_2, E)$, a \emph{bipartite biindependent pair} in $G$
is a pair $(A,B)$ of subsets  $A\subseteq V_1$ and $B\subseteq V_2$ such that no pair of nodes $\{i,j\}\in A\times B$ is an edge of $G$.
The adjective ``bipartite" is used to indicate that we restrict to the pairs $(A,B)$ that respect the bipartite structure of $G$, i.e., with $A\subseteq V_1$ and $B\subseteq V_2$; we will however sometimes omit it for the sake of brevity.   The maximum sum $|A|+|B|$ taken over all bipartite biindependent pairs  
$(A,B)$  is the well-studied parameter $\alpha(G)$, known as the {\em stability number} of $G$.  
We consider the following two other parameters, asking for the maximum product $|A|\cdot |B|$ and the maximum ratio ${|A|\cdot |B|\over |A|+|B|}$,
\begin{align}
g(G)&:=\max \{|A|\cdot |B|:  
(A,B) \text{ is a bipartite biindependent pair in } G\},\label{eqdefg}\\ 
h(G)&:=\max \Big\{\tfrac{|A|\cdot |B|}{ |A|+|B|}:   
(A,B) \text{ is a  bipartite biindependent  pair in } G\Big\}. \label{eqdefh}
\end{align}
If $G$ is a complete bipartite graph, then any bipartite biindependent pair has $A=\emptyset$ or $B=\emptyset$ (and thus $g(G)=h(G)=0$); such a pair is called {\em trivial}. Otherwise, in the definition of $g(G)$ and $h(G)$, 
one may restrict the optimization to {\em nontrivial} pairs $(A,B)$, i.e., with $A,B\ne \emptyset$. A pair $(A,B)$ is called {\em balanced} if $|A|=|B|$. Then a related parameter of interest is $\alphabal(G)$, the maximum number of vertices in a balanced biindependent pair, given by
$$\alphabal(G):=\max\{ |A|+|B|: (A,B) \text{ is a balanced bipartite biindependent pair in } G\}.$$
One can also define the parameters $\gbal(G)$ and $\hbal(G)$ as the analogs of $g(G)$ and $h(G)$, where one restricts the optimization to balanced pairs in (\ref{eqdefg}) and (\ref{eqdefh}), respectively.
 Here are some easy relations that hold among the above parameters.

\begin{lemma}\label{lemrel1}
Let $G$ be a bipartite graph. Then, we have
\begin{align}
\tfrac{1}{4}\alphabal(G)={1\over 2}\sqrt{\gbal(G)}=\hbal(G) \le h(G)\le \tfrac{1}{2}\sqrt{g(G)} \le \tfrac{1}{4}\alpha(G),\label{eqrel0}\\ 
h(G)={1\over 4}\alpha(G)\Longleftrightarrow {1\over 2}\sqrt{g(G)}={1\over 4}\alpha(G)\Longleftrightarrow \alpha(G)=\alphabal(G),\label{eqrel2} 
\end{align}
\end{lemma}
\proof{Proof.} 
The equalities $\tfrac{1}{4}\alphabal(G)={1\over 2}\sqrt{\gbal(G)}=\hbal(G)$ follow from the definitions. We now show the inequalities in (\ref{eqrel0}).  First, if $(A,B)$ is optimal for $\alphabal(G)$,  then $|A|=|B|$ and thus we have $h(G)\ge {|A|\cdot |B|\over |A|+|B|}=|A|/2=\alphabal(G)/4$. 
Second, if $(A,B)$ is optimal for $h(G)$, then ${1\over 2}\sqrt{g(G)}\ge {1\over 2}\sqrt{|A|\cdot |B|} \ge {|A|\cdot |B|\over |A|+|B|}=h(G)$, where the last inequality holds as $(\sqrt{|A|}-\sqrt{|B|})^2\geq 0$.
Third, if $(A,B)$ is optimal for $g(G)$, then ${1\over 4}\alpha(G)\ge {1\over 4}(|A|+|B|)\ge {1\over 2}\sqrt{|A|\cdot |B|}={1\over 2}\sqrt{g(G)}$, where again the last inequality holds as $(\sqrt{|A|}-\sqrt{|B|})^2\geq 0$. This concludes the proof of (\ref{eqrel0}). Moreover, equality ${1\over 4}\alpha(G)={1\over 2}\sqrt{g(G)}$ implies $|A|=|B|$, and thus  $(A,B)$ is a balanced optimal solution for $\alpha(G)$, so that $\alpha(G)=\alphabal(G)$. In addition, if $h(G)=\frac{1}{4}\alpha(G)$, then ${1\over 4}\alpha(G)={1\over 2}\sqrt{g(G)}$ by (\ref{eqrel0}), which, as we just observed, implies $\alpha(G)=\alphabal(G)$. The other implications follow directly from  (\ref{eqrel0}). 
\endproof

\smallskip
In the rest of this section we first explain how the above parameters also permit to model problems about bicliques (in arbitrary graphs) and we mention some applications. Then we present a roadmap through our main results, that deal with complexity questions, and with designing semidefinite bounds and closed-form eigenvalue-based bounds, topics to which we come back 
in  detail in Sections~\ref{seccomplexity}, \ref{secsdp}, \ref{seceig}, and \ref{sec:balanced}.

\subsection*{Biindependent pairs and bicliques in arbitrary graphs.}

{\em Bipartite} biindependent pairs in {\em bipartite} graphs   
also permit to model {\em general} biindependent pairs and bicliques in {\em arbitrary} graphs. 
Consider an arbitrary graph $G=(V,E)$ (not necessarily bipartite).
A {\em biindependent} pair in $G$ is a pair $(A,B)$ of disjoint subsets of $V$ such that no pair of nodes $\{i,j\}\in A\times B$ is an edge of $G$ (but edges are allowed within  $A$ or  $B$). One then defines analogously the parameters $\gbi(G)$ and $\hbi(G)$, respectively,  as the maximum product $|A|\cdot |B|$  and the maximum  ratio $|A|\cdot |B|\over |A|+|B|$, taken over all biindependent pairs in $G$. The analog of relation (\ref{eqrel0}) holds:
$$\hbi(G)\le {1\over 2}\sqrt{\gbi(G)}\le {1\over 4}|V|.$$ Note that  $\hbi(G)\ge {1\over 4}\alpha(G)$ if $\alpha(G)$ is even and $\hbi(G)\ge {1\over 4}\big(\alpha(G)- {1\over \alpha(G)}\big)$
if $\alpha(G)$ is odd (which can be seen by  partitioning a maximum stable set into two almost equally sized parts).
The parameters $\hbi(G)$ and $\gbi(G)$ can in fact be reformulated in terms of the parameters $g(\cdot)$ and $h(\cdot)$ for an associated bipartite graph
$B_0(G)$, the  {\em extended bipartite double} of $G$, defined as follows.
First we define the {\em bipartite double} $B(G)$, whose node set is $V\cup V'$, where $V'=\{i':i\in V\}$ is a disjoint copy of $V$, and whose  edges are the pairs $\{i,j'\}$ and $\{j,i'\}$ for $\{i,j\}\in E$. 
Then, the {\em extended bipartite double} $B_0(G)$ is obtained by adding  all pairs $\{i,i'\}$ ($i\in V$) as edges to $B(G)$. Now, observe that a pair  $(A,B)$ is  biindependent  in $G$ precisely when the pair $(A\subseteq V,B':=\{i':i\in B\}\subseteq V'$) is  bipartite biindependent  in $B_0(G)$.  
Therefore we have 
\begin{equation}\label{eqbi}
\gbi(G)=g(B_0(G))\quad \text{ and } \quad \hbi(G)=h(B_0(G)) \quad \text{ for any graph } G. 
\end{equation}

One can also model  {bicliques} in an arbitrary  graph $G=(V,E)$.
A {\em biclique} in $G$ is a pair $(A,B)$ of disjoint subsets of $V$ such that $A\times B\subseteq E$ or, equivalently,  $(A,B)$ is a biindependent pair in  the complementary graph $\overline G=(V,\overline E)$ of $G$.  In analogy, let $\gbc(G)$ and $\hbc(G)$ denote the maximum product $|A|\cdot |B|$  and   ratio $|A|\cdot |B|\over |A|+|B|$, taken over all bicliques $(A,B)$ in $G$, so that  
\begin{equation}\label{eqbcB0G}
\gbc(G)=\gbi(\overline G)=g(B_0(\overline G)) \quad \text{ and } \quad \hbc(G)=\hbi(\overline G)=h(B_0(\overline G)) \quad \text{ for any graph } G.
\end{equation}
In the case when $G=(V_1\cup V_2,E)$ is a bipartite graph, nontrivial bicliques in $G$ correspond to nontrivial bipartite biindependent pairs in the bipartite graph  $\olG^b:=(V_1\cup V_2, (V_1\times V_2) \setminus E)$, known as  the {\em bipartite complement} of $G$. So  we also have 
 \begin{equation}\label{eqbcG}
\gbc(G)= g(\olG^b)\quad \text{ and } \quad \hbc(G)=h(\olG^b)\quad \text{ for any graph } G.
 \end{equation}
 
So relations (\ref{eqbcB0G}) and (\ref{eqbcG}) offer different formulations for the parameters $\gbc(\cdot)$ and $\hbc(\cdot)$, we will investigate in Section \ref{secHaemers} how the associated  semidefinite bounds relate. 
 
\subsection*{Complexity results.}
As is well-known there are polynomial-time algorithms for computing the  stability number $\alpha(G)$ of a bipartite graph $G$. On the other hand, 
Peeters \cite{Peeters} shows that, given an integer $k$,  deciding whether a bipartite graph $G$ has a biclique $(A,B)$ with $|A|\cdot |B|\ge k$ is an  NP-complete problem.
Hence, computing the parameter $g(G)$ is  an NP-hard problem (by switching between bicliques and biindependent pairs).

We will show that  also $h(G)$ is hard to compute. 
For this we show that the problem (denoted $\alpha$-BAL-BIP in Section \ref{seccomplexity}) of deciding whether a bipartite graph $G$ has a  {\em balanced} maximum independent set, i.e., whether $\alpha(G)=\alphabal(G)$, is NP-complete 
(see Theorem~\ref{theohardness}).
Combining with Lemma~\ref{lemrel1}, it follows that 
 deciding whether $h(G)\ge {1\over 4}\alpha(G)$ is  an NP-complete  problem.

It is known  that, given an integer $k$, deciding whether  a bipartite graph $G$ contains a  bipartite biindependent pair  $(A,B)$ with $|A|=|B|=k$ is an NP-complete problem~\cite{Garey,Johnson} (switching between biindependent pairs and bicliques). Hence our hardness result for problem $\alpha$-BAL-BIP shows hardness of this problem already   for the case  $k={1\over 2}\alpha(G)$.

Our proof technique will in fact permit to show NP-hardness for a broader set of problems, namely for deciding whether any of the following equalities holds: $g(G)=\gbal(G)$, $h(G)=\hbal(G)$, 
$h(G)={1\over 2}\sqrt{g(G)}$, or ${1\over 2}\sqrt{g(G)}={1\over 4}\alpha(G)$ (thus whether the inequalities in (\ref{eqrel0}) hold at equality). See Theorem \ref{theo-equiv-H} and Corollary \ref{coro-complexity}.

\subsection*{Some applications for the parameters $g(\cdot)$ and $h(\cdot)$.}

As explained above the parameter $g(\cdot)$ also allows to model  {maximum  edge cardinality bicliques} in  bipartite (or general) graphs. 
This problem has many real life applications, such as  reducing assembly times in product manufacturing lines and in the area of formal concept analysis, as explained in \cite{DKST2001} (see also \cite{DKT97,ST1998}).
The related parameter asking for the maximum number of vertices in a balanced biclique has also many applications; e.g., in VLSI design (e.g., \cite{AYRP2007,RL1988,T2006}, in the analysis of biological data (as instance of bicluster, e.g., \cite{YWWY2005}) and of interactions of proteins (e.g., \cite{MRM2014}).

The parameter $g(\cdot)$ appears  naturally in the study of cross-intersecting set (or subspace) families.
For integers $n\ge k\ge \ell\ge 1$, let $\calP_{k}$ denote the collection of $k$-subsets of $[n]$ and similarly for $\calP_{\ell}$. Two set families $\mathcal A\subseteq \calP_{k}$ and $\mathcal B\subseteq \calP_{\ell}$ are called {\em cross-intersecting} if $A\cap B\ne \emptyset$ for all $A\in \mathcal A$ and $B\in \mathcal B$.
Then the maximum product $|\mathcal A|\cdot |\mathcal B|$ for such a cross-intersecting pair is the parameter $g(G^{n}_{k,\ell})$, where $G^n_{k,\ell}$ is the bipartite graph with bipartition $\calP_{k}\cup\calP_{\ell}$ and an edge $(A,B)\in \calP_{k}\times \calP_{\ell}$ if $A\cap B=\emptyset$. Pyber \cite{Pyber1986} gives bounds on this parameter $g(G^n_{k,\ell})$, as extension of the classical Erd\"os-Ko-Rado result \cite{EKR}.
Suda and Tanaka \cite{ST2014} consider the analogous question for subspace families. Given a finite field $\mathbb F_q$ ($q\ge 2$), let $\Omega_k$ denote the set of all $k$-dimensional subspaces of $(\mathbb F_q)^n$ and define analogously $\Omega_\ell$. Consider the bipartite graph $G^{n,q}_{k,\ell}$ with bipartition $\Omega_k\cup \Omega_\ell$, where there is an edge $(A,B)\in \Omega_k\times\Omega_\ell$ if $A\cap B=\{0\}$. Then, biindependent pairs in $G^{n,q}_{k,\ell}$ correspond to cross-intersecting families of subspaces. Suda and Tanaka \cite{ST2014} give bounds on the parameter $g(G^{n,q}_{k,\ell})$ based on semidefinite programming (see also Remark \ref{remSudaTanaka}). We also refer to Suda, Tanaka and Togushide \cite{STT2017}, who consider an extension to cross-intersecting families with measures.

The parameter $g(\cdot)$ is also relevant for bounding the  nonnegative rank of a matrix.
Given a matrix $M\in \R^{|V_1|\times |V_2|}_+$, its  {\em nonnegative rank}  $\rank_+(M)$ is the smallest integer  $r\in\N$ such that $
M=\sum_{\ell=1}^r a_\ell b_{\ell}\T$  for some nonnegative vectors $a_\ell\in\R^{|V_1|}_+$ and $b_\ell\in \R^{|V_2|}_+$;  computing $\rank_+(\cdot)$  is an NP-hard problem \cite{Vavasis}.
A classical combinatorial lower bound for $\rank_+(M)$ is the  {\em rectangle covering bound} $\rc(M)$,  defined as the smallest number of  rectangles $A\times B\subseteq V_1\times V_2$
whose  union is equal to the support  $S_M:=\{(i,j)\in V_1\times V_2: M_{ij}\ne 0\}$ of $M$.
(See, e.g., \cite{FKPT}). The rectangle covering bound was used, e.g., in  \cite{FMPTW} to show an exponential lower bound on the extension complexity of combinatorial polytopes such as the traveling salesman  and correlation  polytopes. Also the parameter $\rc(M)$ is not easy to compute. To approximate it, one can consider the bipartite graph $B_M$, with vertex set $V_1\cup V_2$ and  edge set $E_M:=(V_1\times V_2)\setminus S_M$. Then  one can show that 
$\rc(M) \cdot g(B_M)\ge |S_M|$.
Hence, an upper bound on $g(B_M)$ gives directly a lower bound on $\rc(M)$ and thus a lower bound on the nonnegative rank $\rank_+(M)$.

\medskip
The parameter $h(\cdot)$ was introduced by Vallentin~\cite{Vallentin}, who observed its relevance to maximum product-free subsets in groups in work of Gowers~\cite{Gowers}.
Let~$\Gamma$ be a finite group. A set $A\subseteq \Gamma$ is called {\em product-free} if $ab\not\in A$ for all $a,b\in A$, and one is interested in finding the  largest cardinality of a product-free set in $\Gamma$  (see \cite{Gowers,Kedlaya} for background on this problem).  We now briefly indicate how to bound this parameter using the parameter $h(\cdot)$; for the interested reader we will present this connection in more detail in Appendix \ref{appendixgroup}.

Assume $A\subseteq \Gamma$ is product-free. Let $G_{\Gamma,A}=(V_1 \cup V_2 ,E)$ be the associated bipartite Cayley graph, where $V_1$ and $V_2$ are disjoint copies of $\Gamma$ and there is an edge between $v_1\in V_1$ and $v_2\in V_2$ if their product $v_1v_2$ belongs to $A$.
The crucial observation now is that since~$A$ is product-free, the pair~$(A_1,A_2)$ is  (balanced) bipartite biindependent  in~$G_{\Gamma,A}$, where $A_1\subseteq V_1,A_2\subseteq V_2$ are the corresponding disjoint copies of $A$. This implies  $
\tfrac{|A|}{2} \le h(G_{\Gamma,A})$. Hence, upper bounds on~$h(G_{\Gamma,A})$ give  upper bounds on product-free sets in $\Gamma$.  Vallentin \cite{Vallentin} introduced the eigenvalue-based upper bound $h(G)\le {|V|\over 2r}\lambda_2(A_G)$ for any $r$-regular bipartite graph $G$. Applying it to the $|A|$-regular bipartite graph $G_{\Gamma,A}$, he could recover a result by Gowers~\cite{Gowers}, which states that a product-free subset~$A$ in~$\Gamma$ has cardinality~$|A| \leq |\Gamma|/k^{1/3}$, where $k$ is the minimum dimension
of a nontrivial  representation of~$\Gamma$. 
We will show the sharper eigenvalue-based bound $h(G)\le \widehat h(G)={|V|\over 4}{\lambda_2(A_G)\over r+\lambda(A_G)}$ (see Proposition~\ref{prop:hbbioneedgetransitive}), and use it to show a slight sharpening of Gowers' bound, replacing ${|\Gamma|\over k^{1/3}}$ by ${|\Gamma|\over 1+k^{1/3}}$ (see Theorem~\ref{theorem:Gowers}). 

In fact, for this application, one is only interested in  {\em balanced} biindependent pairs in the graph $G_{\Gamma,A}$ and we have $2|A|\le \alphabal(G_A)$ if $A$ is product-free in $\Gamma$.
This motivates investigating whether sharper semidefinite and eigenvalue-based bounds can be found for the balanced parameters. We come back briefly to this question later in the introduction and it will be investigated in detail in Section \ref{sec:balanced}.

\subsection*{Semidefinite approximations.}

The parameters $g(G)$ and $h(G)$ can  be  formulated as polynomial optimization problems, which leads  to hierarchies of semidefinite programming (SDP) upper bounds $g_r(G)$ and $h_r(G)$ (for $r\ge 1$), able to find the original parameters at order $r=\alpha(G)$. 
We investigate in particular the SDP bounds obtained at the first order $r=1$. As we will see they take the form
\begin{align}
g_1(G)&=\max_{ X \in \mathcal{S}^{|V|}} \Big\{ \langle C, X \rangle:   \,  \left(\begin{matrix} 1 &\diag(X) \T\cr \diag(X)& X\end{matrix}\right) \succeq 0, \, X_{ij} = 0 \text{ if } \{i,j\} \in E\Big\}, \label{eqg1}\\
h_1(G)&=\max_{X\in \mathcal S^{|V|}}\{\langle C,X\rangle:  \,X\succeq 0,\ \text{\rm Tr}(X)=1,\ X_{ij}=0 \text{ if } \{i,j\}\in E\}.\label{eqh1}
\end{align}
\noindent
Here 
$C={1\over 2}{\tiny \left(\begin{matrix} 0 & J \cr J & 0 \end{matrix}\right)}\in \R^{|V_1|+|V_2|}$, where $J$ denotes  the all-ones matrix of appropriate size. 
The parameters  $g_1(G)$ and $h_1(G)$  can  be seen as  variations of the parameter $\vartheta(G)$, introduced by Lov\'asz \cite{Lo79} as upper bound on $\alpha(G)$ for any $G$ (and equal to $\alpha(G)$ when $G$ is bipartite). Indeed, if we replace the objective $\langle C,X\rangle$ by $\text{\rm Tr}(X)$ in  program (\ref{eqg1}) and by $\langle J,X\rangle $   in program (\ref{eqh1}), then we obtain $\vartheta(G)$ in both cases (see (\ref{eqthetaI}) and (\ref{eqthetaJ})).
We  will show the following relations between 
the parameters $h(G)$, $g(G)$, $h_1(G)$,  $g_1(G)$, and $\alpha(G)$. 

\begin{proposition}\label{propghg1h1}
For any bipartite graph $G$ we have 
$$h(G) \leq \tfrac{1}{2}\sqrt{g(G)} \leq  h_1(G) \leq \tfrac{1}{2} \sqrt{g_1(G)} \leq \tfrac{1}{4} \alpha(G).$$
\end{proposition}

It is interesting to note  that   $h_1(G)$ may improve the  bound  ${1\over 2}\sqrt{g_1(G)}$ for  ${1\over 2}\sqrt {g(G)}$. Indeed,  the inequality $h_1(G) \leq \tfrac{1}{2} \sqrt{g_1(G)}$ can be strict, e.g., when $G$ is $K_{n,n}$ minus a perfect matching with~$n \ge 5$, as we see in Section~\ref{secexamples}. The key ingredient to show this is getting eigenvalue-based reformulations for the parameters when $G$ enjoys symmetry properties, as we discuss next.

\subsection*{Eigenvalue bounds.}

When $G$ is  a bipartite $r$-regular graph we can give  closed-form  bounds in terms of the second largest eigenvalue of the adjacency matrix $A_G$ of $G$. These bounds are obtained  by restricting in the definitions (\ref{eqg1}) and (\ref{eqh1}) of  $g_1(G)$ and $h_1(G)$ the optimization to matrices with some symmetry.

\begin{proposition}\label{propeigbound}
Assume~$G$ is a bipartite $r$-regular graph,  set~$n:=|V_1|=|V_2|$, and let $\lambda_2$ be the second largest eigenvalue of the adjacency matrix $A_G$ of $G$. Then we have
$$
g_1(G)\leq \widehat g(G):=\begin{cases}
 \frac{n^2\lambda_2^2}{(\lambda_2+r)^2}  &  \text{ if $r\le 3\lambda_2$},\\
 \frac{n^2\lambda_2}{8(r-\lambda_2)}  &  \text{ otherwise,}
\end{cases}
\quad\text{ and } \quad
h_1(G)\leq \widehat h(G):= \frac{n\lambda_2}{ 2(\lambda_2+r) }.
$$
Moreover,  we have equality $g_1(G)=\widehat g(G)$ if $G$ is  vertex- and edge-transitive,
and equality $h_1(G)=\widehat h(G)$  if $G$ is edge-transitive.
\end{proposition}
Observe that the bound $h(G)\le \widehat h(G)$ sharpens the  bound $ h(G) \le \frac{n}{r} \lambda_2$ by Vallentin  \cite{Vallentin}. Moreover, one can check that $\widehat h(G)\le {1\over 2}\sqrt{\widehat g(G)}$, which  mirrors the inequalities $h(G)\le {1\over 2}\sqrt {g(G)}$ and $h_1(G)\le {1\over 2}\sqrt{g_1(G)}$ (in Proposition \ref{propghg1h1}). We will see in Section \ref{secexamples} several classes of graphs for which strict inequality $\widehat h(G)< {1\over 2}\sqrt{\widehat g(G)}$ holds
  and, in Section \ref{seceig}, we will compare the parameter $\widehat h (\cdot)$ with other eigenvalue bounds by Hoffman and by Haemers \cite{Haemers1997,Haemers2001}.
 
\subsection*{Bounds for the balanced parameters.}

 As we have seen earlier, the parameter $\alphabal(G)$, asking for the maximum cardinality of a balanced independent set in $G$,  arises naturally when considering the parameters $h(\cdot)$ and $g(\cdot)$. An additional motivation comes from its relevance to  product-free sets in groups and other applications as in \cite{AYRP2007,MRM2014,RL1988,T2006,YWWY2005}. 
  The question thus arises of finding semidefinite and eigenvalue-based bounds  for $\alphabal(G)$ (and the related parameters $\hbal(G)$ and $\gbal (G)$) that improve on the bounds $h_1(G)$ and $\widehat h(G)$ designed for the general (not necessarily balanced) parameters. We investigate this question in detail in Section~\ref{sec:balanced}. We define  semidefinite bounds $\lasbalone(G)$ and $\thetabal(G)$ for $\alphabal(G)$,  $\gbalone(G)$ for $\gbal(G)$, and $\hbalone(G)$ for $\hbal(G)$, and we show  they satisfy 
 ${1\over 4}\lasbalone(G)\le {1\over 2}\sqrt{\gbalone(G)}\le \hbalone(G)={1\over 4}\thetabal(G)$ (see Proposition \ref{propcomparebal}).
 Interestingly, the ``balanced versions" of the theta number may lead to different parameters, i.e., 
$ \lasbalone(G) <\thetabal(G)$ may hold (see Example~\ref{remarkstrictbal}). On the other hand, we show that the closed-form values obtained by restricting the optimization to symmetric solutions in each of these semidefinite bounds in fact recover (up to the correct transformation) the eigenvalue bound $\widehat h(G)$ (see Proposition \ref{prop:symbalhat2}).

\subsection*{Organization of the paper.}
The paper is organized as follows. 
Section \ref{seccomplexity} is devoted to the study of the complexity status of the parameters $h(\cdot)$, $g(\cdot)$ and their balanced analogs $\alphabal(\cdot),$ $\gbal(\cdot)$ and $\hbal(\cdot)$.
 In Section \ref{secsdp} we investigate  semidefinite bounds for $g(\cdot)$ and $h(\cdot)$ and, in Section \ref{seceig}, we study the corresponding eigenvalue-based bounds. In Section \ref{secexamples} we illustrate the behaviour of the various parameters on several classes of regular bipartite graphs. We turn our attention to bounds for the balanced parameters in Section \ref{sec:balanced} and conclude with several remarks and open questions in the final Section~\ref{secconcluding}. 
In Appendix \ref{appendixgroup} we briefly present  the application of the parameters $h(\cdot),\widehat h(\cdot), \alphabal(\cdot)$ to bounding product-free sets in finite groups and we group several technical proofs in Appendices \ref{appendix-Xx}, \ref{appendixghat} and \ref{appendixproofsym}.
 
\subsection*{Some notation and preliminaries.}
Throughout $\mathcal S^n$ denotes the set of real symmetric $n\times n$ matrices. Let $I_n,J_n\in\mathcal S^n$ denote, respectively, the identity matrix and the all-ones matrix (also denoted as $I$, $J$ when the dimension is clear from the context). Given integers $a,b\ge 1$ we also let $J_{a,b}$ denote the $a\times b$ all-ones matrix. 
Given a graph $G=(V=[n],E)$,  $\mathcal S_G$ denotes the set of matrices  $M\in \mathcal S^n$ that are {\em supported} by $G$, i.e., such that $M_{ij}=0$ for all $i,j\in V$ such that $\{i,j\}\not\in E$. 
For a matrix $X\in \mathcal S^n$, $\diag(X)=(X_{ii})_{i=1}^n\in\R^n$ denotes the vector of its diagonal entries and, for a vector $x\in \R^n$, $\Diag(x)\in \mathcal S^n$ is the diagonal matrix with the $x_i$'s as its diagonal entries. We use the symbol $e\in \R^n$ to denote the all-ones vector (whose dimension should be clear from the context).

For a real symmetric matrix $A\in \mathcal S^{|V|}$ we denote its eigenvalues as $\lambda_1(A)\ge \ldots \ge \lambda_{|V|}(A)$. We will often consider the case when $A$ has  a bipartite structure, of the form
\begin{equation}\label{eqM}
A=\left(\begin{matrix} 0 & M\cr M\T& 0\end{matrix}\right)\in \mathcal S^{|V|},
\end{equation}
where $V$ is partitioned as $V=V_1\cup V_2$ with $|V_1|=|V_2|=:n$ and $M\in \R^{|V_1|\times |V_2|}$. Then the eigenvalues of $A$  are $\pm \sqrt{\lambda_1(MM\T)},\ldots, \pm \sqrt{\lambda_n(MM\T)}$, thus arising from the singular values of~$M$.

For a subset $U\subseteq V$ we let $\chi^U\in\R^{|V|}$ denotes its characteristic vector, whose $i$th entry is 1 if $i\in U$ and 0 if $i\in V\setminus U$. For a matrix $M\in \mathcal S^{|V|}$,  $M[U]=(M_{ij})_{i,j\in U}$ denotes the principal submatrix of $M$ indexed by $U$.

\section{Complexity results.}\label{seccomplexity}

In this section we prove several complexity results.  
Recall that a {\em clique} in $G$ is a set of pairwise adjacent vertices and $\omega(G)$ denotes the maximum cardinality of a clique in $G$, so that $\omega(G)=\alpha(\olG)$. 
We consider the following problems. 

\begin{probl}[$\alpha$-BAL-BIP]\label{ALPHA-IS-BAL-BIP}
Given a bipartite graph $G$, decide whether $\alpha(G)=\alphabal(G)$, i.e., whether $G$ has a balanced maximum independent set.
\end{probl}
\begin{probl}[HALF-SIZE-CLIQUE-EDGE] \label{prob:halfcliqueedge}
Given a graph~$G=(V,E)$ with $|V|$ even and  $|E|=\tfrac{1}{4}|V|(|V|-2)$, decide whether $\omega(G)\ge  
 \tfrac{|V|}{2}$.
\end{probl}
\begin{probl}[HALF-SIZE-CLIQUE]
Given a graph $G=(V,E)$ with $|V|$ even, decide whether $\omega(G)\ge {|V|\over 2}.$
\end{probl}
\begin{probl}[CLIQUE]
Given a graph $G$ and an integer $k\in \N$, decide whether $\omega(G)\ge k$.
\end{probl}
It is well-known that CLIQUE is an NP-complete problem \cite{Karp72} as well as problem HALF-SIZE-CLIQUE;
we refer,  e.g., to \cite{Alonetal} for an easy reduction of CLIQUE to HALF-SIZE-CLIQUE.
In what follows we will show  the following reductions
\begin{align}\label{relation-reductions}
\text{HALF-SIZE-CLIQUE}  \leq_{P}\text{HALF-SIZE-CLIQUE-EDGE}  \leq_{P}\text{$\alpha$-BAL-BIP}.
\end{align}
Here we say that L$_1\leq_P$ L$_2$ if we have a polynomial-time algorithm permitting to encode an instance of  L$_1$ as an instance of L$_2$. We will show the first reduction in Theorem \ref{theorem-half-clique-edge} and the second one in Theorem \ref{theo-equiv-H} below.
Then, using the reductions in (\ref{relation-reductions}),  we obtain the following complexity results.

\begin{theorem}\label{theohardness}
Problem \ref{ALPHA-IS-BAL-BIP} ({\rm $\alpha$-BAL-BIP}) is an NP-complete problem.
\end{theorem}

\begin{corollary}\label{corhardness}
Computing the parameter $h(G)$ for $G$ bipartite is NP-hard.
\end{corollary}

\proof{Proof.}  
Recall that computing $\alpha(G)$ in bipartite graphs can be done in polynomial time. Hence, if there is a polynomial time algorithm for computing $h(G)$, then one can decide in polynomial time whether $h(G)=\frac{\alpha(G)}{4}$, which is equivalent to Problem~\ref{ALPHA-IS-BAL-BIP}, in view of Lemma \ref{lemrel1}. 
\endproof

\smallskip\noindent
The proof technique used to show the reduction from problem HALF-SIZE-CLIQUE-EDGE to problem 
$\alpha$-BAL-BIP  will in fact allow to show 
a  broader set of results. Namely it permits to show hardness of testing whether any of the following equalities holds: $g(G)=\gbal(G)$,  $h(G)=\hbal(G)$, or $h(G)={1\over 2}\sqrt{g(G)}$.
  In  other words, it is NP-hard  to  check whether any of the inequalities in relation (\ref{eqrel0}) holds at equality. See Corollary \ref{coro-complexity} below for these and other hardness results.

\medskip
In the rest of the section we will prove the two reductions from relation (\ref{relation-reductions}) and related hardness results for the other (balanced) parameters. 
For this we use as a first ingredient the following graph constructions. 
\begin{definition}
{\em Let $G=(V(G),E(G))$ and $H=(V(H), E(H))$ be two graphs with disjoint vertex sets and let $k\ge 1$ be an integer.
\begin{description}
\item[(i)] The {\em disjoint union} of $G$ and $H$, denoted by $G\oplus H$, is the graph with vertex set $V(G)\cup V(H)$ and edge set $E(G)\cup E(H)$.
\item[(ii)] The {\em join} of $G$ and $H$, denoted by $G\bowtie H$, is the graph with vertex set $V(G)\cup V(H)$ and edge set $E(G)\cup E(H)\cup (V(G)\times V(H))$. 
\item[(iii)] The {\em $k$-th expansion} of $G$, denoted by $G^{(k)}$, is the graph constructed as follows: 
its vertex set is $\bigcup_{v\in V(G)}X_v$, where $X_v$ are disjoint sets, each of size $k$, and we have a clique on each $X_v$ and a complete bipartite graph between $X_u$ and $X_v$ whenever $\{u,v\}\in E(G)$.
 \end{description}
 }
\end{definition}

Clearly we have the following relations 
\begin{align}
|V(G\oplus H)|= |V(G)|+|V(H)|,\ |E(G\oplus H)|=|E(G)|+|E(H)|,\label{eq0} \\
\omega(G\oplus H)=\max\{\omega(G),
 \omega(H)\},\\
|V(G\bowtie H)|=|V(G)|+|V(H)|,\ |E(G\bowtie H)|= |E(G)|+|E(H)|+|V(G)|\cdot |V(H)|,\\  \omega(G\bowtie H)= \omega(G)+\omega(H),  \label{eq1}\\
|V(G^{(k)})|=k|V(G)|,\  |E(G^{(k)})|=\tbinom{k}{2}|V(G)| + k^2|E(G)|,\  \omega(G^{(k)}) = k\omega(G). \label{eq3}
\end{align}

\definecolor{sqsqsq}{rgb}{0.12549019607843137,0.12549019607843137,0.12549019607843137}
\begin{figure}
\centering
\begin{tikzpicture}[line cap=round,line join=round,x=.4cm,y=.4cm]
\clip(-6.3877982447671355,-1.3900286587123167) rectangle (1.7947183186186364,3.0021770512561883);
\draw [line width=2pt] (-3.54,2.178409421321225)-- (-1.78,2.1784094213212244);
\draw [line width=2pt] (-1.78,2.1784094213212244)-- (-0.9,0.654204710660612);
\draw [line width=2pt] (-0.9,0.654204710660612)-- (-1.78,-0.87);
\draw [line width=2pt] (-1.78,-0.87)-- (-3.54,-0.87);
\draw [line width=2pt] (-3.54,-0.87)-- (-4.42,0.6542047106606134);
\draw [line width=2pt] (-4.42,0.6542047106606134)-- (-3.54,2.178409421321225);
\draw [line width=2pt] (-3.54,2.178409421321225)-- (-1.78,-0.87);
\draw [line width=2pt] (-3.54,-0.87)-- (-1.78,2.1784094213212244);
\draw [line width=2pt] (-4.42,0.6542047106606134)-- (-0.9,0.654204710660612);
\draw [line width=2pt] (-1.78,2.1784094213212244)-- (-1.78,-0.87);
\begin{scriptsize}
\draw [fill=sqsqsq] (-3.54,-0.87) circle (3.5pt);
\draw [fill=sqsqsq] (-1.78,-0.87) circle (3.5pt);
\draw [fill=sqsqsq] (-0.9,0.654204710660612) circle (3.5pt);
\draw [fill=sqsqsq] (-1.78,2.1784094213212244) circle (3.5pt);
\draw [fill=sqsqsq] (-3.54,2.178409421321225) circle (3.5pt);
\draw [fill=sqsqsq] (-4.42,0.6542047106606134) circle (3.5pt);
\end{scriptsize}
\end{tikzpicture}
\caption{Graph $F$, $\omega(F)=3$, 6 nodes, 10 edges.\label{graph-F}}\label{graphF}\vspace{-5pt}
\end{figure}

\vspace*{-3mm}
\begin{theorem}\label{theorem-half-clique-edge}
{\rm HALF-SIZE-CLIQUE} $\leq_P$ {\rm HALF-SIZE-CLIQUE-EDGE}.
\end{theorem}

\proof{Proof.} 
Let $G$ be an instance of HALF-SIZE-CLIQUE, set $|V(G)|=2n$,  $|E(G)|~=~m$. Let $t$ be the smallest integer such that $\binom{t}{2}\geq 9n^2+n+m$. Consider the graph $F$ from Fig. \ref{graphF} and define the graph $H:=((G\bowtie F^{(n)})\bowtie K_t) \oplus H_0$, where $H_0$ is a graph with $t$ nodes and $\binom{t}{2}- (9n^2+n+m)$ edges. So the role of $H_0$ is to add enough edges in order to ensure that $|E(H)|= |V(H)|(|V(H)|-2)/4$. Observe that $H$ can be constructed in polynomial time. 
Using     (\ref{eq0})-(\ref{eq3}), we obtain 
 \begin{align*}
 |V(H)|&=8n+2t,\\
 |E(H)|&=  (m+6\tbinom{n}{2}+10n^2+12n^2)+ \tbinom{t}{2} + 8nt + (\tbinom{t}{2}-9n^2-n-m)
 \\&=(4n+t)(4n+t-1)=\tfrac{1}{4}(8n+2t)(8n+2t-2),\\
 \omega(H)&=  \omega(G)+3n+t.
 \end{align*}
Hence, $H$ is an instance of HALF-SIZE-CLIQUE-EDGE and $\omega(H)\geq |V(H)|/2$ if and only if $\omega(G)\geq |V(G)|/2$. Therefore, if there is a polynomial time algorithm for solving HALF-SIZE-CLIQUE-EDGE, then we can solve HALF-SIZE-CLIQUE in polynomial time.
\endproof

\smallskip
As a next step we show the reduction of HALF-SIZE-CLIQUE-EDGE to 
 $\alpha$-BAL-BIP.
  Our proof is inspired from an argument in~\cite{CK03}, where   
the authors  consider minimum vertex covers in a bipartite graph restricted to have  at least $k_1$ vertices in one side of the bipartition and at least $k_2$ vertices in the other side.
In \cite[Theorem 3.1]{CK03} it is shown  that deciding existence of such vertex covers is NP-complete by giving a reduction from CLIQUE. We adapt this reduction by suitably selecting the values of $k_1$ and $k_2$, considering independent sets (complements of vertex covers) instead of vertex covers, and modifying the graph construction used in \cite{CK03}.

The following graph construction will play a central role for the reduction of HALF-SIZE-CLIQUE-EDGE to $\alpha$-BAL-BIP (and other related problems). 

\smallskip
\begin{definition}\label{defHG}
{\em Given a graph $G=(V,E)$ with $n:=|V|$ and $m:=|E|$,   consider the bipartite graph $H_G=(V_1\cup V_2, E_H)$ constructed as follows. 
\begin{description}
\item [(i)]
For each vertex $v\in V$ we construct two vertices $v_1\in V_1$ and $v_2\in V_2$ and add the edge $\{v_1,v_2\}$ to $E_H$.
\item [(ii)]
 For each edge $e\in E$ we construct two vertex sets $L_e\subseteq V_1$ and $R_e\subseteq V_2$ with $|L_e|=|R_e|=n+1$ and add  all edges in  $L_e\times R_e$ to $E_H$.
 \item [(iii)]
If $v\in V$ is incident to $e\in E$, then  we let $v_1$ be adjacent in $H_G$ to all vertices of $R_e$.
\end{description}
Hence, setting $L_V:=\{v_1: v\in V\}$, $R_V:=\{v_2: v\in V\}$, $L_E:=\bigcup_{e\in E} L_e$, and $R_E:=\bigcup_{e\in E}R_e$, we have 
$V_1=L_V\cup L_E$ and $V_2=R_V\cup R_E$, there is a perfect matching between $L_V$ and $R_V$, there is a complete bipartite graph between $L_e$ and $R_e$ for each $e\in E$, and there is a complete bipartite graph between $v_1\in V_1$ and $R_e$ for each edge $e\in E$ containing $v\in V$.
}
\end{definition}

\smallskip
The next lemma shows that the maximal independent sets in the bipartite graph $H_G$ have a very special structure, which will be  useful for the proof of Theorem \ref{theo-equiv-H} below.

\begin{lemma}\label{aux-lemma-H}
Let $G=(V,E)$ be a graph, $n:=|V|$, $m:=|E|$, and let $H_G$ be the associated bipartite graph as in Definition \ref{defHG}. Assume $I\subseteq V(H_G)=V_1\cup V_2$ is a   maximal independent  set of $H_G$. Then $I$   takes the following form 
\begin{align}\label{shape}
I\cap V_1= \{v_1: v\in A\} \cup  \bigcup_{e\in E_1}L_e,\quad I\cap V_2= \{v_2:v\in B\} \cup  \bigcup_{e\in E_2} R_e,
\end{align}
where $A\subseteq V$, $B=V\setminus A$, $E_1$ is the set of edges $e\in E$ that are incident to some node $v\in A$, and $E_2=E\setminus E_1$ (thus the set of edges $e\in E$ contained in $B$). Moreover, $I$ is  a maximum independent set of $H_G$ and $\alpha(H_G) = n+m(n+1)$. Conversely, any set $I$ as in (\ref{shape})  is a (maximum)  independent set of $H_G$.
\end{lemma}

\proof{Proof.} 
Assume $I\subseteq V_1\cup V_2$ is a maximal independent set of $H_G$. Set $A:=\{v\in V: v_1\in I\}$, $B:=\{v\in V: v_2\in I\}$, and $E_2:=E\setminus E_1$, where $E_1$ is the set of edges $e\in E$ that are incident to some node $v\in A$; we show that (\ref{shape}) holds.
First, we have $A\cap B=\emptyset$ (for, if $v\in A\cap B$, then the edge $\{v_1,v_2\}$ of $H_G$ would be contained in $I$, contradicting that $I$ is independent). Moreover, $A\cup B=V$ (for, if $v\in V\setminus (A\cup B)$, then the set $I\cup\{v_2\}$ would be independent in $H_G$, contradicting the maximality of $I$). 
So we have $I\cap L_V=\{v_1:v\in A\}$ and $I\cap R_V=\{v_2:v\in B\}$.
We now claim that $I\cap L_E=\bigcup_{e\in E_1}L_e$ and $I\cap R_E=\bigcup_{e\in E_1}R_e$.
First note that, if $I\cap R_e\ne \emptyset$, then $e$ is not incident to any node of $A$ and thus $e\in E_2$. Moreover, by maximality of $I$, we have $R_e\subseteq I$ for any $e\in E_2$. So we indeed have $I\cap R_E=\bigcup_{e\in E_2}R_e$ and in turn this implies $I\cap L_E=\bigcup_{e\in E_1}L_e$. Therefore we have $|I|=n+m(n+1)$, which implies that $\alpha(H_G)=n+m(n+1)$ and that $I$ is maximum independent. This concludes the proof (since the last (reverse) claim is straigthforward to check).
\endproof

\begin{corollary}\label{coro-H}
Let $G=(V,E)$ be a graph and let $H_G$ be the bipartite graph as in Definition~\ref{defHG}. The following assertions are equivalent.
\begin{description}
\item[(i)] $\alphabal(H_G)=\alpha(H_G)$.
\item[(ii)] $\gbal(H_G)=g(H_G)$.
\item[(iii)] $\hbal(H_G)=h(H_G)$.
\end{description}
\end{corollary}

\proof{Proof.} 
The implications (i) $\Longrightarrow$ (ii) and (i) $\Longrightarrow$  (iii) follow from relation (\ref{eqrel0}).
Conversely, assume (ii) holds and let $(A,B)$ be a balanced optimal solution for $g(H_G)$. Then $A\cup B$ is  maximal independent in $H_G$ and thus, by Lemma \ref{aux-lemma-H}, it is maximum, so that $\alpha(H_G)=|A\cup B|=\alphabal(H_G)$ as $(A,B)$ is balanced. The same argument shows the implication (iii) $\Longrightarrow$ (i).
\endproof

\smallskip
Now we show the main result of the section, which combined with Theorem \ref{theorem-half-clique-edge}, implies 
Theorem~\ref{theohardness}.

\begin{theorem}\label{theo-equiv-H}
Let $G=(V,E)$ be a graph satisfying $|E|=\frac{1}{4}|V|(|V|-2)$  and let $H_G$ be the associated bipartite graph as in Definition \ref{defHG}. The following assertions are equivalent.
\begin{description}
\item [(i)] $G$ has a clique of size $|V|/2$, i.e., $\omega(G)\ge |V|/2$.
\item [(ii)] $\alpha(H_G)=\alphabal(H_G)$.
\end{description}
Therefore,  {\rm HALF-SIZE-CLIQUE-EDGE} $\leq_P$ {\rm $\alpha$-BAL-BIP}.
\end{theorem}

\proof{Proof.} 
We first show (i) $\Longrightarrow$ (ii). Assume  $C$ is a clique of $G$ with $|C|=|V|/2$. Let $E_2$ be the set of edges of $G$ that are contained in $C$, so that $E_1:=E\setminus E_2$ is the set of edges of $G$ that are incident to some node in $V\setminus C$. 
By the assumption on $G$ we have $\binom{|V|/2}{2}= {|E|\over 2}$ and thus  $|E_2|= \binom{|V|/2}{2}= {|E|\over 2}= |E_1|$.  Consider the  subset $I\subseteq V_1\cup V_2$ of $V(H_G)$, which is defined   by  
$$I\cap V_1= \{v_1: v\notin C\} \cup \bigcup_{e\in E_1} L_e,\quad  I\cap V_2= \{v_2: v\in C\} \cup \bigcup_{e\in E_2} R_e.$$
By Lemma \ref{aux-lemma-H}, $I$ is a maximum independent set in $H_G$ and $\alpha(H_G)=n+m(n+1).$ Moreover, we have $|I\cap V_1|=|I\cap V_2|$, which shows that $\alphabal(H_G)=\alpha(H_G)$. 

Now we show (ii) $\Longrightarrow$ (i).  By the assumption (ii), $H_G$ has a balanced maximum independent  set  $I$. By Lemma \ref{aux-lemma-H},  $I$ takes the form as in (\ref{shape}). As $I$ is balanced we have $|I\cap V_1|=|I\cap V_2|$ and thus $||A|-|B||= (n+1)||E_2|-|E_1||$. If $|E_1|\ne |E_2|$ then the left hand side is at most $n$ while the right hand side is at least $n+1$. Therefore we have $|E_1|=|E_2|=|E|/2$ and $|A|=|B|=|V|/2$. 
Moreover, $|E_2|\le {|B|\choose 2}={|V|/2\choose 2}$ since $E_2$ consists of the edges that are contained in $B$. This gives $|E|=2|E_2| \le 2{|V|/2\choose 2}= |V|(|V|-2)/4$. We now use the assumption $|E|=|V|(|V|-2)/4$ on the number of edges of $G$, which implies that equality holds throughout and thus that $B$ is a clique in $G$ of size $|B|=|V|/2$,  showing (i).
\endproof

\begin{corollary}\label{coro-complexity}
Given a bipartite graph $G$ it is NP-hard to decide whether any of the following equalities holds.
\begin{description}
\item[(i)] $g(G)=\gbal(G)$.
\item[(ii)]$h(G)=\hbal(G)$.
\item[(iii)] $h(G)=\frac{1}{4}\alpha(G)$.
\item [(iv)] $\frac{1}{2}\sqrt{g(G)}=\frac{1}{4}\alpha(G)$.
\item[(v)] $h(G)={1\over 2}\sqrt{g(G)}$.
\end{description}
\end{corollary}
\proof{Proof.} 
We show that it is NP-hard to check any of the equalities (i)-(v) for the class of bipartite graphs that are of the form $H_G$ (as in Definition \ref{defHG}) for some graph $G$ with $|E|=\frac{1}{4}|V|(|V|-2)$. 
The key fact is that, for bipartite graphs of the form $H_G$,  any of the assertions (i)-(v) is equivalent to $\alpha(H_G)=\alphabal(H_G)$; this was shown in Corollary \ref{coro-H} for (i)-(ii) and in relation (\ref{eqrel2}) for (iii)-(iv), and one can easily verify that  (v) implies (i). Then the corollary follows using Theorems  \ref{theorem-half-clique-edge} and~\ref{theo-equiv-H} together with hardness of HALF-SIZE-CLIQUE.  
\endproof
\smallskip

\begin{remark}
{\em The hardness results in Corollary \ref{coro-complexity} hold in fact for a broader class of bipartite graph parameters. For this consider a bivariate function $f:\R^2_+\to \R$ that satisfies the condition
\begin{equation}\label{eqf}
f(a,b)\le {a+b\over 4}, \ \text{ and }\ f(a,b)={a+b\over 4}\Longleftrightarrow a=b, \quad \text{ for all } a,b\in \N
\end{equation}
and define the corresponding graph parameter 
$$f(G):= \max\{f(|A|,|B|): (A,B) \text{ is bipartite biindependent in } G\}\quad \text{ for } G \text{ bipartite}.$$
Using relation (\ref{eqf}) one can check the inequalities ${\alphabal(G)\over 4}\le f(G)\le {\alpha(G)\over 4}$ and the equivalence $f(G)={\alpha(G)\over 4} \Longleftrightarrow \alpha(G)=\alphabal(G).$
Using Theorem \ref{theo-equiv-H}, it follows that computing $f(\cdot)$ is NP-hard (already for the bipartite graphs of the form $H_G$ for some graph $G$ with ${|V|(|V|-2)}/4$ edges).

Examples of functions satisfying (\ref{eqf}) include $f(a,b)={ab\over a+b}$ (giving the parameter $h(G)$) and $f(a,b)={1\over 2}\sqrt {ab}$ (giving ${1\over 2}\sqrt{g(G)}$), or any $f(\cdot)$ nested between $h(\cdot)$ and ${1\over 2}\sqrt{g(\cdot)}$. As another example  consider $f(a,b):= \Big({1\over 2}\sqrt{ab}\Big)^p \Big({a+b\over 4}\Big)^{1-p}$ with $0\le p\le 1$, which gives a graph parameter $f(\cdot)$  nested between ${1\over 2}\sqrt{g(\cdot)}$ and $\alpha(\cdot)\over 4$.
}
\end{remark}

\section{Semidefinite approximations for the parameters \texorpdfstring{$g(G)$ and $h(G)$}{g(G) and h(G)}.}\label{secsdp}

In this section we introduce semidefinite approximations for the parameters $g(\cdot)$ and $h(\cdot)$ from (\ref{eqdefg}) and (\ref{eqdefh}), which are both NP-hard to compute  as we saw in the previous sections.
Let $G=(V=V_1\cup V_2,E)$ be a bipartite graph and let $C$ be the matrix from relation (\ref{eqC}) below. 
The starting point is to formulate the parameters $g(G)$ and $h(G)$  as maximizing, respectively,  the quadratic polynomial $x\T Cx$ and the rational function $x\T Cx \over x\T x$ over the  vectors $x\in \{0,1\}^{|V|}$ such that $x_ix_j=0$ for all $ \{i,j\}\in E$.
Then, to get a tractable approximation, a common approach   is  to linearize the quadratic terms by introducing a matrix $X$ modeling $xx\T$ in the case of $g(G)$, and modeling $xx\T \over x\T x$ in the case of $h(G)$. In this way one obtains the semidefinite bounds $g_1(G)$ and $h_1(G)$  introduced earlier in (\ref{eqg1}) and (\ref{eqh1}).  More generally, one can define hierarchies of semidefinite parameters  $(h_r(G))_{r\in\N}$ and~$(g_r(G))_{r\in\N}$ that upper bound~$h(G)$ and~$g(G)$, respectively,   using polynomial optimization techniques. Then  the parameters~$h_1(G)$ and~$g_1(G)$ correspond to the bounds at the first level~$r=1$  in these hierarchies. We will next briefly recall how the polynomial optimization approach applies for bounding the parameters $g(G)$ and $h(G)$ and after that  we investigate the bounds $g_1(G)$ and $h_1(G)$  in more detail.

\subsection{Polynomial optimization formulations and bounds.}\label{secpop}

We begin with a short recap on notation about polynomials and their use for approximating stable sets in graphs.
 Throughout, $\R[x]=\R[x_1,\ldots,x_n]$ denotes the ring of $n$-variate polynomials. For an integer $r\in \N$,  $\R[x]_r$ denotes the subset of $n$-variate polynomials with degree at most $r$. Then, $\Sigma_r\subseteq \R[x]_{2r}$ denotes the set of sums of squares of polynomials, of the form $\sum_{i=1}^ku_i^2$ with $u_i\in \R[x]_r$ and $k\in \N$. Recall that one can test whether a polynomial $f\in \R[x]_{2r}$ belongs to $\Sigma_r$ via semidefinite optimization. Indeed, $f\in \Sigma_{r}$ if and only if there exists a positive semidefinite matrix $Q$ that satisfies the polynomial identity 
$f(x)=[x]_r\T Q[x]_r$, where $[x]_r $ denotes the vector of square-free (aka multilinear) monomials of degree at most $r$. In particular $[x]_1$ denotes the (column) vector $(1, x_1,\ldots,x_n)\T$.

Let $G=(V=[n],E)$ be a graph. Define the ideal $I_G\subseteq \R[x]$ generated by the polynomials  $x_i^2-x_i$ ($i\in V$) and $x_ix_j$ ($\{i,j\}\in E$),  which consists of the polynomials  
$q=\sum_{i\in V}u_i(x_i^2-x_i)+\sum_{\{i,j\}\in E}u_{ij}x_ix_j$ with $u_i,u_{ij}\in \R[x]$.
For an integer $r\in\N$, 
 let $I_{G,2r}\subseteq \R[x]_{2r}$ denote its degree $2r$ truncation consisting of the above polynomials $q$, where we require that $u_i$ and $u_{ij}$ have degree at most $2r-2$.
The motivation for considering the ideal $I_G$ comes from the fact that  the stable sets in $G$ correspond to the vectors  in its variety $V(I_G)$, i.e., to the vectors $x\in\R^n$ satisfying $x_i^2-x_i=0$ for $i\in V$ and $x_ix_j=0$ for $\{i,j\}\in E$.
This enables reformulating the stability number of $G$ as
\begin{align}
\alpha(G)=\max\Big\{ \sum_{i\in V} x_i: x\in V(I_G)\Big\}& 
=\min\Big\{\lambda: \lambda -\sum_{i\in V}x_i\ge 0 \text{ for all } x\in V(I_G)\Big\} \label{eqalphapos}\\
&= \min\Big\{\lambda: \lambda -\sum_{i\in V}x_i\in \Sigma_{\alpha(G)} +I_{G,2\alpha(G)}\Big\}.\label{eqalphasos}
\end{align}
Here,  the last equality follows from the following well-known key fact:
for a  polynomial $p\in\R[x]$,  
\begin{equation}\label{eqsoslas}
p(x)\ge 0\text{ for all } x\in V(I_G) \Longleftrightarrow p\in \Sigma_{\alpha(G)}+I_{G}
\end{equation}
(see \cite{Lasserre01}, \cite{LaurentMOR2003}).
This motivates defining  the  parameters
\begin{align}\label{eqalphar}
\las_r(G):=\min\Big\{\lambda: \lambda-\sum_{i\in V}x_i\in \Sigma_r +I_{G,2r}\Big\}\quad \text{ for any } r\in \N,
\end{align}
also known as the Lasserre bounds for $\alpha(G)$. The parameter $\las_r(G)$ can be expressed via a semidefinite program and we have $\alpha(G)\le \las_{r+1}(G)\le \las_r(G)$, with equality $\alpha(G)=\las_r(G)$ if $r\ge \alpha(G)$ \cite{LaurentMOR2003}.  At order $r=1$ we obtain the bound $\las_1(G)$ which, after applying SDP duality, can be checked to take the form
\begin{equation}\label{eqthetaI} 
\las_1(G)= \max\Big\{\langle I, X\rangle: X\in \mathcal S^n,\ \left(\begin{matrix} 1 & \diag(X)\T\cr \diag(X)& X\end{matrix}\right)\succeq 0,\ 
X_{ij}=0 \text{ for } \{i,j\}\in E\Big\}.
\end{equation}
Another upper bound on $\alpha(G)$ is the theta number by Lov\'asz \cite{Lo79}, defined by
\begin{equation}\label{eqthetaJ}
\vartheta(G)=\max\{\langle J,X\rangle: X\in \mathcal S^n,\ X\succeq 0,\ \langle I,X\rangle =1,\ X_{ij}=0 \text{ for } \{i,j\}\in E\}.
\end{equation}
As is well-known these two bounds coincide: 
\begin{equation}\label{eqtheta=}
\las_1(G)=\vartheta(G)
\end{equation}
 (see, e.g., \cite{GLSbook}; see also Remark \ref{remtheta}).
Moreover, $\vartheta(G)=\alpha(G)$ if $G$ is bipartite (more generally, if $G$ is perfect, see \cite{GLSbook}).
 We now indicate how the polynomial optimization approach sketched above also applies to  the parameters $g(\cdot)$ and $h(\cdot)$.

\medskip
Assume now $G=(V=V_1\cup V_2,E)$ is a bipartite graph. 
Define the matrix
\begin{equation}\label{eqC}
C:={1\over 2}\left(\begin{matrix} 0 & J_{|V_1|,|V_2|} \cr J_{|V_2|,|V_1|} & 0\end{matrix}\right)\in \mathcal S^{|V|},
\end{equation}
so that $x\T Cx=\big(\sum_{i\in V_1}x_i\big)\big(\sum_{j\in V_2}x_j\big)$. As observed above one can encode a biindependent pair $(A,B)$ with $A\subseteq V_1$ and $B\subseteq V_2$ by its characteristic vector $x=\chi^{A\cup B}$, which belongs to the variety $V(I_G)$.
Then  we can express the parameters $g(G)$ and $h(G)$ as
\begin{align}
\gbbi(G)&=\max\Big\{x\T Cx: x_i^2=x_i \ (i\in V),\ x_ix_j=0 \ (\{i,j\}\in E)\Big\}, \label{eqgbbipop}\\
\hbbi(G)&=\max\Big\{{x\T Cx\over x\T x}: x_i^2=x_i \ (i\in V),\ x_ix_j=0 \ (\{i,j\}\in E)\Big\} \label{eqhbbipop}.
\end{align}
The  Lasserre bounds of order $r$ for $\gbbi(G)$ and $\hbbi(G)$ read, respectively,
\begin{align}
\gbbir(G)&:=\min\{\lambda: \lambda-x\T Cx \in \Sigma_r+I_{G,2r}\},
\label{eqgbbipopr}\\
\hbbir(G)&:=\min\{\lambda: x\T(\lambda I-C)x \in \Sigma_r+I_{G,2r}\},
\label{eqhbbipopr}
\end{align}
and the next result follows as a direct application of relation (\ref{eqsoslas}).

\begin{lemma}
Let $G$ be a bipartite graph. For any integer $r\ge 1$, we have 
$\gbbi(G)\le \gbbir(G)$ and $\hbbi(G)\le \hbbir(G)$, with equality if $r\ge \alpha(G)$.
\end{lemma}

Since sums of squares of polynomials can be modelled using positive semidefinite matrices  the parameters $\las_r(G)$, $g_r(G)$, $h_r(G)$ can be formulated using a semidefinite program. In later sections we will give the explicit semidefinite programs for the parameters $g_1(G)$ and $h_1(G)$, their symmetric versions and  their balanced analogs. An important  property that we will use is that strong duality holds for all these semidefinite programs, which follows from a result in \cite{JoszHenrion}
(thanks to the presence of the equations $x_i^2-x_i=0$ for $i\in V$ in the original polynomial optimization problems).

\subsection{Semidefinite formulations for the Lasserre bounds  \texorpdfstring{$\hbbione(G)$ and $\gbbione(G)$}{h1(G) and g1(G)}.}\label{secsdphg}

In this section we give explicit semidefinite formulations for  the Lasserre bounds (\ref{eqgbbipopr}) and (\ref{eqhbbipopr}) of order $r=1$ for $g(G)$ and $h(G)$. In particular, we indicate how to obtain the formulations  given earlier in (\ref{eqg1}) and (\ref{eqh1}). Recall that $\mathcal S_G$ consists of the matrices in $\mathcal S^{|V|}$ that are supported by $G$. 
We begin with a claim expressing polynomials in the truncated ideal $I_{G,2}$ that we will repeatedly use.

\begin{lemma}\label{lemMIG}
Given a graph $G=(V,E)$ and a matrix $M\in \mathcal S^{1+|V|}$ (indexed by $\{0\}\cup V$) we have $[x]_1\T M [x]_1 \in I_{G,2}$ if and only if $M$ takes the form
\begin{align}\label{eqMIG}
M=\left(\begin{matrix} 0 & -u\T /2\cr -u/2 & \Diag(u)+Z\end{matrix}\right)\ \text{ for some } u\in \R^{|V|},\ Z\in \mathcal S_G.
\end{align}
\end{lemma}

\proof{Proof.} 
By definition, $[x]_1\T M [x]_1\in I_{G,2}$ if  $[x]_1\T M [x]_1= \sum_{i\in V} u_i (x_i^2-x_i)+\sum_{\{i,j\}\in E}u_{ij}x_ix_j$ for some $u_i, u_{ij}\in \R$. The result follows by equating coefficients at both sides of this polynomial identity.
\endproof

We now give semidefinite formulations for the parameters $h_1(G)$ and $g_1(G)$.
 
 \begin{lemma} \label{lemhbbir1}
Let $G=(V=V_1\cup V_2,E)$ be a bipartite graph. Then  the Lasserre bound of order $r=1$ for $\hbbi(G)$ can be reformulated as 
 \begin{align}
 \hbbione(G) & = \min_{\lambda\in \R, Z\in \mathcal S^{|V|}} \{\lambda: \lambda I+Z-C\succeq 0, \  
 Z\in\mathcal S_G\},\label{eqhGV} \\
& =\max_{X\in \mathcal S^{|V|}}\{\langle C,X\rangle: X\succeq 0,  \ \text{\em Tr}(X)=1,\ X_{ij}=0 \text{ for } \{i,j\}\in E\}.\label{eqhGVd}
\end{align} \end{lemma}

\proof{Proof.} 
By definition, $h_1(G)$ is the smallest scalar $\lambda$ for which $x\T(\lambda I-C)x\in \Sigma_2+I_{G,2}$, i.e., the smallest $\lambda$ for which $[x]_1\T Q[x]_1 -x\T (\lambda I-C)x \in I_{G,2}$ for some matrix $Q\succeq 0$ (indexed by $\{0\}\cup V$).
Using Lemma \ref{lemMIG} we obtain that $Q_{00}=0$ and thus $Q_{0i}=0$ for all $i\in V$ (as $Q\succeq 0$). From this follows that the principal submatrix indexed by $V$ takes the form
$Q[V]= Z+\lambda I-C$ for some $Z\in \mathcal S_G$ and we arrive at the formulation (\ref{eqhGV}) for $h_1(G)$. By taking the semidefinite dual we obtain the formulation (\ref{eqhGVd}). As already noted above strong duality holds, as an application of \cite{JoszHenrion}.
\endproof

\begin{remark} \label{remSudaTanaka}
After submission of our paper, H.\ Tanaka attended us on references~\cite{ST2014} and~\cite{STT2017}, where the following bound on~$\sqrt{g(G)}$ is studied. Let~$I_{V_1}=\text{Diag}(\chi^{V_1})$ and~$I_{V_2}=\text{Diag}(\chi^{V_2})$, and define the parameter
\begin{align}
T(G):=\max_{X\in \mathcal S^{|V|}}\{ \langle C,X\rangle: X\succeq 0,  \ \langle X, I_{V_1}\rangle= \langle X, I_{V_2}\rangle=1,\ X_{ij}=0 \text{ for } \{i,j\}\in E,\ X \geq 0 \}.
\end{align}
Note that if  we let~$T'(G)$ denote the same parameter without the entrywise nonnegativity constraint on~$X$, then we have equality~$T'(G)=2h_1(G)$. It is clear that $T'(G)\le 2 h_1(G)$, since an optimal solution $X$ for~$T'(G)$ gives a feasible matrix $X/2$ for the program (\ref{eqhGVd}) defining $h_1(G)$.
 Conversely,  if $X$ is optimal for $h_1(G)$, then $2X$ is feasible for $T'(G)$. Indeed, one can show that such $X$ satisfies $\langle X,I_{V_1}\rangle = \langle X,I_{V_2} \rangle= 1/2$ (using  relations~\eqref{hyAyB2} and~\eqref{hyAyB3} in the proof below). So, we have the inequalities $\sqrt{g(G)}\le T(G)\le 2h_1(G)$.
\end{remark}

\begin{lemma}\label{lemgbbir1}
Let $G$ be a bipartite graph. Then we have
\begin{align}
\gbbione(G) &= \min_{\lambda\in \R, u\in \R^{|V|}, Z\in \mathcal S^{|V|}} \Big\{\lambda:
\left(\begin{matrix} \lambda & u\T /2\cr u/2 & \Diag(u)-C+Z\end{matrix}\right)\succeq 0,\ Z\in \mathcal S_G 
\Big\},\label{gbbione:primal0} \\
 & =\max_{ X \in \mathcal{S}^{|V|}} \Big\{ \langle C, X \rangle:   \,  
 \left(\begin{matrix} 1 & \diag(X)\T\cr \diag(X)& X\end{matrix}\right) \succeq 0, \, X_{ij} = 0 \text{ for  }\{i,j\} \in E  \Big\}.\label{gbbione:dual}
\end{align}
\end{lemma}

\proof{Proof.} 
By definition $g_1(G)$ is the smallest scalar $\lambda$ for which $\lambda-x\T Cx\in \Sigma_2+I_{G,2}$. In other words this is the smallest $\lambda$ for which there exists $Q\succeq 0$ such that 
$[x]_1\T \big( Q- {\tiny \left(\begin{matrix} \lambda & 0\cr 0 & -C\end{matrix}\right)}\big)[x]_1\in I_{G,2}$.
Using Lemma \ref{lemMIG} we obtain the formulation of $g_1(G)$ as in (\ref{gbbione:primal0}). 
Then  the formulation~$\eqref{gbbione:dual}$ follows  by taking the dual of the semidefinite program (\ref{gbbione:primal0}) and strong duality holds, by a result in \cite{JoszHenrion}.

\endproof

\begin{remark}\label{remlasone}
{\em In order to highlight some similarities and differences between the parameters $\las_1(G)$, $g_1(G)$ and $h_1(G)$ we indicate how to derive the formulation (\ref{eqthetaI}) of $\las_1(G)$. Let us start with the definition of $\las_1(G)$ as the smallest $\lambda$ for which $\lambda-\sum_{i\in V}x_i\in \Sigma_2+I_{G,2}$. Since $\sum_{i\in V}x_i-x\T I x\in I_{G,2}$ we can alternatively search for the smallest $\lambda$ for which $[x]_1\T \big(Q -{\tiny \left(\begin{matrix} \lambda & 0\cr 0 & -I\end{matrix}\right)}\big)[x]_1 \in I_{G,2}$. Using Lemma~\ref{lemMIG} we obtain 
\begin{align}\label{eqlas1dual}
\las_1(G)=\min_{\lambda\in\R, u\in \R^{|V|}, Z\in \mathcal S^{|V|}}\Big\{\lambda : 
 \left(\begin{matrix} \lambda & u\T/2 \cr u/2 & \Diag(u)-I+Z\end{matrix}\right)\succeq 0,\ Z\in \mathcal S_G\Big\}.
 \end{align}
 Taking the dual semidefinite program of (\ref{eqlas1dual}) we arrive at the formulation (\ref{eqthetaI}).
 
Note the similarity between programs (\ref{gbbione:primal0}) and (\ref{eqlas1dual}), which are the same up to exchanging the matrices $C$ and $I$. Note also that
it is possible to simplify program (\ref{eqlas1dual}) and to bring it in the  form 
 \begin{align}\label{eqlas1dualsimple}
 \las_1(G)=\min_{\lambda\in\R, Z\in \mathcal S^{|V|}} \Big\{\lambda: 
  \left(\begin{matrix} \lambda & e\T\cr e& I+Z\end{matrix}\right)\succeq 0,\ Z\in \mathcal S_G\Big\},
  \end{align}
which is another well-known formulation of $\vartheta(G)$. 
To see this, call $Q$ the matrix in program (\ref{eqlas1dual}). As $Q_{ii}=u_i-1\ge 0$ we have  $u_i\ge 1$ for all $i\in V$. By scaling the $i$th column/row of $Q$ by $2/u_i$ and adding $1-{4\over u_i^2}(u_i-1)={(u_i-2)^2\over u_i^2}\ge 0$ to entry $Q_{ii}$, we obtain a new  matrix $Q'\succeq 0$ satisfying $Q'_{0i}=Q'_{ii}=1$ for all $i\in V$, thus feasible for (\ref{eqlas1dualsimple}). This shows the equivalence of (\ref{eqlas1dual}) and (\ref{eqlas1dualsimple}).

Note, however,  that the above rescaling trick could not be applied to  program (\ref{gbbione:primal0}); indeed if $Q$ denotes the matrix appearing in (\ref{gbbione:primal0}), then one must have $Q_{ij}=-1/2$ for all positions $(i,j)\in V_1\times V_2$ corresponding  to non-edges of $G$. }
 \end{remark}

\smallskip
Finally, we mention a natural strengthening of $h_1(G)$, obtained by adding one row/column to the matrix variable (as in the definition (\ref{gbbione:dual}) of $g_1(G)$):
\begin{equation}\label{eqh1p}
h_1'(G):= \max\Big\{\langle C,X\rangle: \left(\begin{matrix}1 & x\T\cr x & X\end{matrix}\right)\succeq 0, \ \text{\rm Tr}(X)=1, \ x=\diag(X), \ X_{ij}=0 \text{ for }\{i,j\}\in E\Big\}.
\end{equation}
We have $$h(G)\le h_1'(G)\le h_1(G).$$ The inequality $h_1'(G)\le h_1(G)$ is clear since any feasible solution of (\ref{eqh1p}) gives a feasible solution of (\ref{eqhGVd}). To see that $h(G)\le h_1'(G)$, let $(A,B)$ be an optimal solution for $h(G)$  and set $y:=\chi^{A\cup B}$. Then $x:=y/e\T y$ and $X:=yy\T/e\T y$ provide a feasible solution for $h_1'(G)$,   with value $\langle C,X\rangle =|A|\cdot |B|/|A\cup B|= h(G)$ (using the fact that  $X-xx\T= yy\T (e\T y-1)/(e\T y)^2\succeq 0$).
In the next section we will show that $h_1(G)$ upper bounds also ${1\over 2}\sqrt{g(G)}$; the next example shows this is not true for  $h_1'(G)$.

\smallskip
\begin{example}
Let~$G=(V_1\cup V_2,E)$ be the bipartite graph  
with~$V_1=\{1,2\}$, $V_2=\{3,4\}$, and a single edge~$\{1,3\}$, see Figure~\ref{fig:smallgraphexample}.
\begin{figure}[ht]
\centering
\begin{tikzpicture}[thick,
  every node/.style={draw,circle},
  thenode/.style={fill=black, inner sep = 2pt},
  every fit/.style={ellipse,draw,inner sep=-2pt,text width=1.25cm,text height=1.5cm},
  -,shorten >= 1pt,shorten <= 1pt
]
\begin{scope}[start chain=going below,node distance=5mm]
  \node[thenode,on chain] (f1)  [label={[xshift=-0.3cm, yshift=-0.5cm]$1$}] {};
  \node[thenode,on chain] (f2)  [label={[xshift=-0.3cm, yshift=-0.5cm]$2$}] {};
\end{scope}

\begin{scope}[xshift=2.5cm,start chain=going below,node distance=5mm]
  \node[thenode,on chain] (s4) [label={[xshift=0.3cm, yshift=-0.5cm]$3$}] {};
  \node[thenode,on chain] (s5)  [label={[xshift=0.3cm, yshift=-0.5cm]$4$}] {};
\end{scope}

\node [fit=(f1) (f2),dotted,label={[xshift=0.0cm, yshift=-0.2cm]$V_1$}] {};
\node [fit=(s4) (s5),dotted,label={[xshift=0.0cm, yshift=-0.2cm]$V_2$}] {};

\draw[ultra thick] (f1) -- (s4);
\end{tikzpicture}
\caption{Graph $G$ with  $\alpha(G)=3$, $\alphabal(G)=2$, $h(G)=2/3$, and $g(G)=2$}\label{fig:smallgraphexample}
\end{figure}
We have $h_1'(G)< \big({\sqrt 2\over 2}=\big) {1\over 2}\sqrt{g(G)} = h_1(G)$.
Indeed,  
$h_1(G) \ge  {1 \over 2} \sqrt{g(G)}$ holds by  Proposition \ref{propghg1h1},
and~$h_1(G) \le {\sqrt 2 \over 2}$ follows from the fact that~${\sqrt 2 \over 2}I + A_G - C \succeq 0$, which exhibits a feasible solution to~\eqref{eqhGV}. Moreover, the strict inequality $h_1'(G)<\tfrac{\sqrt 2}{2}$ follows from the fact that the dual program~\eqref{eq:h1primedual} (defined below) of~\eqref{eqh1p} has feasible solution~$\lambda=0.0002$, $\eta = 0.7068$, $u=(-0.01,0.004,-0.01,0.004)\T$, $Z=0.99A_G$, with objective value~$0.707< \tfrac{\sqrt 2}{2}$.
\end{example}

\subsection{Comparison of the Lasserre bounds \texorpdfstring{$h_1(G)$ and $g_1(G)$}{h1(G) and g1(G)}.}

In this section we  show  the following inequalities 
$$h(G)\le  \tfrac{1}{2}\sqrt{g(G)} \leq  h_1(G) \leq \tfrac{1}{2} \sqrt{g_1(G)} \leq \tfrac{1}{4} \alpha(G)\quad \text{ for any bipartite graph } G,$$ 
that were claimed in Proposition~\ref{propghg1h1}.   
One may have strict inequalities  $h_1(G) < \tfrac{1}{2} \sqrt{g_1(G)} < \tfrac{1}{4}\alpha(G)$, e.g., when $G$ is  the complete bipartite graph~$K_{n,n}$ minus a perfect matching and  $n\geq 5$ (see Section~\ref{secKnnminM}).
To show the above inequalities we will use, in particular, the fact that the theta number $\vartheta(G)$ admits the 
two equivalent formulations that were given earlier in (\ref{eqthetaI}) and (\ref{eqthetaJ}) (recall (\ref{eqtheta=}), see also Remark \ref{remtheta}) and the fact that $\vartheta(G)=\alpha(G)$ when $G$ is a bipartite graph. 
Recall  that we already know $h(G) \leq {1\over 2}\sqrt{\gbbi(G)}$ from Lemma~\ref{lemrel1}. Hence, in order to show Proposition~\ref{propghg1h1}, it suffices to show the inequalities
${1\over 2}\sqrt{\gbbi(G)}\le \hbbione(G)$, $h_1(G)\le {1\over 2}\sqrt{g_1(G)}$, $h_1(G)\le {1\over 4}\alpha(G)$,  and $g_1(G)\le \alpha(G)h_1(G)$. 

\paragraph{Proof of  ${1\over 2}\sqrt{\gbbi(G)}\le \hbbione(G)$.} 
Let $(A,B)$ be an optimal solution for $\gbbi(G)$ with $|A|=:a, |B|=:b$ and let $(\lambda, Z)$ be a feasible solution for the formulation (\ref{eqhGV}) of $h_1(G)$; we show that $\lambda\ge {1\over 2}\sqrt{ab}$. By assumption, the matrix $M:=\lambda I +Z-C$ is positive semidefinite and thus also its principal submatrix $M[A\cup B]$ is positive semidefinite. Observe that $M[A\cup B]$ has the block-form
$$ M[A\cup B]= \left(\begin{matrix} \lambda I_a & -\frac{1}{2}J_{a,b} \cr -\frac{1}{2}J_{b,a} & \lambda I_b\end{matrix}\right),$$
because $Z_{ij}=0$ for $i\in A, j\in B$ as $A\cup B$ is independent. By taking a Schur complement we obtain that $M[A\cup B]\succeq 0$ if and only if $\lambda I_a - \frac{b}{4\lambda}J_{a,a} \succeq 0$.
This implies  $\lambda \geq \frac{1}{2}\sqrt{ab}={1\over 2}\sqrt{\gbbi(G)}$ and thus 
$\hbbione(G)\ge {1\over 2}\sqrt{\gbbi(G)}$.\hfill\qed

\paragraph{Proof of $h_1(G)\le {1\over 2}\sqrt{g_1(G)}$.}
Let $X$ be an optimal solution for the formulation (\ref{eqhGVd}) of $h_1(G)$. Then $X\succeq 0$ and thus $X=(y_i\T y_j)_{i,j\in V}$ for some vectors $y_i\in \R^{|V|}$ ($i\in V$). We may assume without loss of generality that $y_i\ne 0$ for $i\in V$ (since, if $y_i=0$, then we just replace $X$ by its principal submatrix indexed by $V\setminus \{i\}$).
Define the vectors $y':=\sum_{i\in V_1}y_i$ and $y'':=\sum_{i\in V_2}y_i$, so that $h_1(G)=\langle C,X\rangle =(y')\T y''$. To shorten notation we set $h:=h_1(G)=(y')\T y''$. We may assume $h>0$, else there is nothing to prove.
For $\epsilon=\pm 1$, define the vector $d_\epsilon:={y'+\epsilon y''\over \|y'+\epsilon y''\|}$.
Here the convention is that we consider the vector $d_\epsilon$ only if $y'+\epsilon y''\ne 0$. Note that at least one of $d_1$ and $d_{-1}$ is well-defined (since otherwise one would have $y'=y''=0$, implying $h_1(G)=0$, a contradiction).
Then  let $X_\epsilon$ denote the Gram matrix of the vectors  
${d_\epsilon\T y_i\over \|y_i\|^2}y_i$  for $i\in V$;  we  claim that  $X_\epsilon$ is feasible for the formulation (\ref{gbbione:dual}) of $g_1(G)$. 
To see it, consider the matrix $Y_\epsilon$ defined as the Gram matrix of the vectors $d_\epsilon$ and ${d_\epsilon\T y_i\over \|y_i\|^2}y_i$  for $i\in V$, so that $X_\epsilon$ is its principal submatrix indexed by $V$, and note that $Y_\epsilon\succeq 0$, $(Y_\epsilon)_{00}=1$, $(Y_\epsilon)_{0i}=(Y_\epsilon)_{ii}$ for $i\in V$, and $(Y_\epsilon)_{ij}=0$ if $\{i,j\}\in E$.
Hence, if one can show that $\langle C, X_\epsilon\rangle \ge 4\langle C,X\rangle^2$ for some $\epsilon\in \{\pm 1\}$, then this implies $g_1(G)\ge \langle C, X_\epsilon\rangle \ge 4\langle C,X\rangle^2=4h_1(G)^2$ and the proof is complete. The rest of the proof is devoted to  showing that $\langle C, X_\epsilon\rangle \ge 4\langle C,X\rangle^2$ for some $\epsilon\in \{\pm 1\}$, and is a bit technical. 

In a first step, we show that the vectors $y_i$ ($i\in V$) satisfy the following relations
\begin{align}
y_i\T y'' &= 2 h \norm{y_i}^2 \quad (i \in V_1),\label{hyAyB2}\\
y_j\T y' &= 2 h \norm{y_j}^2   \quad (j \in V_2).\label{hyAyB3}
\end{align}
For this consider an optimal solution $S:=hI+Z-C$ of the program (\ref{eqhGV}) defining $\hbbione(G)$, where  $Z\in\mathcal S_G$. As $X$ and $S$ are primal and dual optimal solutions we must have $XS=0$, i.e.,  
$0= hX+XZ-XC$. We now compute the diagonal entries. Note that $(XZ)_{ii}=0$ for all $i\in V$ (since, for each $k\in V$, we have $X_{ik}=0$ or $Z_{ki}=0$). Hence, for $i\in V_1$, we have
$h\|y_i\|^2=h X_{ii}= (XC)_{ii}= {1\over 2}\sum_{j\in V_2} X_{ij}= {1\over 2}y_i\T y''$, and, for $j\in V_2$, we have
$h\|y_j\|^2=hX_{jj}= (XC)_{jj}= {1\over 2}\sum_{i\in V_1}X_{ij}= {1\over 2}y_j\T y'.$
So (\ref{hyAyB2}) and (\ref{hyAyB3}) hold.

We now proceed to compute 
\begin{align}\label{eqpf1}
\langle  C,X_\epsilon\rangle & = 
\sum_{(i,j)\in V_1\times  V_2} {d_\epsilon\T y_i \cdot d_\epsilon\T y_j  \over \|y_i\|^2\|y_j\|^2}\cdot y_i\T y_j.
\end{align}
First,   we compute (part of) the inner term for $i\in V_1$ and $j\in V_2$:
\begin{align}
{d_\epsilon\T y_i \cdot d_\epsilon\T y_j  \over \|y_i\|^2\|y_j\|^2} &=
{1\over \|y'+\epsilon y''\|^2} {(y'+\epsilon y'')\T y_i \cdot (y'+\epsilon y'')\T y_j  \over \|y_i\|^2\|y_j\|^2}\\
& ={1\over \|y'+\epsilon y''\|^2}\Big( 2h {(y')\T y_i\over \|y_i\|^2} + 2h {(y'')\T y_j \over \|y_j\|^2} + \epsilon {(y')\T y_i \cdot (y'')\T y_j \over \|y_i\|^2\|y_j\|^2} + 4h^2\epsilon\Big),\label{eqpf2}
\end{align}
where we have used relations (\ref{hyAyB2}), (\ref{hyAyB3}) and that $\epsilon^2=1$  to carry out the simplifications.
Next  observe that
\begin{align}\label{eqpf3}
\sum_{(i,j)\in V_1\times V_2} {(y')\T y_i\over \|y_i\|^2}y_i\T y_j & = \sum_{i\in V_1}  {(y')\T y_i\over \|y_i\|^2}\Big(\sum_{j\in V_2}y_i\T y_j\Big)=  \sum_{i\in V_1}  {(y')\T y_i\over \|y_i\|^2} y_i\T y''= 2h\sum_{i\in V_1}(y')\T y_i=2h\|y'\|^2,
\end{align}
where we have used again relation (\ref{hyAyB2}).
In the same way we have
\begin{align}\label{eqpf4}
\sum_{(i,j)\in V_1\times V_2} {(y'')\T y_j\over \|y_j\|^2}y_i\T y_j = 2h\|y''\|^2.
\end{align}
Combining (\ref{eqpf1}), (\ref{eqpf2}), (\ref{eqpf3}) and (\ref{eqpf4}), we obtain
\begin{align*}
\langle  C,X_\epsilon\rangle & = 
{1\over \|y'+\epsilon y''\|^2} \Big( 4h^2(\|y'\|^2 + \|y''\|^2 + \epsilon (y')\T y'') +
 \epsilon \sum_{(i,j)\in V_1\times V_2} {(y')\T y_i \cdot (y'')\T y_j \cdot y_i\T  y_j\over \|y_i\|^2\|y_j\|^2}\Big)\\
& = {1\over \|y'+\epsilon y''\|^2}\Big( 4h^2 \|y'+\epsilon y''\|^2 -4h^2\epsilon (y')\T y'' +
 \epsilon \sum_{(i,j)\in V_1\times V_2} {(y')\T y_i \cdot (y'')\T y_j \cdot y_i\T  y_j\over \|y_i\|^2\|y_j\|^2}\Big)\\
 &= 4h^2 +{\epsilon \over \|y'+\epsilon y''\|^2} 
 \Big(\underbrace{\sum_{(i,j)\in V_1\times V_2} {(y')\T y_i \cdot (y'')\T y_j \cdot y_i\T  y_j\over \|y_i\|^2\|y_j\|^2}  -4h^3}_{=: \varphi}\Big) = 4h^2 +{\epsilon \cdot \varphi \over \|y'+\epsilon y''\|^2}.
  \end{align*}
We can now conclude the proof. Assume first $y'\pm y''\ne 0$, so that both $d_1$ and $d_{-1}$ are well-defined.
If $\varphi\ge 0$ then $\langle  C,X_1\rangle \ge 4h^2$. Otherwise, if $\varphi<0$, then $\langle  C,X_{-1}\rangle \ge 4h^2$. So we have shown the desired result: $\langle C,X_\epsilon\rangle\ge 4h^2$ for some $\epsilon\in \{\pm 1\}$.
Consider now the case when $y'=\epsilon  y''$ for some $\epsilon\in\{\pm 1\}$. 
 Then, using relations  (\ref{hyAyB2}) and (\ref{hyAyB3}), we obtain that 
$\varphi =0$. Hence, if $y'=y''$ (resp., $y'=-y''$), then we have $\langle C,X_{1}\rangle \ge 4h^2$ (resp., $\langle C,X_{-1}\rangle \ge 4h^2)$, which concludes the proof.  \hfill\qed

\begin{remark}\label{remtheta}
{\em Note that the proof for the inequality $h_1(G)\le {1\over 2}\sqrt{g_1(G)}$ resembles - but is  technically more involved than - the classical proof for the inequality
$\las_1(G)\ge \vartheta(G)$, where $\las_1(G)$ is given by (\ref{eqthetaI}) and $\vartheta(G)$ by (\ref{eqthetaJ}) and $G$ is an arbitrary graph. (The reverse inequality $\vartheta(G)\ge \las_1(G)$ is straightforward.) We sketch the  proof for $\las_1(G)\ge \vartheta(G)$ in order to highlight the resemblance with the proof above for  $ {1\over 2}\sqrt{g_1(G)}\ge h_1(G)$. So assume $X$ is optimal for (\ref{eqthetaJ}) (defined as the Gram matrix of vectors $y_i$ for $i\in V$) and construct the matrix $X_1$ (as the Gram matrix of the vectors ${d_1\T y_i\over \|y_i\|^2}y_i$ for $i\in V$, where $d_1:=(\sum_{i\in V}y_i)/\|\sum_{i\in V}y_i\|$). Then, $\vartheta(G)=\langle J,X\rangle=\|\sum_{i\in V}y_i\|^2$,  $1=\langle I,X\rangle=\sum_{i\in V}\|y_i\|^2$, and $y_i\T y_j=0$ if $\{i,j\}\in E$. This implies  $X_1$ is feasible for (\ref{eqthetaI}), and thus $\las_1(G)\ge \langle X_1,I\rangle$. It suffices now to check that $\langle X_1,I\rangle =\sum_{i\in V}{(d_1\T y_i)^2\over \|y_i\|^2} \ge \|\sum_{i\in V}y_i\|^2=\vartheta(G)$. But this follows easily using Cauchy-Schwartz inequality, namely
$$\|\sum_{i\in V}y_i\|^2 =(d_1\T\sum_{i\in V}y_i)^2=\big(\sum_{i\in V}{d_1\T y_i \over \|y_i\|}\|y_i\|\big)^2 \le \big(\sum_{i\in V}{(d_1\T y_i )^2\over \|y_i\|^2}\big) (\sum_{i\in V}\|y_i\|^2)
= \sum_{i\in V}{(d_1\T y_i )^2\over \|y_i\|^2}.$$
}
\end{remark}

\paragraph{Proof of $h_1(G)\le {1\over 4}\alpha(G)$.}
Let $X$ be optimal for the formulation  (\ref{eqhGVd}) of $h_1(G)$. Then $X$ is feasible for (\ref{eqthetaJ}) and thus $\vartheta(G)\ge \langle J,X\rangle$. Since $J-4C\succeq 0$ this implies $\langle J,X\rangle \ge 4\langle C,X\rangle= 4 h_1(G)$. Combining both inequalities we get $4h_1(G)\le \vartheta(G)=\alpha(G)$.\hfill\qed

\paragraph{Proof of $g_1(G)\le \alpha(G)h_1(G)$.} 
Let $X$ be an optimal solution for the formulation (\ref{gbbione:dual}) of $g_1(G)$. Then $X\over \Tr(X)$ is feasible for $h_1(G)$ and thus $g_1(G)=\langle C,X\rangle \le h_1(G) \cdot \Tr(X)$. 
On the other hand, $X$ is feasible for (\ref{eqthetaI}), which gives $\vartheta(G)\ge \Tr(X)$. Combining these two facts we obtain that 
$g_1(G)\le h_1(G)\cdot \vartheta(G)=h_1(G)\cdot \alpha(G)$.\hfill\qed

\begin{remark}
{\em So we have the following chain of inequalities for any bipartite graph $G$,
$${1\over 4}\alphabal(G)\le \hbbi(G)\le {1\over 2}\sqrt{\gbbi(G)}\le \hbbione(G)\le {1\over 4}\alpha(G)$$ (Proposition \ref{propghg1h1} and Lemma \ref{lemrel1}).
Hence equality $\alpha(G)=\alphabal(G)$ implies  $h_1(G)=h(G)$. Observe that the reverse implication holds when restricting to the bipartite graphs of the form $H_G$ (constructed from some graph $G$ as in Definition \ref{defHG}).  Indeed,  $h_1(H_G)=h(H_G)$ implies ${1\over 2}\sqrt{g(H_G)}=h(H_G)$, which in turn implies  $g(H_G)=\gbal(H_G)$ (Corollary \ref{coro-complexity} and its proof) and thus $\alpha(H_G)=\alphabal(H_G)$ (Corollary \ref{coro-H}).
This shows that deciding whether the parameter $h(\cdot)$ coincides with its  semidefinite relaxation $h_1(\cdot)$  
is an NP-hard problem (already when restricting to the bipartite graphs of the form $H_G$, recall Theorem \ref{theo-equiv-H}). This  can be seen as an analog of the hardness of deciding whether the basic semidefinite relaxation of the maximum cut problem is exact, as shown in  \cite{DelormePoljak}.
}
\end{remark}

\section{Eigenvalue bounds for the parameters \texorpdfstring{$g(G)$ and $h(G)$}{g(G) and h(G)}.}\label{seceig}

Let $G=(V,E)$ be a bipartite graph, with adjacency matrix $A_G$. We have introduced  in Lemmas~\ref{lemhbbir1} and \ref{lemgbbir1} the parameters  $g_1(G)$ and $h_1(G)$ that, respectively, upper bound the parameters $g(G)$ and $h(G)$. 
For convenience, we repeat their formulations 
\begin{align*}
g_1(G)& =  \min_{\lambda\in \R,  Z \in \mathcal S^{|V|}, u\in \R^{|V|}}\big\{\lambda: \lambda(\Diag(u)-C+Z) -\tfrac{1}{4}uu\T\succeq 0,\ \lambda \ge 0,\ Z\in \mathcal S_G\big\},\\  
h_1(G)&= \min_{\lambda\in \R, Z\in \mathcal S^{|V|}} \{\lambda: \lambda I+Z-C\succeq 0, \  Z\in \mathcal S_G\} 
\end{align*}
(where the formulation for  $g_1(G)$ follows from (\ref{gbbione:primal0}) after taking the Schur complement with respect to the upper left corner $\lambda$).
In order to obtain closed-form parameters 
one may restrict the optimization in each of the above programs to matrices  $Z=tA_G$ (for some $t\in \R$) and, for the parameter $g_1(G)$, to vectors $u=\mu e$ (for some $\mu\in \R$).  Let $\widehat g(G)$ and $\widehat h(G)$ denote the parameters obtained in this way, so that $g_1(G)\le \widehat g(G)$ and $h_1(G)\le \widehat h(G)$.
When the graph $G$ is  regular, the all-ones vector is an eigenvector of  the matrices involved in the programs defining $\widehat g(G)$ and $\widehat h(G)$, and, as we will show below,  this  allows to  show the closed-form expressions claimed in  Proposition~\ref{propeigbound} for   $\widehat g(G)$ and $\widehat h(G)$  in terms of the second largest eigenvalue $\lambda_2$ of $A_G$ and $n:=|V_1|=|V_2|$.

We will  use the following basic result about the eigenvalues of $A_G$. We refer, e.g., to the book by Brouwer and Haemers \cite{BH17} for general background about eigenvalues of graphs.

\begin{lemma}\label{lembipeig}
Assume $G=(V_1\cup V_2,E)$ is a bipartite $r$-regular graph with $|V_1|=|V_2|=:n\ge 2$. Then its adjacency matrix is of the form
\begin{equation}\label{eqAMG}
A_G = \begin{pmatrix} 
0 & M_G \\
M_G\T & 0
\end{pmatrix}, \quad \text{ where } M_G\in \R^{|V_1|\times |V_2|},
\end{equation}
the  eigenvalues of $A_G$ are $\pm \sqrt{\lambda_i(M_GM_G\T)}$ for $i\in [n]$,   $\lambda_1(A_G)=r$, $\lambda_{2n}(A_G)=-r$, and $\lambda_2(A_G)\ge 0$, with equality $\lambda_2(A_G)=0$ if and only if $G$ is complete bipartite.
In the case when $G=B(H)$ is the bipartite double of an $r$-regular  graph $H$, we have $M_G=A_{H}$, the eigenvalues of $A_{B(H)}$ are $\pm \lambda_i(A_H)$ for $i\in [n]$ and thus 
$\lambda_2(A_{B(H)}) = \max\{\lambda_2(A_H),- \lambda_n(A_H)\}$.
When $G=B_0(H)$ is the extended bipartite double of $H$, we have  $M_G=A_{H}+I$ and $\lambda_2(A_{B_0(H)})=\max\{\lambda_2(A_H)+1,-\lambda_n(A_H)-1\}$.
\end{lemma}

\subsection{An eigenvalue-based upper bound \texorpdfstring{$\widehat h(G)$ for $h(G)$}{hhat(G) for h(G)}.}

We give a closed-form eigenvalue-based upper bound  for the parameter $h(G)$ in the case when the bipartite graph $G$ is $r$-regular. 
Let $n:=|V_1|=|V_2|$ and let $\lambda_2$ denote the second largest eigenvalue of $A_G$ (i.e., the second largest singular value of $M_G$, by Lemma \ref{lembipeig}).
Vallentin \cite{Vallentin} shows that $\hbbi(G)\le \frac{n}{r} \lambda_2$, our next result gives a sharpening of this bound. 
\begin{proposition}\label{prop:hbbioneedgetransitive}
Assume~$G$ is a bipartite $r$-regular graph,   set~$|V_1|=|V_2|=:n$, and let $\lambda_2$ be the second largest eigenvalue of its adjacency matrix $A_G$.  Then we have 
\begin{align}\label{eq:hhateigenvaluebound} \hbbione(G)  \leq \widehat{h}(G)={n\over 2}  \frac{ \lambda_2}{r+\lambda_2} \leq \frac{n}{r} \lambda_2.
\end{align}
 Moreover, equality $\hbbione(G) ={n\over 2}  \frac{ \lambda_2}{r+\lambda_2} $ holds when $G$ is  edge-transitive. \end{proposition}

\proof{Proof.} 
We may assume $G$ is not complete bipartite (else $\lambda_2=0$ and $h(G)=h_1(G)=\widehat h(G)=0$).
The inequality ${n\over 2}{\lambda_2\over r+\lambda_2}\le {n\over 2}\lambda_2$ is clear; we now show $h_1(G)\le {n\over 2}{\lambda_2\over r+\lambda_2}$. For this we use the formulation of $\hbbione(G)$ from (\ref{eqhGV}), where we restrict the optimization to matrices $Z$ of the form $Z=tA_G$ for some scalar $t\in \R$; we will show that the resulting optimal value is equal to ${n\over 2}{\lambda_2\over r+\lambda_2}$.  Note that when $G$ is  edge-transitive  this restriction can be made without loss of generality.
Thus we aim to compute the optimum value of the program 
\begin{equation}\label{eqh1hat}
\widehat {h}(G):=\min_{\lambda, t\in \R} \{ \lambda: \lambda I + tA_G - C \succeq 0 \},
\end{equation}
which upper bounds $\hbbione(G)$ and is equal to it when $G$ is  edge-transitive.
 By taking a Schur complement, the matrix
$$
 \lambda I + tA_G - C  = \begin{pmatrix} \lambda I & tM_G-\tfrac{1}{2}J \\ tM_G\T -\tfrac{1}{2} J & \lambda I \end{pmatrix}
$$
is positive semidefinite if and only if $\lambda > 0$ and the matrix
\begin{align*}
    \lambda^2 I -(tM_G-\tfrac{1}{2}J)(tM_G\T-\tfrac{1}{2}J) 
    &=\lambda^2 I -(t^2M_G M_G\T - \tfrac{t}{2} M_G J - \tfrac{t}{2} JM_G\T +\tfrac{1}{4}J^2)
    \\&= \lambda^2 I -t^2M_G M_G\T +  \tfrac{rt}{2} J+  \tfrac{rt}{2} J -\tfrac{n}{4}J
     \\&= \lambda^2 I -t^2M_G M_G\T+ (rt - \tfrac{n}{4})J =:Q
\end{align*}
is positive semidefinite. Since $G$ is not complete bipartite we have $\lambda >0$. We now analyze when $Q$ is positive semidefinite. The all-ones vector $e$ is an eigenvector of $M_G M_G\T$ 
and $J$, and thus also of $Q$.   Any eigenvector $w\perp e$  of $M_G M_G\T$ for 
$\lambda_i(M_GM_G\T)$ ($2\le i\le n$) is an eigenvector of $Q$. Then the eigenvalues of $Q$ at these eigenvectors are as follows:
\begin{align*}
    \text{ at $e$: }\quad  &\lambda^2 -t^2r^2 +n(tr-\tfrac{n}{4}),\\
    \text{ at $w \perp e$: } \quad&\lambda^2-t^2\lambda_i(M_GM_G\T)\quad \text{ for } i= 2,\ldots,n.
\end{align*}
Hence, $Q\succeq 0$ if and only if $ \lambda^2 -t^2r^2 +n(tr-\tfrac{n}{4}) \ge 0$ and 
$ \lambda^2-t^2\lambda_i(M_GM_G\T)\ge 0$ for any $i\ge 2$, which is equivalent to $\lambda^2-t^2\lambda_2^2\ge 0$ (recall Lemma \ref{lembipeig}).
Therefore,  we must select $t$ such that
$$
\max \{ t^2 \lambda_2^2 , t^2 r^2 -ntr + \tfrac{n^2}{4}\} \text{ is smallest possible.}
$$
This maximum value is minimized at a root of the  quadratic function 
$    \phi(t) := (t^2r^2 - trn +  \tfrac{n^2}{4})-t^2 \lambda_2^2= t^2(r^2-\lambda_2^2) - trn + \tfrac{n^2}{4}.$ 
Its discriminant is $r^2 n^2 -n^2(r^2-\lambda_2^2) = n^2\lambda_2^2$ and $\phi(t)$ has two roots
$\frac{rn +\epsilon n\lambda_2}{2(r^2-\lambda_2^2)}=
\frac{n}{2(r-\epsilon\lambda_2)}$ for $\epsilon=\pm 1$.  
So
$\max \{ t^2 \lambda_2^2 , t^2 r^2 -ntr + \tfrac{n^2}{4}\}$  is minimized at the smallest root 
$t:={n\over 2(r+\lambda_2)}$. Therefore  we have $\widehat h(G) =t\lambda_2= \frac{n \lambda_2}{2(r+\lambda_2)}$,
which proves~(\ref{eq:hhateigenvaluebound}). \endproof

\smallskip
The parameter $h_1'(G)$ introduced in (\ref{eqh1p}) provides an upper bound for $h(G)$ that is at least as good as  $h_1(G)$.
A natural question is whether one can derive from it another closed-form bound for $h(G)$ that may improve on $\widehat h(G)$ when $G$ is  regular.
To define such a bound one  follows the same strategy as for $\widehat h(G)$. First, one writes the dual formulation of (\ref{eqh1p}), which reads 
\begin{align}\label{eq:h1primedual}
\min_{\lambda,\eta\in\R, u\in \R^n, Z\in \mathcal S^n}\Big\{\lambda+\eta: \left(\begin{matrix} \lambda & -u\T/2\cr -u\T/2 & \Diag(u)+\eta I+Z-C\end{matrix}\right)\succeq 0, \ Z\in \mathcal S_G\Big\}.
\end{align}
Then one restricts the optimization to $u=\mu e$ and  $Z=tA_G$ for scalars $\mu,t\in \R$ and, after that, one takes again the dual, which gives the parameter
\begin{equation}\label{eqhpsym}
\widehat h'(G):=\max\Big\{ \langle C,X\rangle: \left(\begin{matrix} 1 & x\T \cr x & X\end{matrix}\right)\succeq 0,\ \text{\rm Tr}(X)=1,\ e\T x = 1,\ \langle A_G,X\rangle =0\Big\}.
\end{equation}
To ease the comparison with $\widehat h(G)$, let us also write the dual program of (\ref{eqh1hat}), which reads
\begin{equation}\label{eqh1hatX}
\widehat h(G)= \max \{\langle C,X\rangle: X\succeq 0,\ \text{\rm Tr}(X)=1,\ \langle A_G,X\rangle =0\}.
\end{equation}
(Strong duality holds since  (\ref{eqh1hat}) is strictly feasible.)  
Both parameters $\widehat h(G)$ and $\widehat h'(G)$  in fact coincide. To show this we need the following auxiliary result, whose proof is postponed to Appendix~\ref{appendix-Xx}.

\begin{lemma}\label{lemXx}
Assume $X\in \mathcal S^n$ satisfies $X\succeq 0$ and $\text{\rm Tr}(X)=1$. Then   there exists a vector $x\in \R^n$ such that $\left(\begin{matrix} 1 & x\T\cr x& X\end{matrix}\right)\succeq 0$ and $e\T x=1$ if and only if $\langle J,X\rangle \ge 1$.
\end{lemma}

\begin{proposition} \label{prophhp}
For any bipartite regular graph $G$ we have $\widehat h(G)=\widehat h'(G)$.
\end{proposition}

\proof{Proof.} 
Comparing (\ref{eqhpsym}) with  (\ref{eqh1hatX}) it is clear that $\widehat h'(G)\le \widehat h(G)$.
If $G=K_{n,n}$, then both bounds are equal to 0. Assume $G\ne K_{n,n}$ and let $X$ be an optimal solution for (\ref{eqh1hatX}). 
As $J-4C\succeq 0$ we have $\langle J,X\rangle \ge 4 \langle C,X\rangle = 4\cdot  \widehat h(G)\ge 4 \cdot h(G)\ge 2$ (where $h(G)\ge 1/2$ follows by considering a biindependent pair $(\{a\},\{b\})$ with $a\in V_1$ and $b\in V_2$).
Hence we can apply  Lemma \ref{lemXx} and  find a vector $x$ such that $(x,X)$ is feasible for (\ref{eqhpsym}), which shows that   $\widehat h'(G)\ge \langle C,X\rangle =\widehat h(G)$.
\endproof

\subsection{An eigenvalue-based upper bound  \texorpdfstring{$\widehat g(G)$ for $g(G)$}{ghat(G) for g(G)}.}

In the same way one can give an eigenvalue-based upper bound $\widehat g(G)$ for the parameter $g(G)$ when $G$ is bipartite $r$-regular. It is obtained by solving analytically the following optimization problem
$$\widehat g(G):= \min_{\lambda, \mu,t\in\R} \Big\{\lambda: \lambda( \mu I -C + tA_G) - \tfrac{\mu^2}{4}J \succeq 0,\ \lambda \ge 0  \Big\}.$$
The details are analogous to those for the parameter $\widehat h(G)$ considered in the previous section, but  technically more involved. So we postpone the proof of the next result to Appendix \ref{appendixghat}.

\begin{proposition}\label{prop:gbbionesymmetric}
Assume~$G$ is a bipartite $r$-regular graph,  set~$n:=|V_1|=|V_2|$, and let $\lambda_2$ be the second largest eigenvalue of the adjacency matrix $A_G$ of $G$. Then we have 
$$
g_1(G)\le \widehat g(G)=\begin{cases}
 \frac{n^2\lambda_2^2}{(\lambda_2+r)^2}  &  \text{ if $r\le 3\lambda_2$},\\
 \frac{n^2\lambda_2}{8(r-\lambda_2)}  &  \text{ otherwise.}
\end{cases}
$$
Moreover, equality $g_1(G)=\widehat g(G)$ holds if $G$ is vertex- and edge-transitive.
\end{proposition}

\begin{remark}
{\em Here are examples of regular bipartite graphs satisfying $r\le 3\lambda_2$, or the reverse inequality $3\lambda_2\le r$: If $G$ is a perfect matching on $2n$ vertices, then $\lambda_2=r=1$ and thus $r< 3\lambda_2$ (see Section~\ref{secperfectM}); on the other hand, if $G$ is the complete bipartite graph $K_{n,n}$ minus a perfect matching, then $r=n-1$ and $\lambda_2=1$ and thus $r\ge 3\lambda_2$ if $n\ge 4$ (see Section~\ref{secKnnminM}).

Recall the inequalities $h(G)\le {1\over 2}\sqrt{g(G)}$ (from Lemma \ref{lemrel1}) and $h_1(G)\le {1\over 2}\sqrt{g_1(G)}$ (from Proposition~\ref{propghg1h1}). One can  check 
that also the eigenvalue bounds  satisfy the analogous relation  $\widehat h(G)\le {1\over 2}\sqrt{\widehat g(G)},$  with equality if and only if $r\le 3\lambda_2$. Hence, in the regime $3\lambda_2<r$,  the parameter $\widehat h(G)$ provides a strictly better bound than ${1\over 2}\sqrt{\widehat g(G)}$ for both $h(G)$ and ${1\over 2}\sqrt{g(G)}$.

So we have $h_1(G)\le \min\{\widehat h(G),{1\over 2}\sqrt{g_1(G)}\}
\le \max \{\widehat h(G),{1\over 2}\sqrt{g_1(G)}\}\le {1\over 2}\sqrt{\widehat g(G)}$. We now observe that  the two parameters $\widehat h(G)$ and ${1\over 2}\sqrt{g_1(G)}$ are  incomparable.
Indeed, as observed above, strict inequality $\widehat h(G)<{1\over 2}\sqrt{g_1(G)}$ may hold (e.g., for $K_{n,n}$ minus a perfect matching). 
On the other hand, there are regular bipartite graphs satisfying ${1\over 2}\sqrt{g_1(G)}<\widehat h(G)$ (such $G$ is not edge-transitive). As an example, let $G$ be the disjoint union of $C_4$ and $C_6$, thus $2$-regular with $\lambda_2=2$. Then, we verified that  ${1\over 2}\sqrt{g_1(G)}={1\over 2}\sqrt 6<{5\over 4}=\widehat h(G)$.
}
  \end{remark}

\subsection{Links to some other  eigenvalue bounds.}
\label{secHaemers}

In this section we investigate links between the new bounds introduced in previous sections and some known eigenvalue bounds in the literature. First we point out a natural link between $\hat h(\cdot)$ and Hoffman's ratio bound  (\ref{eqHoffman}) for the stability number of a graph.
After that, we present links to some spectral parameters  $\varphi(G)$,  $\varphi'(G)$ and $\varphi_H(G)$ by Haemers \cite{Haemers1997,Haemers2001}, 
which he used to bound the parameter $\gbc(G)$,  the maximum number of edges in a  biclique of an arbitrary graph $G$; see (\ref{eqHG1}), (\ref{eqHG3}) and (\ref{eqHG5}) below for the exact definitions.
As $\gbc(G)=\gbi(\olG)=g(B_0(\olG))$, also  the parameter $h_1(B_0(\olG))$ provides an upper bound for $\gbc(G)$. We will review the parameters of Haemers and investigate their relationships  with the parameters $h_1(\cdot)$ and $\widehat h(\cdot)$.

\subsubsection{Linking the parameter \texorpdfstring{$\widehat h(B(G))$}{hhat(B(G))}  to Hoffman's bound for \texorpdfstring{$\alpha(G)$}{alpha(G)}.}\label{secHoffman}

Let $G=(V=[n],E)$ be  an arbitrary graph and let $\lambda_n(A_G)$ be the smallest eigenvalue of its adjacency matrix. If $G$ is $r$-regular, then the following bound holds for its stability number:
\begin{equation}\label{eqHoffman}
\alpha(G)\le n{- \lambda_n(A_G)\over r-\lambda_n(A_G)}.
\end{equation}
This bound was proved by Hoffman (unpublished) and is known as Hoffman's ratio bound (see  Haemers \cite{Haemers2021} for a short proof and a historical account). 
There is  a tight link between Hoffman's ratio bound for $G$ and the parameter $\widehat h(\cdot)$ for its  bipartite double $B(G)$. 
Indeed, if $A\subseteq V$ is an independent set in $G$, then the pair $(A,A)$ is a balanced biindependent pair in $B(G)$. So $|A|\le \alpha(G)$ and $2|A|\le \alphabal(B(G)) \le 4\cdot \widehat h(B(G))$, giving
\begin{equation}\label{eqHoffmananalog}
\alpha(G)\le {1\over 2}\alphabal(B(G)) \le 2\cdot \widehat h(B(G))= n {\lambda_2(A_{B(G)})\over r+\lambda_2(A_{B(G)})}.
\end{equation}
By Lemma \ref{lembipeig}, we have  
 $\lambda_2(A_{B(G)})=\max\{\lambda_2(A_G), -\lambda_n(A_G)\}$, and thus 
$$n{- \lambda_n(A_G)\over r-\lambda_n(A_G)} \le  2\cdot \widehat h(B(G))= n {\lambda_2(A_{B(G)})\over r+\lambda_2(A_{B(G)})}.$$ 
 Lov\'asz \cite{Lo79} showed that also $\vartheta(G)$ is upper bounded by Hoffman's ratio bound.   The parameters $\vartheta(G)$ and $h_1(B(G))$ satisfy the analogous relationship: 
$\vartheta(G) \le  2\cdot h_1(B(G))$. Indeed, if $X$ is an optimal solution to program (\ref{eqthetaJ}), then the matrix
$X':={1\over 2}{\tiny \left(\begin{matrix} X& X\cr X& X\end{matrix}\right)}$ is feasible for (\ref{eqhGVd}) with objective value $\langle C,X'\rangle = {1\over 2}\langle J,X\rangle ={1\over 2}\vartheta(G)$, giving the desired inequality.
 
\subsubsection{Linking the parameter \texorpdfstring{$\hbbione(B_0(G))$}{h1(B0(G))} to Haemers' spectral bound \texorpdfstring{$\varphi(G)$}{phi(G)}.}
\label{sec:haemersdetails}

As we saw earlier, for any bipartite graph $G$, the parameter $h_1(G)$ provides an upper bound for the parameter $g(G)$, via ${1\over 2}\sqrt{g(G)} \le h_1(G)$.  
This also directly gives a bound for the parameter $\gbi(G)=g(B_0(G))$ when $G$ is an arbitrary graph, namely
${1\over 2} \sqrt{\gbi(G)}\le h_1(B_0(G))$. 

For an arbitrary graph $G=(V,E)$, Haemers \cite{Haemers2001} introduces the  spectral parameter \begin{equation}\label{eqHG1}
\varphi(G):= \min_{M\in \mathcal S^{|V|}}\{\lambda_{abs}(M): 
M_{ij}=1 \text{ for all } \{i,j\}\in E\},
\end{equation}
where $\lambda_{abs}(M)$ 
denotes the maximum absolute value of an eigenvalue of $M$, and he  shows that $\varphi(G)$ provides an upper bound for the parameter $\gbc(G)=\gbi(\olG)$ via the inequality 
\begin{equation}\label{eqHG2}
\sqrt{\gbc(G)}\le \varphi(G).
\end{equation}

So we have two bounds for $\gbc(G)$, namely  ${1\over 2}\sqrt{\gbc(G)}\le {1\over 2}\varphi(G)$ and ${1\over 2}\sqrt{\gbc(G)}\le h_1(B_0(\olG))$. We now show that these two upper bounds in fact coincide.

\begin{lemma}\label{lemhbivarphi}
For any graph $G$, we have $ \hbbione(B_0(G))={1\over 2}\varphi(\olG)$.
\end{lemma}

\proof{Proof.} 
Let  $G=(V,E)$ and $\olG=(V,\olE)$. First observe  the parameter $\varphi(\olG)$ can be reformulated as
\begin{equation}\label{lemphi}
\varphi(\olG)= \min\Big\{\lambda_{\max}(Y): Y=\left(\begin{matrix}0 & M\cr M &0\end{matrix}\right), \ M\in \mathcal S^{|V|},\ M_{ij}=1 \ \text{ for all } \{i,j\}\in \olE\Big\};
\end{equation}
this follows  from the fact that the eigenvalues of any  $Y$ in (\ref{lemphi}) are $\pm \lambda_i(M)$ for $i\in [|V|]$.
Let $V\cup V'$ be the vertex set of  the extended bipartite double  $B_0(G)$, where $V'$ is a disjoint copy of $V$, and let $C$ be the matrix from (\ref{eqC}), which is now indexed by $V\cup V'$.
We use the formulation (\ref{eqhGV}) of $\hbbione(B_0(G))$,  defined as  the smallest scalar $\lambda$ for which $\lambda I-C+Z\succeq 0$ for some  $Z\in\mathcal S_{B_0(G)}$ or, equivalently, as the minimum value of $\lambda_{\max}(C-Z)$ for $Z\in \mathcal S_{B_0(G)}$.
Since  the condition $Z\in \mathcal S_{B_0(G)}$ corresponds to $Y:=2(C-Z)$ being feasible for (\ref{lemphi}), we can conclude that $2\hbbione(B_0(G))=\varphi(\olG)$.

\endproof

\subsubsection{Linking  \texorpdfstring{$\hbbione(B_0(G))$}{h1(B0(G))} to Haemers'  spectral bounds \texorpdfstring{$\varphi'(G)$}{phi'(G)} and \texorpdfstring{$\varphi_H(G)$}{phiH(G)}.} \label{sec:haemerseigenvalues}

In the previous section we mentioned the spectral bound $\varphi(G)$ from (\ref{eqHG1}) of Haemers \cite{Haemers2001}  for the parameter $\gbc(G)$ and observed its link to the parameter $h_1(\cdot)$, recall (\ref{eqHG2}) and Lemma~\ref{lemhbivarphi}.
In some earlier work \cite{Haemers1997}, Haemers  introduced  the following spectral parameter for an arbitrary graph $G=(V=[n],E)$,
\begin{align}\label{eqHG3}
\varphi'(G):= \min_{M\in \mathcal S^{|V|}} \left\{ n \frac{\lambda(M)}{1+\lambda(M)}\,: Me=e, \ M_{ij}=0 \, \text{ for } \{i,j\} \in E\right\},
\end{align}
where~$\lambda(M)$ denotes the second largest absolute value of an eigenvalue of $M$. 
Haemers~\cite{Haemers2001} showed that~$\varphi(G) \leq \varphi'(G)$ for all~$G$ and that there are graphs~$G$ for which the inequality is strict. 

Let $L_G$ denote the Laplacian matrix of $G$ that is defined as $L_G=D_G-A_G$, where $D_G\in \mathcal S^n$ is the diagonal matrix whose $i$-th entry is the degree of vertex $i\in V$ in $G$. 
In what follows we let $0=\mu_1\le \mu_2\le \ldots\le \mu_n$ denote  the eigenvalues of the Laplacian matrix $L_G$. 
In \cite[Theorem~2.4]{Haemers1997} Haemers shows the inequality 
\begin{align}\label{eqHG4}
 \varphi'(\overline{G})\leq  \varphi_H(G):= {n\over 2}\Big(1-{\mu_2\over \mu_n}\Big)
 \end{align}
for any graph $G$ (on $n$ nodes), and he shows that equality holds in (\ref{eqHG4}) if $G$ is vertex- and edge-transitive. So we have the following inequalities 
\begin{align}\label{haemers:chain}
(h_1(B_0(G)) =)\ \tfrac{1}{2}\varphi(\overline{G})\leq \tfrac{1}{2}\varphi'(\overline{G})\leq {1\over 2}\varphi_H(G)= {n\over 4}\Big(1-{\mu_2\over \mu_n}\Big),
\end{align} 
where the right most  inequality is an equality if~$G$ is vertex- and edge-transitive. We next sharpen this latter result and show that $h_1(B_0(G))=  {n\over 4}\Big(1-{\mu_2\over \mu_n}\Big)$  if $G$ is vertex- and edge-transitive.

\begin{proposition}\label{prop:hbionesymmetric}
Let~$G=(V,E)$ be a graph,  set~$n:=|V|$, and let  $0=\mu_1\le \mu_2\le \ldots\le \mu_n$ denote the eigenvalues of the Laplacian matrix of $G$.  Then we have 
$$\hbbione(B_0(G))={1\over 2}\varphi(\olG) \leq {1\over 2} \varphi_H(G)={n\over 4}\Big(1-{\mu_2\over \mu_n}\Big),$$
with equality if $G$ is vertex- and edge-transitive. 
\end{proposition}

\proof{Proof.} 
Consider the parameter $\widetilde {h}(G)$ obtained from the definition of $\hbbione(B_0(G))$ in (\ref{eqhGV}), where   we restrict the optimization  to matrices $Z$ of the form $Z=\tiny{\begin{pmatrix} 0 & tL_G+\mu I\\ tL_G+\mu I& 0 \end{pmatrix}}$ for  scalars $t,\mu\in \R$.  
Hence, $\hbbione(B_0(G))\le \widetilde {h}(G)$. First, we show that if $G$ is vertex- and edge-transitive (hence regular), then  this restriction can be made without loss of generality and thus $\hbbione(B_0(G))= \widetilde {h}(G).$ 

For this, for any permutation  $\sigma$  of $V$ consider the associated permutation $\widetilde \sigma$ of $V\cup V'$ (the vertex set of $B_0(G)$, where $V'$ is a disjoint copy of $V$) defined by $\widetilde \sigma(i)=\sigma(i)$ and $\widetilde\sigma(i'):=\sigma(i)'$ for $i\in V$; clearly,  $\widetilde \sigma$ is an automorphism of $B_0(G)$ if $\sigma$ is an automorphism of $G$. Consider in addition the automorphism $\pi$ of $B_0(G)$ obtained by flipping $V$ and $V'$: $\pi(i)=i'$ and $\pi(i')=i$ for $i\in V$. Then, under the action of the group of automorphisms of $B_0(G)$ generated by $\pi$ and $\widetilde \sigma$ (for $\sigma$ automorphism of $G$), the edge set of $B_0(G)$  is partitioned into two orbits, the orbit $\Omega_V:=\{\{i,i'\}: i\in V\}$ and the orbit $\Omega_E:=\{\{i,j'\},\{i',j\}:\{i,j\}\in E\}$. Now, if $(\lambda,Z)$ is feasible for $h_1(B_0(G))$, then the same holds for its symmetrization obtained by averaging over the group of automorphisms of $B_0(G)$ just described. This gives a new feasible solution $(\lambda, Z)$, where the entries of $Z$   take two possible nonzero values, depending whether the entry corresponds to an edge  in $\Omega_V$ or in $\Omega_E$, and thus $Z$ has indeed the desired form claimed above.

We now aim to compute the optimum value of the program 
$$\widetilde {h}(G)=\min_{\lambda, t,\mu\in \R} \Big\{ \lambda \, : \,\,\begin{pmatrix} \lambda I & tL_G+\mu I - \tfrac{1}{2}J\\ tL_G+\mu I-\tfrac{1}{2}J& \lambda I \end{pmatrix} \succeq 0  \Big\}$$
and to show it is equal to ${n\over 4}\Big(1-{\mu_2\over \mu_n}\Big)$.
By taking a Schur complement (and assuming~$\lambda >0$) the matrix in the above semidefinite program
is positive semidefinite if and only if the matrix
\begin{align*}
    \lambda^2 I -(tL_G+\mu I -\tfrac{1}{2}J)(tL_G+ \mu I-\tfrac{1}{2}J) 
       &= (\lambda^2 -\mu^2)I-t^2L_G^2-2t\mu L_G+(\mu -\tfrac{n}{4})J =: Q
\end{align*}
is positive semidefinite. Let~$e$ denote the all-ones vector, which is an eigenvector of $L_G$ for its smallest eigenvalue $\mu_1=0$,  and let $w_i\perp e$ be an eigenvector of $L_G$ for its eigenvalue $\mu_i$ with $i\ge 2$. Then the eigenvalues of $Q$ at these eigenvectors are as follows:
\begin{align*}
    \text{ at $e$: }\quad  &\lambda^2-\mu^2 +n(\mu-\tfrac{n}{4})= \lambda^2- (\mu-\tfrac{n}{2} )^2,\\
    \text{ at $w_i \perp e$: } \quad&\lambda^2-(t \mu_i +\mu)^2, \quad \text{for $i=2,\ldots,n$.}
\end{align*}
Hence $Q\succeq 0$ if and only if all these eigenvalues are nonnegative and thus  we must select $t,\mu$ such that
$$
\max \Big\{  
( \mu-\tfrac{n}{2})^2, (t \mu_2 +\mu)^2,(t \mu_n +\mu)^2\Big\} \text{ is smallest possible.}
$$
 So we must find the smallest value of $\lambda$ for which there exist~$t,\mu$ satisfying the system
\begin{align*}
\lambda &\geq |t \mu_2 + \mu|,\quad \lambda \geq |t \mu_n + \mu|, \quad \lambda \geq | \mu -\tfrac{n}{2}|.
\end{align*}
First, note that taking~$\mu:= \tfrac{n}{4}+\tfrac{n\mu_2}{4\mu_n}$, $t:={-n \over 2 \mu_n}$ and~$\lambda:={n \over 4}(1-{\mu_2 \over \mu_n})$ is feasible for the above system (since $t\mu_2+\mu=\lambda$, $t\mu_n+\mu= \mu-{n\over 2}=-\lambda$), which shows $\widetilde h(G)\le {n \over 4}(1-{\mu_2 \over \mu_n})$. We now show the reverse inequality. Assume $\lambda,t,\mu$ satisfy the above system.
The conditions~$\lambda \geq -t \mu_n - \mu$ and~$\lambda \geq t\mu_2 + \mu$ together give~$\lambda \geq \tfrac{1}{2}(\mu_2-\mu_n)t$,
and  the  conditions~$\lambda \geq t\mu_2+\mu$ and~$\lambda \geq -\mu+\tfrac{n}{2}$ give~$\lambda \geq \tfrac{\mu_2}{2} t + \tfrac{n}{4}$.
Therefore, $\widetilde h(G)$ is at least the smallest value of $\lambda$ for which there exists $t$ such that  ~$\lambda \ge \max \{\tfrac{1}{2}(\mu_2-\mu_n)t , \tfrac{\mu_2}{2} t + \tfrac{n}{4}\}$. Now observe that  this maximum  is minimized at the intersection point, where~$t=-\tfrac{n}{2\mu_n}$ (since $\mu_2-\mu_n\le 0$ and $\mu_2\ge 0$).
This gives the desired relation
$\widetilde h(G)\ge  \tfrac{1}{2}(\mu_2-\mu_n)\big({n\over -2\mu_n}\big)= {n \over 4}(1-{\mu_2 \over \mu_n})$, which concludes the proof. 
\endproof

\smallskip
An interesting feature of the closed-form  bound ${1\over 2} \varphi_H(G)={n\over 4}\Big(1-{\mu_2\over \mu_n}\Big)$ in Proposition~\ref{prop:hbionesymmetric} is that it is valid without any regularity assumption on the graph $G$.

\medskip Assume now  $G$ is $r$-regular,  still arbitrary (not necessarily bipartite) on $n$ nodes. Then its adjacency matrix $A_G$ satisfies   $A_G=rI-L_G$ and thus its eigenvalues are $\lambda_i=r-\mu_{i}$ for $i\in [n]$, with $\lambda_1=r\ge \lambda_2\ge \ldots\ge \lambda_n$. Therefore, for any $r$-regular graph $G$, we have
\begin{equation}\label{eqHG5}
h_1(B_0(G))\le {1\over 2}  \varphi_H(G)={n\over 4}\Big(1-{\mu_2\over \mu_n}\Big)= {n\over 4} {\lambda_2-\lambda_n\over r-\lambda_n}.
\end{equation}
As shown in Proposition \ref{prop:hbionesymmetric},
equality $h_1(B_0(G))={1\over 2} \varphi_H(G)$ holds if $G$ is vertex- and edge-transitive.
Since the extended bipartite double  graph $B_0(G)$ is $(r+1)$-regular,  one can also upper bound 
$h_1(B_0(G))$ by the parameter $\widehat h(B_0(G))$ (as defined in Proposition~\ref{propeigbound}). By Lemma \ref{lembipeig} the second largest eigenvalue of the adjacency matrix of $B_0(G)$ equals $\max\{\lambda_2 +1,-\lambda_n-1\}$, and thus
\begin{equation}\label{eqhhat}
h_1(B_0(G))\le \widehat h(B_0(G))= {n\over 2} {\max\{\lambda_2+1,-\lambda_n-1\} \over \max\{\lambda_2 +1,-\lambda_n-1\} +r+1}.
\end{equation} 
Next we compare  the upper bounds in (\ref{eqHG5}) and (\ref{eqhhat}).

\begin{proposition}\label{lemphihhat}
Let $G$ be an $r$-regular graph. Then we have ${1\over 2}  \varphi_H(G) \le \widehat h(B_0(G)),$ with equality if and only if $\lambda_2=r$ or $\lambda_2+\lambda_n+2=0$. 
\end{proposition}

\proof{Proof.} 
Set $\mu:= \max\{\lambda_2 +1,-\lambda_n-1\}$ and note that ${1\over 2} \varphi_H(G) \le \widehat h(B_0(G))$ is equivalent to
$\psi:=\mu(\lambda_2+\lambda_n-2r) +(r+1)(\lambda_2-\lambda_n)\le 0.$
If $\lambda_2+\lambda_n+2\ge 0$ then $\mu=\lambda_2+1$ and we have $\psi=(\lambda_2-r)(\lambda_2+\lambda_n+2)\le 0$. Otherwise, $\lambda_2+\lambda_n+2\le 0$, $\mu=-\lambda_n-1$ and  we have 
$\psi= (r-\lambda_2)(\lambda_2+\lambda_n+2)\le 0$.
\endproof

 So Haemers' bound $\varphi_H(G)$ improves on  the bound $\widehat h(B_0(G))$ for any regular  graph $G$.
On the other hand,  also the reverse situation may occur,  where  the parameter $\widehat h$ improves on Haemers' bound $\varphi_H$. 
For this consider a bipartite graph $G=(V_1\cup V_2,E)$. As observed in  (\ref{eqbcG}), we have $\gbc(G)= g(\olG^b)$, where $\olG^b=(V_1\cup V_2, (V_1\times V_2)\setminus E)$ is the bipartite complement of $G$. Hence we have the inequalities
\begin{equation}\label{eqgbch}
\begin{split}
{1\over 2}\sqrt{\gbc(G)}={1\over 2}\sqrt{g(\olG^b)}\le h_1(\olG^b)\le \widehat h(\olG^b),\\
{1\over 2}\sqrt{\gbc(G)}={1\over 2}\sqrt{g(B_0(\olG))} \le h_1(B_0(\olG))\le {1\over 2} \varphi_H(\olG),
\end{split}
\end{equation}
where we assume that $G$ is regular when considering the parameters $\widehat h(\olG^b)$ and $\varphi_H(\olG)$.
Next we show 
 that $h_1(B_0(\olG))=h_1(\olG^b)$ and that $\widehat h(\olG^b)\le{1\over 2}  \varphi_H(\olG)$.
 
 \begin{proposition}\label{propeqHaemers}
 Let $G$ be a bipartite graph. Then we have $h_1(B_0(\olG))=h_1(\olG^b)$.
 Moreover, if $G$ is $r$-regular,   $n:=|V_1|=|V_2|$ and $\lambda_2$ denotes the second largest eigenvalue of $A_G$, 
 then we have 
   \begin{equation}\label{eqhHaemers}
\widehat h(\olG^b)= {n\over 2}{\lambda_2\over \lambda_2+n-r} \le {1\over 2}  \varphi_H(\olG)= {n\over 2} {\lambda_2+r\over 2n-r+\lambda_2},
\end{equation}
with strict inequality precisely when  $\lambda_2<r<n$, i.e., when $G$ is connected and $G\ne K_{n,n}$.
   \end{proposition}
 
\proof{Proof.} 
 First we prove $h_1(B_0(\olG))=h_1(\olG^b)$. For this we use the formulation (\ref{eqhGVd}) for the parameter $h_1(\cdot)$. Recall the definition (\ref{eqC}) of the matrix $C\in \mathcal S^{|V|}$   for the bipartition $V=V_1\cup V_2$, and let  $\widetilde C\in \mathcal S^{|V|+|V'|}$ denote the analogous matrix corresponding now to the bipartition $V\cup V'$, where $V=V_1\cup V_2$ and $V'=V_1'\cup V_2'$ is a disjoint copy of $V$. The matrices $\widetilde C$ and $A_{B_0(\olG)}$ have the form
 ${\tiny
 \widetilde C={1\over 2}\left(\begin{matrix} 0 & J & J & 0\cr J & 0 & 0 & J \cr J &  0 & 0 & J \cr 0 & J & J & 0\end{matrix}\right)
 }$  and 
$ {\tiny 
 A_{B_0(\olG)}= \left(\begin{matrix} 0 & A(\olG^b) & I & 0 \cr A(\olG^b) & 0 & 0 & I\cr I & 0 & 0 & A(\olG^b)\cr 0 & I &  A(\olG^b) & 0\end{matrix}\right)
   }$
 with respect to the partition $V_1\cup V_2'\cup V_1'\cup V_2$ (taken in that order), 
 setting $A(\olG^b):=A_{\olG^b}$ for easier notation.
 If $X\in \mathcal S^{|V|}$ is optimal for $h_1(\olG^b)$, then $Y:= {1\over 2} {\tiny \left(\begin{matrix} X & 0 \cr 0 & X\end{matrix}\right)}$
 is feasible for $h_1(B_0(\olG))$ with $\langle \widetilde C,Y\rangle =\langle C,X\rangle$, which shows 
 $h_1(B_0(G))\ge h_1(\olG^b)$.  
Conversely, assume $Y\in \mathcal S^{|V|+|V'|}$ is optimal for $h_1(B_0(\olG))$. Let $X$ (resp., $X'$) denote the principal submatrix of $Y$ indexed by $V_1\cup V_2'$ (resp., $V_1'\cup V_2$). Then $X/\text{Tr}(X)$ and $X'/\text{Tr}(X')$ are both feasible for $h_1(\olG^b)$, which implies $h_1(\olG^b)\cdot \text{Tr}(X) \ge \langle C,X\rangle$ and 
  $h_1(\olG^b)\cdot \text{Tr}(X') \ge \langle C,X'\rangle$. Summing up and using 
  $\text{Tr}(X)+\text{Tr}(X')=\text{Tr}(Y)=1$, we get 
  $h_1(\olG^b)\ge \langle C,X\rangle +\langle C,X'\rangle =\langle \widetilde C,Y\rangle =h_1(B_0(\olG)).$

Assume now $G$ is bipartite $r$-regular,  $\lambda_2=\lambda_2(A_G)$ and $n:=|V_1|=|V_2|$; we show (\ref{eqhHaemers}).
  First we compute the parameter $\widehat h(\olG^b)$. For this note that  $\olG^b$ is $(n-r)$-regular.
  Moreover, if  $M_G$  denotes the incidence matrix of $G$, then the incidence matrix of $\olG^b$ is  $J-M_G$, whose second largest singular value  is equal to the second largest singular value of $M_G$ and thus to $\lambda_2$. Hence, using relation (\ref{eq:hhateigenvaluebound}), we obtain $\widehat h(\olG^b)={n\over 2}{\lambda_2 \over n-r + \lambda_2}$, as desired. Next we compute the parameter $\varphi_H(\overline G)$. For this note that $\olG$ is $(2n-1-r)$-regular, the second largest eigenvalue of $A_{\olG}$ is $-1-\lambda_{\min}(A_G)=r-1$ and its smallest eigenvalue is $-1-\lambda_2(A_G)=-1-\lambda_2$. In view of (\ref{eqHG5})
 we get
  $\varphi_H(\olG)= {n}{ r+\lambda_2\over 2n-r+\lambda_2}$, as desired. 
  One can then easily check that the inequality in (\ref{eqhHaemers}) is equivalent to $(r-\lambda_2)(n-r)\ge 0$, which holds since $\lambda_2\le r\le n$.  Hence the inequality in (\ref{eqhHaemers}) is strict precisely when $ \lambda_2<r<n$, i.e., when $G$ is connected and $G\ne K_{n,n}$.
\endproof

\smallskip
We   summarize the various bounds obtained above for the parameter $\gbc(G)$ when $G$ is an arbitrary $r$-regular graph (Figure \ref{figG}) and when $G$ is bipartite $r$-regular (Figure \ref{figGbip}). As before  let $\lambda_1=r\ge \lambda_2\ge \ldots\ge \lambda_n$ denote the eigenvalues of $A_G$. Then $\olG$ is $(n-1-r)$-regular, with $\lambda_2(A_{\olG})= -1-\lambda_n$ and $\lambda_n(A_{\olG})=-1-\lambda_2$.

\begin{figure}[H]
\begin{subfigure}[b]{0.43\textwidth}
\scalebox{0.82}{\parbox{.43\linewidth}{\begin{align*} 
\frac{1}{2}\sqrt{\gbc(G)} \leq
\overunderbraces{&\br{3}{\substack{\text{with equality if $\olG$ is vertex-} \\ \text{and edge-transitive} \\ \text{\textbf{  Prop. \ref{prop:hbionesymmetric}}}}}} 
{ &h_1(B_0(\olG))& \leq &\frac{1}{2}\varphi_H(\olG) &\leq  \widehat h(B_0(\olG))&}{&&&\br{2}{\substack{\text{with equality if and only if}\\ \text{$\lambda_n=r-n$ or $\lambda_2+\lambda_n=0$} \\ \textbf{ Prop. \ref{lemphihhat}} }}}
\end{align*}}}\caption{\scriptsize Bounds on $\gbc(G)$ for  $G$ $r$-regular}\label{figG}
\end{subfigure}
\begin{subfigure}[b]{0.45\textwidth}
\scalebox{0.82}{\parbox{.49\linewidth}{\begin{align*} 
\frac{1}{2}\sqrt{\gbc(G)} \leq
\underbrace{h_1(B_0(\olG))=h_1(\olG^b)}_{\textbf{Prop. \ref{propeqHaemers}}}
\le 
\overunderbraces{&\br{2}
{\substack {\text{with equality if and only if }\\ \text{$\lambda_2=r$ or $r=n$} \\ \textbf{ Prop \ref{propeqHaemers}}}}&} 
{ &\widehat h(\olG^b) \le& {1\over 2}\varphi_H(\olG) & \le  \widehat h(B_0(\olG))&} 
{&& \br{2}
{\substack{\text{with equality if and only if}\\ \text{$\lambda_n=r-n$ or $\lambda_2+\lambda_n=0$}\\\textbf{Prop. \ref{lemphihhat}}}}} 
\end{align*}}}\caption{\scriptsize Bounds on $\gbc(G)$ for  $G$ bipartite $r$-regular}\label{figGbip}
\end{subfigure}
\caption{Bounds on $\gbc(G)$; recall $h_1(B_0(\olG))\le \widehat h(B_0(\olG))$, with equality if $B_0(\olG)$ is edge-transitive (Proposition~\ref{prop:hbbioneedgetransitive}).}
\end{figure}

\section{Examples.}\label{secexamples}

We now illustrate the behaviour of the various parameters discussed above  on some classes of  regular graphs. Recall the definition of the matrix  $M_G$ in Lemma \ref{lembipeig}.

\subsection{The perfect matching.}\label{secperfectM}

For $n\ge 2$, let~$G$ be a perfect matching on~$2n$ vertices. Then~$M_G = I$, $r=1$, $\lambda_2=1$, and $G$ is vertex- and edge-transitive. Using Proposition \ref{propeigbound} we obtain  
$$
h_1(G) =\widehat{h}(G) = {n\over 2}\frac{ \lambda_2}{r+\lambda_2} = \frac{n}{4} \quad  \text{ and } \quad
g_1(G) =\widehat{g}(G) =  \frac{n^2}{4}. 
$$
We have $g(G) = \floor{n/2}\ceil{n/2}$ and $h(G)={1\over n}  \floor{n/2}\ceil{n/2}$ (obtained by maximizing $ab$ and ${ab\over a+b}$ with  $a,b\ge 0$  integers and $a+b\le n$). 
Hence, $h_1(G)={1\over 2}\sqrt {g_1(G)}$ and $h_1(G)$, $g_1(G)$  give tight bounds for $h(G)$, $g(G)$ (with equality for $n$ even and up to rounding for $n$ odd).

\subsection{The complete bipartite graph \texorpdfstring{$K_{n,n}$}{Kn,n} minus a perfect matching.}\label{secKnnminM}

For $n\ge 2$, let $G$ be the complete bipartite graph~$K_{n,n}$ with  a deleted perfect matching (also known as the {\em crown graph} on $2n$ vertices). Then~$G$ is vertex- and edge-transitive, $(n-1)$-regular, $M_G = J_n-I_n$, and $\lambda_2=1$. We have $\hbbi(G)=\tfrac{1}{2}$ and~$\gbbi(G)=1$. 
Using Proposition \ref{propeigbound} we obtain
$$
h_1(G) =\widehat{h}(G) ={n\over 2} \frac{ \lambda_2}{r+\lambda_2} =\frac{1}{2}, \,\text{ and } \, 
g_1(G) =\widehat{g}(G) =\begin{cases}
  \frac{n^2}{8(n-2)} & n \geq 4, \\
  1 & n \leq 4. \end{cases}    
$$
Hence the bound $h_1(G)$ is tight for both $h(G)$ and ${1\over 2}\sqrt {g(G)}$, while the ratio $g_1(G)/g(G)$ grows linearly in $n$. Note that $h_1 (G)<{1\over 2}\sqrt{g_1(G)}$ 
for $n\ge 5$, which gives an example with strict separation between the parameters $h_1$ and ${1\over 2}\sqrt{g_1}$ (and thus $\widehat h$ and ${1\over 2}\sqrt{\widehat g}$).  
In view of  (\ref{eqgbch}), the parameter $\gbc(G)$  
  is upper bounded by  $4\widehat h(\olG^b)^2$ and by $\varphi_H(\olG)^2$.
 Note that    $4\widehat h(\olG^b)^2= 4({n\over 4})^2={n^2\over 4}$, which  improves  on Haemers' bound  $\varphi_H(\olG)^2= ({n^2\over n+2})^2$ for $n\ge 3$. This thus gives a class of graphs for which strict inequality holds in (\ref{eqhHaemers}). 

\subsection{The cycle graph~\texorpdfstring{$C_n$}{Cn}.} 

Let $G$ be the cycle $C_n$ on $n\ge 3$ vertices, which is vertex- and edge-transitive, and $2$-regular.
The eigenvalues of the adjacency matrix~$A_{C_n}$ are~$2 \cos(2\pi j / n)$ where~$j=0,\ldots,n-1$ (see, e.g., \cite{BH17}), so~$\lambda_2(A_{C_n})=2\cos(2\pi/n)$,  and $\lambda_n(A_{C_n}) = -2$ if~$n$ is even,  $\lambda_n(A_{C_n}) = -2\cos(\pi/n)$ if $n$ is odd.  

First we compute the parameters for the extended bipartite double graph $B_0(C_n)$. 
Using Proposition \ref{prop:hbionesymmetric} and relations (\ref{eqHG5}), (\ref{eqhhat}), 
we get
$$ h_1(B_0(C_n)) ={1\over 2}\varphi_H(C_n)=\begin{cases} \frac{n}{4} \cos(\pi/n)^2 &\mbox{if } n \text{ even}, \\
  \frac{n}{4} (2 \cos(\pi/n)-1)&\mbox{if } n\text{ odd}, \end{cases} \,\,\, \widehat{h}(B_0(C_n)) ={n\over 4} \frac{ 2\cos(2\pi/n)+1}{\cos(2\pi/n)+2}.
$$
Hence we have  $h_1(B_0(C_n)) = \widehat{h}(B_0(C_n)) (=0)$ for $n=3$ (in which case $B_0(C_3)=K_{3,3}$), and 
strict inequality  $h_1(B_0(C_n)) < \widehat{h}(B_0(C_n))$  for $n\ge 4$  (as expected from Proposition \ref{lemphihhat}). Note also that $B_0(C_n)$ is not edge-transitive if $n\ge 4$.  One can also show that 
\begin{align*}
h(B_0(C_n)) =\begin{cases} \tfrac{1}{4}(n-2) &\mbox{if } n \text{ even}, \\
  \frac{(n-1)(n-3)}{4(n-2)}&\mbox{if } n \text{ odd}, \end{cases} \,
g(B_0(C_n)) =\begin{cases} \tfrac{1}{4}(n-2)^2&\mbox{if } n \text{ even}, \\
  \tfrac{1}{4}(n-1)(n-3) &\mbox{if } n \text{ odd}. \end{cases}
\end{align*} 
So~$h(B_0(C_n))\le {1\over 2}\sqrt{g(B_0(C_n))}$, with equality for $n$ even. Moreover,  the ratio $\widehat h(B_0(C_n))/h(B_0(C_n))$ tends to 1 as $n\rightarrow \infty$, so the bound $\widehat h(B_0(C_n))$ (and thus $h_1(B_0(C_n))$ too) is asymptotically tight for $h(B_0(C_n))$ and ${1\over 2}\sqrt{g(B_0(C_n))}$.

\medskip
For $n$ even the graph $G=C_n$ is bipartite. Then we have
$$h(C_n)\le h_1(C_n)=\widehat h(C_n)= {n\over 4}{\lambda_2\over \lambda_2+r}= {n\over 4} {\cos(2\pi/n)\over \cos(2\pi/n)+1}\le {\alpha(C_n)\over 4}={n\over 8}.$$
So $h_1(C_n)=\Theta (n/8)=\Theta(\alpha(C_n)/4)$.
Moreover, one can construct a bipartite biindependent pair $(A,B)$ showing  $h(C_n)=\Theta(n/8)$ (see also \cite{CLZXL2021}). Namely, for $n\equiv 0 \pmod{4}$, set $A=\{1,3,\ldots,{n\over 2}-1\}$, $B=\{{n\over 2}+2, {n\over 2}+4,\ldots, n-2\}$  with $|A|={n\over 4},|B|={n\over 4}-1$, and, for  $n\equiv 2 \pmod{4}$, set $A=\{1,3,\ldots,{n\over 2}-2\}$, $B=\{{n\over 2}+1,{n\over 2}+3,\ldots, n-2\}$ with $|A|=|B|={n-2\over 4}$.

\subsection{The hypercube graph~\texorpdfstring{$Q_r$}{Qr}.}

The hypercube graph~$Q_r$ is the bipartite graph with vertex set $V=\{0,1\}^r$, where two vertices are adjacent when their Hamming distance is~$1$. So the bipartition is $V=V_1\cup V_2$, where $V_1$ (resp., $V_2$) consists of all $x\in V$ with an even (resp., odd)  Hamming weight $|x|$. The graph~$Q_r$ is vertex- and edge-transitive, and~$r$-regular. The eigenvalues of $A_{Q_r}$ are~$r-2k$ for~$k=0,\ldots,r$, where the eigenvalue~$r-2k$ has multiplicity~$\tbinom{r}{k}$. So~$\lambda_2(A_{Q_r})=r-2$.  
Thus the parameter $h_1(Q_r)$ is given by
\begin{align*}
h_1(Q_r) =\widehat{h}(Q_r) =2^{{r-3}}\frac{r-2}{r-1}.
\end{align*}
One can show that $\lim_{r\to \infty} h_1(Q_r)/h(Q_r) =1$. For this, we will show that $h(Q_r)\ge {a(r-1)\over 4}$, 
where the sequence $(a(r))_{r\ge 0}$ is defined recursively by 
\begin{align}\label{eqar}
a(2r):=2^{2r}-\binom{2r}{r},\ a(2r+1):=2\cdot a(2r)  \text{ if } r\ge 1, \text{ and } a(0)=0.
\end{align}
 Using the fact that ${2r\choose r}\sim {2^{2r}\over \sqrt{\pi r}}$ one  obtains $a(r-1) \sim 2^{r-1}$ and $h(Q_r)\ge 2^{r-3}(1-c/\sqrt r)$ (for some constant $c>0$) and thus $h_1(Q_r)/h(Q_r)$ tends to 1 as $r\to\infty$.  Note that the  bound~$h(Q_r) \leq \alpha(Q_r)/4=2^{r-1}/4=2^{r-3}$ from Lemma~\ref{lemrel1} is slightly weaker than~$h(Q_r)\le h_1(Q_r)$, but already strong enough to exhibit $h(Q_r) \sim 2^{r-3}$ (when combined with the lower bound $h(Q_r)\ge {a(r-1)\over 4}$).

We now  show that $h(Q_r)\ge {a(r-1)\over 4}$. For this, it is useful to observe that the graph $Q_r$ is isomorphic to $B_0(Q_{r-1})$, the extended bipartite double of $Q_{r-1}$ (the bipartition of $Q_r$ provides the bipartition of $B_0(Q_{r-1})$ by simply deleting the last coordinate in all vertices of $Q_r$). Thus we have $h(Q_r)=h(B_0(Q_{r-1})) = \hbi(Q_{r-1})$, where the last equality follows from (\ref{eqbi}).
Hence, instead of searching for bipartite biindependent pairs in $Q_r$ we may as well search for (general) biindependent pairs in $Q_{r-1}$, which is a simpler task. We show that $\hbi(Q_r)\ge {1\over 4}a(r)$ for all $r\ge 1$.
First consider the case of $Q_{2r}$. Define the sets 
$$L:=\{x\in \{0,1\}^{2r}: |x|\leq r-1\},\quad U:=\{x\in \{0,1\}^{2r}: |x|\geq r+1\}.$$
Then, $(L,U)$ is a (balanced) biindependent pair in $Q_{2r}$, with $|L|=|U|={1\over 2}\big(2^{2r}-{2r\choose r}\big)={1\over 2}a(2r)$, which implies $\hbi(Q_r)\ge {1\over 4}a(2r)$.
Consider now the case of $Q_{2r+1}$. Define  $L':=L\times \{0,1\}$ and  $ U':=U\times \{0,1\}\subseteq \{0,1\}^{2r+1}$. Then the pair $(L',U')$ is (balanced) biindependent in $Q_{2r+1}$, with $|L'|=|U'|=a(2r)={1\over 2}a(2r+1)$, which implies $\hbi(Q_{2r+1})\ge {1\over 4}a(2r+1).$

The above construction can be used to show that $\alphabal(Q_r)\ge a(r-1)$ for all $r\ge 1$. 
For this, given  $A\subseteq \{0,1\}^r$, define the following subsets of $\{0,1\}^{r+1}$ obtained by adding a parity bit,
\begin{align*}A_{\rm even}:=\{(x,|x|\text{ mod } 2): x\in A\}\subseteq \{0,1\}^{r+1},\
A_{\rm odd}:=\{(x,|x|+1\text{ mod } 2): x\in A\}\subseteq \{0,1\}^{r+1}.
\end{align*}
Applying this to the above sets $L,U\subseteq \{0,1\}^{2r}$, we obtain $L_{\rm even}, U_{\rm odd}\subseteq \{0,1\}^{2r+1}$ such that $(L_{\rm even}, U_{\rm odd})$ is balanced bipartite biindependent in $Q_{2r+1}$ with $|L_{\rm even}|=|U_{\rm odd}|= |L|=a(2r)/2$, which implies $\alphabal(Q_{2r+1})\ge a(2r)$. Similarly, using the sets $L',U'\subseteq \{0,1\}^{2r+1}$, we obtain $L'_{\rm even},U'_{\rm odd}\subseteq \{0,1\}^{2r+2}$ that provide a balanced bipartite biindependent pair in $Q_{2r+2}$ with
$|L'_{\rm even}|=|U'_{\rm odd}|=|L'|=a(2r+1)/2$, which implies 
$\alphabal(Q_{2r+2})\ge a(2r+1)$.
\begin{conjecture}\label{conjecture:cube}
We conjecture that equality $\alphabal(Q_r)=a(r-1)$ holds for all $r\ge 1$.
\end{conjecture}
We have verified numerically that Conjecture~\ref{conjecture:cube} indeed holds for any $r\le 13$. For $r\le 8$ this can be verified using 
an integer programming solver (like Gurobi \cite{Gurobi}).
For larger values  $r\le 13$ we show this in an indirect manner. We consider the semidefinite  upper bound on $\alphabal(Q_r)$ that is obtained from the Lasserre relaxation of order 2. After applying a symmetry reduction (as done in \cite{GMS,LPS17}), we solve the resulting semidefinite program numerically and obtain an upper bound that coincides with $a(r-1)$ for  $r\le 13$. 
In addition,  $\alphabal(Q_r)/a(r-1)\to 1$ as $r\to\infty$ since $\alphabal(Q_r)\le \alpha(Q_r)=2^{r-1}$ and  $a(r-1)\sim 2^{r-1}$. 
 
Observe that $\alphabal(Q_{r+1})\ge 2\cdot \alphabal(Q_r)$. For this, for $x\in\{0,1\}^r$ let $x'\in \{0,1\}^r$ be obtained by switching the last bit of $x$, so that the weights of $x,x'$ have distinct parities and, for a set $A\subseteq \{0,1\}^r$ and $b\in \{0,1\}$, define $Ab:=\{(x,b): x\in A\}\subseteq \{0,1\}^{r+1}$. The claim now follows from the fact that if 
 $(A,B)$ is a balanced bipartite biindependent pair in $Q_r$, then the pair 
 $(B1 \cup B'0, A1 \cup A'0)$ is balanced bipartite biindependent in $Q_{r+1}$ with size $2|A\cup B|$.
 Hence, the above conjecture implies equality $\alphabal(Q_{r+1})= 2\cdot \alphabal(Q_r)$ for $r$ odd.
 
 Interestingly, the sequence $a(r)$ in (\ref{eqar}) corresponds to the sequence A307768 in OEIS \cite{oesis.org}, which counts  
the number of heads-or-tails games of length $r$ during which at some point there are as many heads as tails. 
It is also related to several other well-known combinatorial counting problems; see, e.g., \cite{EK99} or \cite[Chapter III]{Fel} for an overview. It is interesting to understand the exact relationship of this sequence with the parameter $\alphabal(Q_r)$.

\section{Lasserre bounds for the balanced parameters.}\label{sec:balanced}

In this section we turn our attention to the ``balanced" parameters $\alphabal(G)$,  $\gbal(G)$ and $\hbal(G)$ that are obtained by restricting the optimization to balanced bipartite biindependent pairs in the definition of $\alpha(G)$,  $g(G)$ and $h(G)$. Recall from  (\ref{eqrel0}) that  
${1\over 4}\alphabal(G)={1\over 2}\sqrt{\gbal(G)}=\hbal(G)$.  
Since these are NP-hard parameters  one is interested in finding efficient bounds for them, strengthening those for the original parameters $g(G)$ and $h(G)$.

Let $G=(V=V_1\cup V_2,E)$ be a bipartite graph. Following the approach in Section \ref{secpop}, each of the parameters $\alphabal(G)$, $\gbal(G)$ and $\hbal(G)$  has a natural polynomial optimization formulation, which offers the starting point to define several hierarchies of semidefinite relaxations.  For this define the vector $f:=\chi^{V_1}-\chi^{V_2}$. Let $\IGbal$ denote the ideal in $\R[x]$ that is generated by the ideal $I_G$ (itself generated by $x_i^2-x_i$ for $i\in V$ and $x_ix_j$ for $\{i,j\}\in E$) and the polynomial $f\T x$.  For an integer $t$ let $I_{G,\text{\rm bal},t}$ denote its truncation at degree $t$, where all summands are restricted to have degree at most $t$. 
Then the formulation for $\alphabal(G)$ follows by replacing the ideal $I_G$ (resp., $I_{G,2\alpha(G)}$) by the ideal $\IGbal$ (resp., $I_{G,\text{\rm bal},2\alpha(G)}$) in (\ref{eqalphapos}) (resp., (\ref{eqalphasos})).
Similarly, $\gbal(G)$ (resp., $\hbal(G)$) is obtained by adding the ``balancing" constraint $f\T x=0$ to the program (\ref{eqgbbipop}) defining $g(G)$ (resp., to the program (\ref{eqhbbipop}) defining $h(G)$). 
Now each of these polynomial optimization formulations can be used to define a Lasserre-type hierarchy. In this way one obtains 
the  hierarchies $\lasbalr(G)$, $\gbalr(G)$, and $\hbalr(G)$ for $r\in\N$ that converge to $\alphabal(G)$, $\gbal(G)$, and $\hbal(G)$, respectively, after $r\ge \alpha(G)$ steps. They are obtained, respectively, from the programs (\ref{eqalphar}) (defining $\las_r(G)$), (\ref{eqgbbipopr}) (defining $g_r(G)$), and (\ref{eqhbbipopr}) (defining $h_r(G)$) by replacing the truncated ideal $I_{G,2r}$ by its balanced analog $I_{G,\text{\rm bal},2r}$; that is,
\begin{align*}
\lasbalr(G) =\min\big\{\lambda: \lambda-x\T x\in \Sigma_2+I_{G,\text{\rm bal},2r} \big\},\\
\gbalr(G)=\min\big\{\lambda: \lambda -x\T Cx\in \Sigma_2+I_{G,\text{\rm bal},2r}\big\},\\
\hbalr=\min\big\{\lambda: x\T(\lambda I- C)x\in \Sigma_2+I_{G,\text{\rm bal},2r}\big\}.
\end{align*}
We will now   focus on the Lasserre bounds of  order $r=1$. We will give explicit semidefinite formulations and show relationships between the various parameters. The parameter $\lasbalone(G)$  is the analog of  $\las_1(G)= \vartheta(G)$ obtained by adding a balancing constraint to program (\ref{eqthetaI}). However, adding a balancing constraint to  the formulation of $\vartheta(G)$ in  (\ref{eqthetaJ}) leads to another parameter $\thetabal(G)$ that is in general weaker than $\lasbalone(G)$. The parameters $\gbalone(G)$ and $\hbalone(G)$ are  obtained by adding a balancing constraint to the respective programs defining $g_1(G)$ and $h_1(G)$. Moreover, they can be shown to be nested between $\lasbalone(G)$ and $\thetabal(G)$, see Proposition~\ref{propcomparebal} below.  For bipartite regular graphs we will investigate some natural symmetric variations  
of these parameters, with the hope of obtaining a new closed-form parameter strengthening $\widehat h(G)$. However, as we will show, 
it turns out that in all cases one recovers the  parameter $\widehat h(G)$, see Propositions~\ref{prop:symbalhat} and \ref{prop:symbalhat2}. So  the refined formulations taking into account the balancing constraints do not yet lead to stronger eigenvalue bounds for the parameter $\alphabal(\cdot)$.

\subsection{The Lasserre bounds of order \texorpdfstring{$r=1$}{r=1} for the balanced parameters.}

We begin with semidefinite reformulations for the parameter $\lasbalone(G)$.
\begin{lemma}\label{lemlasbalone}
For any bipartite graph  $G=(V,E)$ we have
\begin{align}
\lasbalone(G) & =  \max_{ X \in \mathcal{S}^{|V|}} \Big\{ \langle I, X \rangle :   \,  \setlength\arraycolsep{0pt}\left(\begin{matrix} 1 & \diag(X)\T\cr \diag(X)& X\end{matrix}\right) \succeq 0, \, X_{ij} = 0 \text{ if } \{i,j\} \in E, \ \langle ff\T,X\rangle =0\Big\},\label{eq:thetalasbal} \\
&=\min_{Z\in \mathcal S^{|V|}, u\in \R^{|V|}, s\in \R}\left\{ \lambda: \begin{pmatrix} \lambda & -u\T/2 \\ -u/2 & \Diag(u)-I+Z +sff\T \end{pmatrix}\succeq 0, \ Z\in \mathcal S_G \right\}.\label{eq:thetalasbaldual}
\end{align}
\end{lemma}

\proof{Proof.} 
As in Section \ref{secsdphg} the proof   uses Lemma \ref{lemMIG}.
By definition,  $\lasbalone(G)$ is the smallest scalar $\lambda$ 
for which $\lambda -x\T Ix \in \Sigma_2+\IGbaltwo$, i.e., $\lambda -x\T Ix - (a_0+a\T x)f\T x \in \Sigma_2+I_{G,2}$ for some $a_0\in \R$, $a\in \R^n$. 
Thus $\lasbalone(G)$ is the smallest $\lambda$ such that 
$[x]_1\T \big(Q- {\tiny \left(\begin{matrix} \lambda & a_0f\T/2 \cr a_0f/2  & -I+ {af\T+fa\T\over 2}\end{matrix}\right)}  \big)[x]_1 \in I_{G,2}$ for some $a_0\in \R, a\in \R^n$.
Applying Lemma \ref{lemMIG} we arrive at the program  
$$\lasbalone(G)=\min_{Z\in \mathcal S^{|V|}, u,a\in\R^{|V|}, a_0\in \R}
\Big\{ \lambda : 
{ \left(\begin{matrix} \lambda & {1\over 2}(-u+a_0f)\T\cr {1\over 2}(-u+a_0f) &
\Diag(u)-I+Z+{af\T +fa\T\over 2}\end{matrix}\right)}
\succeq 0, Z\in \mathcal S_G\Big\}.$$
Now we take the  dual of this semidefinite program. We also  apply some simplifications, such as observing  that $Xf=0$ is equivalent to $\langle ff\T,X\rangle =0$ when $X\succeq 0$, which in turn implies 
$f\T \diag(X)=0$ when  ${\tiny \left(\begin{matrix} 1 & \diag(X)\T \cr \diag(X) & X\end{matrix}\right)}\succeq 0$. In this way we arrive at the program (\ref{eq:thetalasbal}). Taking  the dual of  (\ref{eq:thetalasbal}) gives the (simplified) program (\ref{eq:thetalasbaldual}). Note that strong duality holds since program (\ref{eq:thetalasbaldual}) is strictly feasible (e.g., take $s=0$, $Z=0$, $u=\mu e$ with $\mu>1$, and $\lambda>{n\over 4}{\mu^2\over \mu-1}$).
\endproof

\smallskip
Hence program (\ref{eq:thetalasbal}) is the analog of program (\ref{eqthetaI}) defining $\las_1(G)=\vartheta(G)$ to which we add the balancing condition $\langle ff\T,X\rangle =0$.
Next we consider the analog of program (\ref{eqthetaJ}) to which we add the balancing conditions $\langle ff\T,X\rangle =0$ and $f\T \diag(X)=0$, giving the parameter
\begin{equation}\label{eq:thetabal}
\begin{split}
\thetabal(G):=\max_{X\in \mathcal S^{|V|}}\big\{\langle J,X\rangle: &\, X\succeq 0, \  \text{\rm Tr}(X)=1,\ X_{ij}=0 \text{ if } \{i,j\}\in E,   \\
& \langle ff\T,X\rangle =0,\  \langle \Diag(f),X\rangle =0 \big\}, 
\end{split}
\end{equation}
\begin{align}
&=\min_{Z\in \mathcal S^{|V|}, \lambda, s, v\in \R} \left\{ \lambda : \lambda I-J+Z +v \Diag(f)+sff\T  \succeq 0,\, 
Z \in \mathcal S_G\right\},\label{eq:thetabaldual}
\end{align}
where the second formulation (\ref{eq:thetabaldual}) follows by taking the dual of (\ref{eq:thetabal}) (and observing that  (\ref{eq:thetabaldual}) is strictly feasible).
We will see in Proposition~\ref{propcomparebal}  below that $\thetabal(G)$ provides a weaker bound for $\alphabal(G)$  than $\lasbalone(G)$.

\medskip
We now consider the parameter $\gbalone(G)$.
By definition,  $\gbalone(G)$ is  the smallest scalar $\lambda$ for which $\lambda -x\T Cx\in \Sigma_2+\IGbaltwo$. Comparing with the definition of $\lasbalone(G)$ we see that it suffices to exchange the matrices $C$ and $I$ to get the semidefinite formulations of $\gbalone(G)$ in the next lemma (recall also Remark \ref{remlasone}).

\begin{lemma}
For any bipartite graph  $G=(V,E)$ we have
\begin{align} 
 \gbalone(G) &=\max_{ X \in \mathcal{S}^{|V|}} \Big\{ \langle C, X \rangle:     \left(\begin{matrix} 1 & \diag(X)\T\cr \diag(X)& X\end{matrix}\right) \succeq 0, \, X_{ij} = 0 \text{ if  }\{i,j\} \in E,\,  \langle ff\T,X \rangle =0  \Big\},
\label{eq:gbalone} \\
&= \min_{\lambda,s\in \R, u\in \R^{|V|}, Z\in \mathcal S^{|V|}} \Big\{\lambda:\left(\begin{matrix} \lambda & -u\T/2\cr
-u/2 & \Diag(u)-C+Z+sff\T \end{matrix}\right)\succeq 0,\ Z\in\mathcal S_G\Big\}.\label{eq:gbalonedual} 
 \end{align}
\end{lemma}

Finally we give semidefinite formulations for  the parameter $\hbalone(G)$.

 \begin{lemma} \label{lemhbal1}
Let $G=(V,E)$ be a bipartite graph. Then  we have 
\begin{equation}\label{eq:hbalone}
\begin{split}
 \hbalone(G)=\max_{X\in\mathcal S^{|V|}}
\{\langle C,X\rangle:\  & X\succeq 0,\ \text{\em Tr}(X)=1,\ X_{ij}=0\ \text{ if } \{i,j\}\in E, \\
&\,  \langle f f\T, X\rangle =0, \, \langle \Diag(f), X\rangle  = 0 \}, 
\end{split}
\end{equation}
\begin{equation}\label{eq:hbalonedual}
 \hbalone(G)= \min_{\substack{\lambda,v,s\in \R, Z\in \mathcal S^{|V|}}} \{\lambda: \lambda I-C +Z+v \Diag(f) + s ff\T \succeq 0, \  Z\in\mathcal S_G\}.
\end{equation}
\end{lemma}

\proof{Proof.} 
The argument is similar to the one used to show Lemma \ref{lemlasbalone}. Namely, one starts with the definition of $\hbalone(G)$ as the smallest  $\lambda$ for which $x\T(\lambda I-C)x\in \Sigma_2 +\IGbaltwo$. Using Lemma \ref{lemMIG} one arrives at a semidefinite program whose dual can be shown (after some simplifications) to take the form (\ref{eq:hbalone}). Then one takes the dual of program (\ref{eq:hbalone}), which has the form (\ref{eq:hbalonedual}).
\endproof

\smallskip
We now compare the parameters $\lasbalone(G)$, $\thetabal(G)$, $\gbalone(G)$ and $\hbalone(G)$.

\begin{proposition}\label{propcomparebal}
For any bipartite graph~$G$, we have the inequalities
$${1\over 4}\lasbalone(G) \leq  {1\over 2}\sqrt{\gbalone(G)} \leq  \hbalone(G) ={1\over 4} \thetabal(G).$$
Moreover, we have  
${1\over 2}\sqrt{\gbalone(G)} = {1\over 4} \thetabal(G)$ $\Longleftrightarrow$ $\lasbalone(G)=\thetabal(G)$.
\end{proposition}

\proof{Proof.} 
The equality $\thetabal(G)=4\hbalone(G)$ follows from the fact that the programs (\ref{eq:thetabal}) (defining $\thetabal(G)$) and (\ref{eq:hbalone}) (defining $\hbalone(G)$) differ only in their objective functions that are, respectively, $\langle J,X\rangle$ and $\langle C,X\rangle$, combined with the identity  $J-4C=ff\T$.

The inequality $\lasbalone(G)\le \thetabal(G)$ follows using the formulations (\ref{eq:thetalasbal}) and (\ref{eq:thetabal}) and a classic argument (repeated for convenience). If $X$ is optimal for (\ref{eq:thetalasbal}) with $x:=\diag(X)$, then $X-xx\T\succeq 0$, $f\T x=0$, $\text{\rm Tr}(X)=e\T x$, so $X/\text{Tr}(X)=X/e\T x$ is feasible for  (\ref{eq:thetabal}) and thus we have
 $\thetabal(G)\ge {1\over e\T x}\langle J,X\rangle\ge {1\over e\T x} \langle J,xx\T\rangle =e\T x =\lasbalone(G)$.
 
For the inequality $\lasbalone(G)^2\le 4\cdot \gbalone(G)$, pick  an optimal solution $X$ for (\ref{eq:thetalasbal}) with $x:=\diag(X)$, so that $X-xx\T\succeq 0$,  and use again the fact that $4C=J-ff\T$. Then we have 
 $4\cdot \gbalone(G)\ge \langle 4C,X\rangle = \langle J,X\rangle \ge \langle J,xx\T\rangle =(e\T x)^2 = \langle I,X\rangle ^2=\lasbalone(G)^2$. 
 
We now show the inequality $4\cdot \gbalone(G)\le \thetabal(G)^2$. For this let $X$ be optimal for program (\ref{eq:gbalone}) defining $\gbalone(G)$. Then $X$ is feasible for (\ref{eq:thetalasbal}) and thus 
$\lasbalone(G)\ge \text{Tr}(X)$.
In addition, $X/\text{Tr}(X)$ is feasible for (\ref{eq:thetabal}) and thus $\thetabal(G)\ge {1\over \text{Tr}(X)}\langle J,X\rangle$.
Using $4C=J-ff\T$, we obtain 
$4\cdot \gbalone(G)=\langle 4C,X\rangle = \langle J,X\rangle =  \text{Tr}(X) \cdot\langle J,X/\text{Tr}(X)\rangle\le
 \lasbalone (G)\cdot \thetabal(G)\le  \thetabal(G)^2.$
Finally, this argument also shows that equality $4\cdot \gbalone(G)=\thetabal(G)^2$ implies
$\lasbalone(G)=\thetabal(G)$, which concludes the proof. 
 \endproof

\smallskip
Quite surprisingly, while we had the inequality $h_1(G)\le {1\over 2} \sqrt{g_1(G)}$ (recall Proposition \ref{propghg1h1}), we now have the reverse inequality ${1\over 2}\sqrt{\gbalone(G)}\le \hbalone(G)$ for the balanced analogs. We next give an example where this inequality is strict.

\smallskip
\begin{example}\label{remarkstrictbal}
Let~$G$ be the bipartite graph from Figure~\ref{fig:smallgraphexample}.
One can check that $\hbalone(G)=2/3$, $\gbalone(G)=4/3$ and $\lasbalone(G)= 9/4$, which shows that the strict inequalities 
${1\over 4}\lasbalone(G) <  {1\over 2}\sqrt{\gbalone(G)} <  \hbalone(G)$ hold.
To see this consider the matrices
\begin{align*}
X_1=\tfrac{1}{12}\begin{pmatrix}
1 & 1 & 0 &2 \\
1 & 5 & 2 &4\\
0 & 2 & 1 & 1\\
2 & 4 & 1 & 5
\end{pmatrix},\quad 
X_2=\tfrac{1}{9}\begin{pmatrix}
3 & 1 & 0 &4 \\
1 & 7 & 4 &4\\
0 & 4 & 3 & 1\\
4 & 4 & 1 & 7
\end{pmatrix}, \quad
X_3=\tfrac{1}{32}\begin{pmatrix}
12 & 3 & 0 &15 \\
3 & 24 & 15 &12\\
0 & 15 & 12 & 3\\
15 & 12 & 3 & 24
\end{pmatrix}.
\end{align*}
Then, $X_1$ is feasible for (\ref{eq:hbalone}) with $\langle C,X_1\rangle =2/3$, $X_2$ is feasible for (\ref{eq:gbalone}) with $\langle C,X_2\rangle =4/3$, and $X_3$ is feasible for (\ref{eq:thetalasbal}) with $\langle I,X_3\rangle = 9/4$. One can check optimality of these solutions for the respective programs (for this, use the constraint $\langle ff\T,X\rangle =0$ to  reduce the semidefinite program to an equivalent semidefinite program involving smaller matrices, and then construct a solution of the dual program with the same objective value).  
\end{example}

\subsection{Symmetric versions of the parameters \texorpdfstring{$\lasbalone(G)$, $\thetabal(G)$ and $\gbalone(G)$}{lasbal1(G), thetabal(G) and gbal1(G)}.}

Here we address the question whether it is possible to obtain closed-form eigenvalue-based upper bounds  for $\alphabal(G)$ that improve on the spectral parameter $\widehat h(G)$ from (\ref{eq:hhateigenvaluebound}).  
For this a natural approach is to restrict the optimization in the programs (\ref{eq:thetalasbaldual}), (\ref{eq:thetabaldual}), (\ref{eq:gbalonedual}) to  matrices $Z=tA_G$ for some $t\in \R$ and, for (\ref{eq:thetalasbaldual}) and (\ref{eq:gbalonedual}), to vectors $u=\mu e$ for some $\mu\in \R$.  
Moreover, we add a term $v\Diag (f)$ to the matrix involved   in (\ref{eq:thetalasbaldual}) and (\ref{eq:gbalonedual}), which amounts to adding the redundant constraint $\langle \Diag(f),X\rangle =0$ to the programs (\ref{eq:thetalasbal}) and (\ref{eq:gbalone}). The motivation for this is to get possibly sharper bounds. In addition, the bounds obtained in this way are easier to compare (see Proposition \ref{prop:hatcomparison}). However, as we will show in Proposition \ref{prop:symbalhat}, these additional constraints will turn out to be redundant for bipartite regular graphs.

So we consider the parameters
\begin{align}
&\lasbalhat(G) := \min_{\lambda,\mu,t,s,v \in\R} \{ \lambda: 
\left(\begin{matrix}\lambda &-\mu e\T/2\cr -\mu e/2 & (\mu-1)I + tA_G +sff\T+v\Diag(f)\end{matrix}\right)\succeq 0
\},\label{eq:thetalasbalhat}\\
& = \max_{X \in \mathcal{S}^{|V|}, x\in \R^{|V|}} \Big\{ \langle I, X \rangle:     \left(\begin{matrix} 1 & x\T\cr x& X\end{matrix}\right) \succeq 0,  \text{Tr}(X)=e\T x,   \langle  A_G,X \rangle = 0,   \langle ff\T,X\rangle =0, \langle \Diag(f),X\rangle =0 \Big\},\label{eq:thetalasbalhatI}
  \end{align}
  \begin{align}
& \thetabalhat(G) := \min_{\lambda,t,v,s\in\R} \{ \lambda: \lambda I-J+tA_G +v \Diag(f) +s f f\T  \succeq 0\},\label{eq:thetabalhat}\\
&=\max\{\langle J,X\rangle:\ X\succeq 0,\ \text{Tr}(X)=1,\ \langle A_G,X\rangle = 0, \, \langle f f\T,X\rangle =0,\ \langle \Diag(f),X \rangle = 0\},\label{eq:thetabalhatJ}
\end{align}
\begin{align}
& \gbalonehat(G) := \min_{\lambda,\mu,t,s,v\in \R} \Big\{\lambda:\left(\begin{matrix} \lambda & -\mu e\T /2\cr
-\mu e/2 & \mu I-C+tA_G+sff\T +v\Diag(f)\end{matrix}\right)\succeq 0\Big\},\label{eq:gbalonehat} \\
 &=\max_{X \in \mathcal{S}^{|V|},x\in\R^{|V|}} \Big\{ \langle C, X \rangle:    \left(\begin{matrix} 1 & x\T\cr x& X\end{matrix}\right) \succeq 0,   \text{Tr}(X)=e\T x, \langle X, A_G \rangle = 0,  \langle ff\T,X \rangle =0,\langle \Diag(f),X\rangle =0  \Big\}.\label{eq:gbalonehatC}
 \end{align}
(Since each of the programs (\ref{eq:thetalasbalhat}), (\ref{eq:thetabalhat}), (\ref{eq:gbalonehat}) is strictly feasible, strong duality holds as claimed above.) We begin with comparing the above parameters and  show the analog of Proposition \ref{propcomparebal}.

\begin{proposition}\label{prop:hatcomparison}
For any bipartite graph~$G$, we have 
$${1\over 4}\lasbalhat(G) \leq {1\over 2}\sqrt{\gbalonehat(G) } \leq {1\over 4}\thetabalhat(G).$$ 
\end{proposition}

\proof{Proof.} 
We use the formulations  (\ref{eq:thetalasbalhatI}), (\ref{eq:thetabalhatJ}), (\ref{eq:gbalonehatC}) for the parameters $\lasbalhat(G)$, $\thetabalhat(G)$,  $\gbalonehat(G)$, respectively. Then the  inequalities follow in the same way as in the proof of Proposition~\ref{propcomparebal}.
\endproof

\smallskip
Next we compute the parameter $\thetabalhat(G)$ and show its relation to $\widehat h(G)$.
\begin{proposition}\label{prop:symbalhat}
Assume $G=(V_1 \cup V_2, E)$ is bipartite $r$-regular, set~$n:=|V_1|=|V_2|$ and let~$\lambda_2$ denote the second largest eigenvalue of~$A_G$. Then we have $\thetabalhat(G) = \frac{2n \lambda_2}{r+\lambda_2}=4\cdot \widehat h(G)$. 
\end{proposition}
\noindent
We delay the proof, which is a bit technical, to Appendix \ref{appendixproofsym}.
As the proof will show, the program (\ref{eq:thetabalhat}) defining $\thetabalhat(G)$ admits an optimal solution with $v=0$. Hence, when $G$ is bipartite regular, the constraint $\langle \Diag(f),X\rangle=0$ is redundant in program (\ref{eq:thetabalhat}) and one can set $v=0$ in program (\ref{eq:thetabalhat}), and
the same observation applies to the programs defining $\gbalonehat(G)$ and $\lasbalhat(G)$.

\medskip
We can now  compute the parameters $\lasbalhat(G)$ and $\gbalonehat(G)$ and show their relation to $\widehat h(G)$. 

\begin{proposition}\label{prop:symbalhat2}  
For any regular bipartite graph  $G$  we have   
$${1\over 4}\lasbalhat(G)={1\over 2}\sqrt{\gbalonehat(G)}={1\over 4}\thetabal(G)=\widehat h(G).$$
\end{proposition}

\proof{Proof.} 
Assume $G$ is bipartite regular and set $n:=|V_1|=|V_2|$. 
If $G$ is  complete bipartite, then $\alphabal(G)=0$ and, using (\ref{eq:thetabalhatJ}) and Proposition \ref{prop:hatcomparison}, one can check that $\thetabalhat(G)=0$, so the result holds.
We now assume that $G$ is not complete bipartite. In view of Propositions \ref{prop:hatcomparison} and \ref{prop:symbalhat} it suffices to show $\lasbalhat(G)\ge \thetabalhat(G)$.
 Assume that~$(\lambda, \mu, t,s,v)$ is feasible for the program (\ref{eq:thetalasbalhat}) defining $\lasbalhat(G)$, we construct a feasible solution for the program (\ref{eq:thetabalhat}) defining $\thetabalhat(G)$ with the same objective value $\lambda$.
 Call $Q\in \mathcal S^{1+|V_1|+|V_2|}$ the matrix appearing in program (\ref{eq:thetalasbalhat}).
 By  taking a Schur complement with respect to its upper left corner entry $\lambda$, we obtain 
 $$\lambda((\mu-1) I + t A_G + s ff\T+v\Diag(f))-\tfrac{\mu^2}{4} J \succeq 0.$$
We now claim that $\mu>1$. For this observe that the submatrices of $Q$ indexed by $V_1$ and $V_2$ read $(\mu-1)I_n +s J_n\pm vI_n$. Since they are both positive semidefinite this implies $(\mu-1)I_n+s J_n\succeq 0$ and thus $\mu\ge 1$.
Assume that $\mu=1$. Then the conditions $sJ_n\pm v I_n\succeq 0$ imply $v=0$. 
Let $i\in V_1$ and $j\in V_2$ that are not adjacent (they exist since $G\ne K_{n,n}$). Then the principal submatrix of $Q$ indexed by $\{0,i,j\}$ takes the form 
$\tiny \left(\begin{matrix}\lambda & -1/2 & -1/2 \cr -1/2 & s & -s \cr -1/2 & -s & s\end{matrix}\right)$ and it must be positive semidefinite, so we reach a contradiction.
Hence we have  $\mu>1$. Thus we can scale the above matrix and obtain 
\begin{align*}
\lambda I + \frac{\lambda t}{\mu-1} A_G + \frac{\lambda s}{\mu-1}ff\T +{\lambda v\over \mu-1}\Diag(f)- \frac{\mu^2}{4(\mu-1)} J \succeq 0.
\end{align*}
 Note that~$\tfrac{\mu^2}{4(\mu-1)} - 1 = \tfrac{(\mu-2)^2}{4(\mu-1)}\geq 0$ and add $(\tfrac{\mu^2}{4(\mu-1)} - 1) J \succeq 0$ to the above matrix. So we obtain  
 $$
\lambda I  + \frac{\lambda t}{\mu-1} A_G + \frac{\lambda s}{\mu-1}ff\T+{\lambda v\over \mu -1} \Diag(f) -J \succeq 0,
$$
which gives a feasible solution to the formulation~\eqref{eq:thetabalhat} of~$\thetabalhat(G)$ and thus shows $\thetabalhat(G) \leq \lambda = \lasbalhat(G)$. 
\endproof

\smallskip
\begin{remark}\label{remweaker}
{\em One idea for trying to get a stronger closed-form bound for $\alphabal(G)$ could be to
consider a possibly weaker symmetrization of 
the parameter $\lasbalone(G)$, where  we now allow a vector $u$ taking distinct values for nodes in $V_1$ and in $V_2$ instead of restricting to $u=\mu e$ for some $\mu\in \R$.
So we consider the following variation $\lasbaltilde(G)$ of the parameter 
$\lasbalhat(G)$,  defined by 
\begin{equation}\label{eq:lasbaltilde} 
\min_{\lambda,\mu_1,\mu_2, t,s,v\in\R}
\Big\{\lambda: \left(\begin{matrix} \lambda & -u\T/2\cr
-u/2& \Diag(u)-I+tA_G +s ff\T +v \Diag(f) \end{matrix}\right)\succeq 0,\ u=\mu_1\chi^{V_1}+\mu_2\chi^{V_2}\Big\}.
\end{equation}
By its definition, the parameter $\lasbaltilde (G)$ lower bounds $\lasbalhat(G)$, for which the optimization is restricted to the case $\mu_1=\mu_2$.
Nevertheless it turns out that the two parameters are in fact equal. 
To see this let us use the dual semidefinite program of (\ref{eq:lasbaltilde}), which reads
\begin{equation}\label{eq:lasbaltildeI}
\begin{split}
\lasbaltilde(G)= \max\Big\{ \langle I,X\rangle: 
& \left(\begin{matrix} 1 & x\T\cr x & X\end{matrix}\right)\succeq 0, 
\ \langle A_G,X\rangle =0,
\ \langle ff\T, X\rangle=0,\\
&\langle \Diag(f),X\rangle =0,\ \langle \Diag(\chi^{V_k}),X\rangle =x\T \chi^{V_k} \text{ for } k=1,2 
\Big\}.
\end{split}
\end{equation}
Assume $(x,X)$ is optimal for the program (\ref{eq:thetalasbalhatI}) defining $\lasbalhat(G)$. In order to show that $(x,X)$  is feasible for (\ref{eq:lasbaltildeI}) we only need to check that 
$\langle \Diag(\chi^{V_k}),X\rangle=x\T \chi^{V_k}$ for $k=1,2$. For this note that feasibility for (\ref{eq:thetalasbalhatI}) 
implies $x\T f=0$ and thus $x\T \chi^{V_1}=x\T \chi^{V_2}$. Moreover 
$\langle \Diag(f),X\rangle=0$ gives $\langle \Diag(\chi^{V_1}),X\rangle = \langle \Diag(\chi^{V_2}),X\rangle$ and $\text{\rm Tr}(X)=e\T x$ gives 
$\langle \Diag(\chi^{V_1}),X\rangle+\langle \Diag(\chi^{V_2}),X\rangle=
x\T \chi^{V_1}+x\T \chi^{V_2}$. Combining these facts we get the desired identities $\langle \Diag(\chi^{V_k}),X\rangle=x\T \chi^{V_k}$ for $k=1,2$. This shows   
$\lasbalhat(G)\le \lasbaltilde (G)$ and thus equality $\lasbalhat(G)= \lasbaltilde (G)$ holds.
}
\end{remark}

\section{Concluding remarks.}\label{secconcluding}

In this paper we investigate the parameters~$g(G)$, $h(G)$,  $\alphabal(G)$ (and other related parameters) dealing with (balanced) bipartite biindependent pairs in a bipartite graph $G$.  We show that deciding whether $\alphabal(G)=\alpha(G)$ is an NP-complete problem and that this implies NP-hardness of  the parameters $\alphabal(G), h(G), g(G)$. We offer a systematic study of the basic semidefinite bounds that are obtained at the first level of sums-of-squares (Lasserre) hierarchies. In particular, we introduce the semidefinite bounds $h_1(G), g_1(G) $ (for $g(G), h(G)$), and $\lasbalone(G)$, $\thetabal(G)$ (for $\alphabal(G)$). These semidefinite bounds can be seen as natural variations 
of the celebrated theta number $\vartheta(G)$
of Lov\'asz \cite{Lo79}, allowing a quadratic objective (for $h_1(G)$, $g_1(G)$) or adding a balancing constraint (for $\lasbalone(G)$, $\thetabal(G)$). However, while $\vartheta(G)=\alpha(G)$ when $G$ is bipartite, the parameters $h_1(G)$, $g_1(G)$, $\lasbalone(G)$, $\thetabal(G)$  give only upper bounds for the respective combinatorial graph parameters. An interesting fact is that  $h_1(G)$ in fact provides a better bound for $g(G)$ than $g_1(G)$ (recall Proposition \ref{propghg1h1}). Another interesting fact is that $\lasbalone(G)\le \thetabal(G)$ and that the inequality may be strict, while the unbalanced analogs both coincide with $\vartheta(G)$ (recall Proposition~\ref{propcomparebal} and relation (\ref{eqtheta=})). We also show that deciding whether $h(G)=h_1(G)$ is an NP-hard problem. An object of further study will be to investigate the numerical behaviour of the various bounds introduced in this paper.

When $G$ is an $r$-regular bipartite graph, we give closed-form eigenvalue-based bounds that are obtained by restricting to symmetric solutions in the definitions of $h_1(G),$ $g_1(G)$, $\lasbalone(G)$, and $\thetabal(G)$. In this way we obtain the parameter $\widehat h(G)={n\over 2}{\lambda_2\over r+\lambda_2}$, where $\lambda_2$ is the second largest eigenvalue of $A_G$ and $G$ has $n$ vertices on each side of its bipartition.  Then $h(G)\le h_1(G)\le \widehat h(G)$ holds  
and it turns out that $\widehat h(G)$ provides a better bound for $g(G)$ than its corresponding eigenvalue bound $\widehat g(G)$. Moreover, only edge-transitivity is required to show equality $h_1(G)=\widehat h(G)$, while one needs vertex- and edge-transitivity to show $g_1(G)=\widehat g(G)$. 
Bipartite regular graphs that are edge-transitive but not vertex-transitive are known as {semi-symmetric} graphs; the smallest such graph, constructed by Folkman \cite{Folkman1967},  is 4-regular with 20 vertices.  
We  show that the natural eigenvalue bounds corresponding to the various semidefinite relaxations of    $\alphabal(G)$ all coincide (up to simple transformation) with the parameter $\widehat h(G)$, and that the same holds for a natural strengthening of $h_1(G)$ (recall Proposition \ref{prophhp}). Hence, finding a stronger  closed-form bound for $\alphabal(G)$ that is able to take  advantage of the restriction to balanced independent sets remains an open problem.

We have considered the parameter $\widehat h(G)$ (in relation (\ref{eqh1hat})) and gave a closed-form expression for it in the case when $G$ is a $r$-regular bipartite graph.
More generally, one can consider the case when $G$ is a bipartite $(r_1,r_2)$-regular graph, which means that every vertex in $V_1$ has degree $r_1$ and every vertex in $V_2$ has degree $r_2$.
Then, along the same lines as for Proposition \ref{prop:hbbioneedgetransitive}, one can also compute a closed-form expression for $\widehat h(G)$. It will be interesting to compute this parameter for the graphs $G^n_{k,\ell}$ and $G^{n,q}_{k,\ell}$, that are related to cross-intersecting subset and subspace families as mentioned in the introduction, and to check their relationship with the bounds by Pyber \cite{Pyber1986} and Suda and Tanaka~\cite{ST2014}.

So, we see in this paper an application of the second largest eigenvalue $\lambda_2$ to the study of parameters involving (balanced) independent sets in bipartite graphs.
The second largest eigenvalue $\lambda_2$ has been widely studied and has well-known applications to various graph properties. For instance, there is a classical upper bound on $\lambda_2$ for any $r$-regular graph in terms of $r$ and its diameter~\cite{Nil91}, and large $r$-regular graphs with small second eigenvalue are shown to be Hamiltonian~\cite{KS03}. A notable application of $\lambda_2$ is for bounding the edge expansion (or isoperimetric number), sometimes denoted $h_G$, and defined as the minimum value of $E(S,V\setminus S)/|S|$ taken over all $S\subseteq V$ with $1\le |S|\le |V|/2$.  Namely, if $G$ is $r$-regular, then  $(r-\lambda_2)/2\le h_G\le \sqrt{r^2-\lambda_2^2}$ (see \cite{Mohar}). We refer, e.g., to \cite{Cvetkovic,BH17} and further references therein for more information.

Among other examples we have considered the hypercube $G=Q_r$  on $\{0,1\}^r$. We show that $\alphabal(Q_r)\ge a(r-1)$ for all $r\ge 1$, where $a(r)$ is as defined in (\ref{eqar}). Computational experiments suggest that this is the exact value. Showing $\alphabal(Q_r)=a(r-1)$ for all $r$ is an interesting open problem that would offer a new link from balanced biindependent sets to  other combinatorial counting problems such as the number of $r$-steps random walks on a line starting from the origin and returning to it at least once.

%
%
%

\appendix

\section{Application to product-free sets in finite groups.}\label{appendixgroup}

Let $\Gamma$ be a finite group. A subset $A\subseteq \Gamma$ is called {\em product-free} if $uv\not\in A$ for all $u,v\in A$.
A problem of interest is to find the maximum cardinality of a product-free set in $\Gamma$; see, e.g.,  Kedlaya \cite{Kedlaya}, Gowers \cite{Gowers} for  background and an overview of results on this problem.
As in  \cite{Kedlaya} let $\beta(\Gamma)$ denote the maximum density $|A|/|\Gamma|$ of a product-free set  $A\subseteq \Gamma$. Clearly, $\beta(\Gamma)\le 1/2$ (since, for any $x\in A$,  the sets $A$ and $xA$ are disjoint subsets of $\Gamma$). It is known that any finite abelian group satisfies 
$1/7\le \beta(\Gamma)\le 1/2$. Moreover any finite group satisfies $\beta(\Gamma)=\Omega(1/n^{3/14})$   
and the question arose whether $\beta(\Gamma)=\Omega(1/n^\epsilon)$ for all $\epsilon>0$.
Gowers \cite{Gowers} answered in the negative by showing that $\beta(\text{PSL}_2(q))=O(1/n^{1/9})$
(see Example \ref{exampleGowers}  below).

As a crucial ingredient in his proof (which applies in fact to a more general setting) Gowers \cite{Gowers} introduces an upper bound on the product-free set density of $\Gamma$  in terms of the second eigenvalue of an associated bipartite Cayley  graph. We follow the exposition by Kedlaya \cite{Kedlaya} and Vallentin~\cite{Vallentin}, which relies on using (a variation of) the parameter $\widehat h$ applied to this bipartite Cayley graph.

Let us fix a product-free set $A\subseteq \Gamma$ and define the bipartite Cayley graph $G_{\Gamma,A}=(V_1\cup V_2,E)$, where $V_1,V_2$ are two disjoint copies of $\Gamma$, where $u\in V_1$, $v\in V_2$ are adjacent in $G_{\Gamma,A}$ if  $uv\in A$.    Note that the graph $G_{\Gamma,A}$ is $|A|$-regular. Let $A_k$ denote the copy of $A$ within the set $V_k$ for $k=1,2$.
Then, by construction, $(A_1,A_2)$ is a biindependent pair in $G_{\Gamma,A}$ since $A$ is product-free.
 
 The next result relates the size of $|A|$ to the second largest eigenvalue of the adjacency matrix of $G_{\Gamma,A}$. 
 It is essentially based on  \cite[Lemma 3.2]{Gowers},   \cite[Lemma 5.3]{Kedlaya}
  (and Vallentin's presentation~\cite{Vallentin}).
 
\begin{lemma}
\label{lemma:boundlambda2}
Let $\Gamma$ be a finite group, $n:=|\Gamma|$, and let $k$ denote the minimum dimension of a non-trivial  representation of  $\Gamma$. Let $A\subseteq \Gamma$ be a product-free set and let $\lambda_2$ denote the second largest eigenvalue of the adjacency matrix of the bipartite Cayley graph $G_{\Gamma,A}$.  Then we have
\begin{align} \label{eqGamma}
\lambda_2 \leq \sqrt{|A|(n-|A|)\over k}. 
\end{align}
\end{lemma}

\proof{Proof.} 
Set $G:=G_{\Gamma,A}$ and write its adjacency matrix as in (\ref{eqAMG}).
By Lemma \ref{lembipeig}, $\lambda_2^2$ is the second largest eigenvalue of $M_GM_G\T$;  let $k_2$ denote its multiplicity. 
Since $G$ is $|A|$-regular, $G$ has $n|A|$ edges and thus $\text{Tr}(M_GM_G\T)= n|A|$.
On the other hand, by considering the spectral decomposition of $M_GM_G\T$, we obtain 
$\text{Tr}(M_GM_G\T) \ge |A|^2 + \lambda_2^2 k_2$. By combining both facts we deduce 
$n|A|\ge  |A|^2+ \lambda_2^2 k_2$ and thus $\lambda_2\le \sqrt{|A|(n-|A|)/k_2}$.

We now show that $k_2\ge k$, which, combined with the above inequality for $\lambda_2$, gives the desired inequality (\ref{eqGamma}). For this let $W$ denote the eigenspace of $M_GM_G\T$ corresponding to the eigenvalue $\lambda_2^2$, so that $W$ has dimension $k_2$. One can easily check that
$$W=\{x\in \R^{|V_1|}: x\T e=0,\ x\T M_GM_G\T x= \lambda_2^2 \|x\|^2\}.$$
We show that $W$ is invariant under some non-trivial action of $\Gamma$. For this, consider the action of $\Gamma$ on the space $\R^{|V_1|}$ defined by right multiplication; that is, 
for $\gamma\in \Gamma$ and $x=(x_{u})_{u\in \Gamma}$, define $x^\gamma:=(x_{u\gamma})_{u\in \Gamma}$.
We claim that $(M_G\T x^\gamma )_v=(M_G\T x)_{\gamma^{-1}v}$ for any $v\in\Gamma$. Indeed,
$$(M_G\T x^\gamma )_v
=\sum_{u\in\Gamma} M_G(u,v)x^\gamma_u=\sum_{u\in\Gamma: uv\in A} x_{u\gamma}
=\sum_{w\in \Gamma: w\gamma^{-1}v\in A} x_w= \sum_{w\in\Gamma}M_G(w,\gamma^{-1}v)x_w=
(M_G\T x)_{\gamma^{-1}v}.
$$
From this follows that $x\T M_GM_G\T x = (x^\gamma)\T M_GM_G\T x^\gamma$ and $e\T x=e\T x^\gamma$. Hence $x\in W$ implies $x^\gamma\in W$ and thus the space $W$ is invariant under this action of $\Gamma$. This action  is non-trivial since a nonzero vector $x\in W$ is not a multiple of the all-ones vector and thus $x^\gamma\ne x$ for some $\gamma\in \Gamma$.
Therefore, we can conclude that $k_2=\dim W\ge k$ and the proof is complete.
\endproof

\smallskip
We can now show the following  bound on the product-free set density,  which is essentially Theorem 3.3 of Gowers \cite{Gowers} (see Remark \ref{remGowers}  below).

\begin{theorem} 
\label{theorem:Gowers}
Let $\Gamma$ be a finite group and let $k$ denote the minimum dimension of a non-trivial representation of $\Gamma$. If $A$ is  a product-free set in  $\Gamma$, then we have 
$|A|\le {|\Gamma|\over 1+k^{1/3}}$. 
\end{theorem}

\proof{Proof.} 
Since $(A_1,A_2)$ is a bipartite biindependent pair in~$G_{\Gamma,A}$,  we have
$
\frac{|A|}{2}  \leq h(G_{\Gamma,A})$ and thus 
${|A|\over 2}\le  \widehat h(G_{\Gamma,A})  
= \frac{n}{2} {\lambda_2\over |A|+\lambda_2}$, which implies $|A|^2\le \lambda_2(n-|A|) \le (n-|A|) 
\sqrt{|A|(n-|A|)/k}$, using (\ref{eqGamma}).
This implies $\big({|A|\over n-|A|}\big)^{3/2}\le {1\over k^{1/2}}$ and thus $|A|\le {n\over 1+k^{1/3}}$ as desired. 
\endproof 

\smallskip
\begin{remark}\label{remGowers} 
{\em The upper estimates in (\ref{eqGamma}) and Theorem \ref{theorem:Gowers} offer  a slight sharpening of the known results.
Indeed Gowers shows $\lambda_2:=\lambda_2(A_{G_{\Gamma,A}}) \le \sqrt{n|A|/k}$ (Lemma 3.2 in \cite{Gowers}), a bound that is a bit weaker than the one in (\ref{eqGamma}), which he then uses to show $|A|\le {n\over k^{1/3}}$ (Theorem 3.3 in \cite{Gowers}). Vallentin \cite{Vallentin} uses his eigenvalue bound to conclude ${|A|\over 2}\le h(G_{\Gamma,A})\le {n \over |A|} \lambda_2\le {n\over |A|} \sqrt{n|A|/k}$, and thus   $|A|\le 2^{2/3}{n\over k^{1/3}}$. Our slightly sharper estimate $|A|\le {n\over 1+k^{1/3}}$ follows using the sharper bound in (\ref{eqGamma}) and the sharper eigenvalue bound $h(G_{\Gamma,A})\le \widehat h(G_{\Gamma,A})={n\over 2} {\lambda_2\over |A|+\lambda_2}$.

One recovers the known bound $\beta(\Gamma)\le 1/2$ using Theorem~\ref{theorem:Gowers}. 
This bound is tight, for instance, when $\Gamma$ is 
the symmetric group $S_n$ (in which case  $k=1$, since the sign representation is a non-trivial representation of dimension 1).  
Since the set $S_n\setminus A_n$  consisting of all permutations  with an odd sign is product-free with size $n!/2$, one gets $\beta(S_n)\ge 1/2$ and thus the bound is tight: $\beta(S_n)=1/2$. 
By contrast it has been a long standing open problem to determine the product-free density of the alternating group $A_n$; it was shown recently in \cite{Keevash22} that $\beta(A_n)=\Theta(1/\sqrt n)$.}
\end{remark}

\smallskip
\begin{example} [Gowers~\cite{Gowers}]\label{exampleGowers}
Consider the group~$\Gamma = \text{PSL}_2(q)$, which is the group of all~$2 \times 2$-matrices over~$\F_q$ with determinant~$1$, quotiented by the subgroup~$\{I,-I\}$. As Gowers notes, it is one of the simplest infinite families of finite simple groups (i.e., nontrivial groups whose only normal subgroups are the trivial group and the group itself). It is natural to consider \emph{simple} finite groups, because any product-free subset in a quotient of a finite group lifts to a product-free subset in the group itself. 

The order of  $\text{PSL}_2(q)$ is $n= q(q^2-1)/2$.  Frobenius proved that every non-trivial representation of~$\text{PSL}_2(q)$ has dimension at least~$k=(q-1)/2$, which is at least $n^{1/3}/{4}$. Applying Theorem~\ref{theorem:Gowers} one obtains  that the maximum size of a product-free subset in~$\Gamma$ is at most ${4^{1/3}n^{8/9}}$ and thus $\beta(\text{PSL}_2(q)) =O(1/n^{1/9})$.
\end{example}

\section{Proof of Lemma \ref{lemXx}.}\label{appendix-Xx}
We use the fact that ${\tiny \left(\begin{matrix}1 & x\T\cr x& X\end{matrix}\right)}\succeq 0$ $\Longleftrightarrow X-xx\T\succeq 0$.

\noindent
The ``only if" part in Lemma \ref{lemXx} is easy: if $X-xx\T \succeq 0$ and $e\T x=1$, then $\langle J,X\rangle \ge e\T xx\T e=1$. We now show the ``if part". So, assume $X\in \mathcal S^n$ satisfies $X\succeq 0$, $\text{\rm Tr}(X)=1$ and $\langle J,X\rangle \ge 1$; we construct $x\in \R^n$ such that $e\T x=1$ and $X-xx\T\succeq 0$. For this, consider the spectral decomposition $X=\sum_{i=1}^n \beta_i u_iu_i\T$, where the $u_i$'s form an orthonormal basis of eigenvectors, $\beta_i\ge 0$, and $\sum_{i=1}^n \beta_i=\text{Tr}(X)=1$. Define the vectors 
$$a:= (\sqrt{\beta_i} \cdot e\T u_i)_{i=1}^n\quad \text{ and } \quad x:=\sum_{i=1}^n {\beta_i \cdot e\T u_i\over \|a\|^2}u_i.
$$
Then, we have  $\|a\|^2 = \sum_{i=1}^n \beta_i (e\T u_i)^2 =\langle J,X\rangle \ge 1$ and
$e\T x= 1$. We now show $X-xx\T\succeq 0$. For this let $z\in\R^n$ be any vector; we show that $z\T(X-xx\T)z\ge 0$. Indeed we have
\begin{align*}
z\T(X-xx\T)z& = z\T Xz-(z\T x)^2=\sum_{i=1}^n \beta_i (z\T u_i)^2 - {1\over \|a\|^4}\Big(\sum_{i=1}^n \beta_i \cdot e\T u_i \cdot z\T u_i\Big)^2\\
& = \sum_{i=1}^n \beta_i (z\T u_i)^2 -{1\over \|a\|^4}\Big(\sum_{i=1}^n \sqrt{\beta_i}(e\T u_i) \cdot 
\sqrt{\beta_i}(z\T u_i)\Big)^2\\
& \ge  \sum_{i=1}^n \beta_i (z\T u_i)^2 -{1\over \|a\|^4} \Big(\sum_{i=1}^n \beta_i (e\T u_i)^2\Big)
\Big(\sum_{i=1}^n \beta_i (z\T u_i)^2\Big)\\
&=  \sum_{i=1}^n \beta_i (z\T u_i)^2\Big(1- {1\over \|a\|^2}\Big)\ge 0,
\end{align*} 
using Cauchy-Schwartz inequality for the first inequality and $\|a\|\ge 1$ for the last one. \hfill\qed

\section{Proof of Proposition \ref{prop:gbbionesymmetric}.}\label{appendixghat}

Here we show the result of Proposition \ref{prop:gbbionesymmetric}.
As starting point we use the formulation of $\gbbione(G)$ from (\ref{gbbione:primal0}), where we restrict the optimization to matrices $Z$ of the form $Z=tA_G$ for some scalar $t\in \R$, and to vectors~$u$ of the form~$u=\mu e$ for some~$\mu \in \R$.  Note that when $G$ is vertex- and edge-transitive  this restriction can be made without loss of generality. Then we consider the equivalent reformulation obtained by taking the Schur complement with respect to the upper left corner $\lambda$ of the matrix in (\ref{gbbione:primal0}). 
So we aim to compute the optimum value of the program 
\begin{equation}\label{eqgonehat}
\widehat {g}(G):=\min_{\lambda, \mu ,t\in \R} \Big\{ \lambda \,\, | \,\,\lambda( \mu I -C + tA_G) - \tfrac{\mu^2}{4}J \succeq 0,\ \lambda\ge 0  \Big\},
\end{equation}
which upper bounds $\gbbione(G)$ and is equal to it when $G$ is vertex- and edge-transitive; we will show that this optimum value has the form claimed in Proposition \ref{prop:gbbionesymmetric}.
For this we need to express the condition that the eigenvalues of the matrix $\lambda(\mu I -C + tA_G) - \tfrac{\mu^2}{4}J$ are nonnegative.
By considering the eigenvalue of this matrix for the all-ones vector we get the condition 
\begin{equation}\label{eqeig1}
\lambda(\mu -{n\over 2}+tr) -{\mu^2\over 2}n\ge 0.
\end{equation}
In addition to this  we need to ensure that $ \mu I -C + tA_G\succeq 0$. Note that the matrix $A:=tA_G-C$ has the block-form (\ref{eqM}), with $M:=  tM_G-{1\over 2}J_n$. We have $MM\T= t^2 M_GM_G\T +(n/4 -tr)J_n$. Hence the eigenvalue of $MM\T$ at the all-ones vector is equal to $t^2r^2 +n(n/4-tr)=(tr-n/2)^2$ and its second largest eigenvalue is $t^2\lambda_2(M_GM_G\T)=t^2 \lambda_2^2$.
Therefore, we obtain that 
\begin{equation}\label{eqeig2}
 \mu I -C + tA_G\succeq 0\Longleftrightarrow \mu\ge |tr-n/2|\ \text{ and } \ \mu \ge  |t\lambda_2|=|t|\lambda_2.
\end{equation}
Again we assume $G$ is not complete bipartite and thus $\lambda>0$. Then it follows from (\ref{eqeig1})  that $\mu -{n\over 2}+tr\ge 0$ and, combined with (\ref{eqeig2}), we must have $\mu -{n\over 2}+tr>0.$
Set 
\begin{equation}\label{eqmut}
\mu(t):=\max\{|tr-n/2|,|t|\lambda_2\}.
\end{equation}
Then we can conclude that $\widehat {g_1}(G)$ can be reformulated as 
\begin{equation}\label{eqg1wh}
\widehat {g}(G)= \min_{\mu, t\in\R} \Big\{ F(\mu):={n\over 2}{\mu^2\over \mu +tr-n/2}: 
\mu+tr-n/2>0,\ \mu\ge \mu(t)\Big\}.
\end{equation}
Our task is now to compute the minimum value of the above program (\ref{eqg1wh}).
It is useful to see the behaviour of the function $F(\mu)$. For this observe that its derivative is 
$F'(\mu)= {n\over 2}{\mu^2-\mu(n-2tr)\over (\mu+tr-n/2)^2}$. Hence $F'(\mu)\le 0$ (and thus $F(\mu)$ is monotone nonincreasing) when $\mu$ lies between $0$ and $n-2tr$, and $F'(\mu) \ge 0$ (and thus $F(\mu)$ is monotone nondecreasing) when $\mu$ lies outside the interval $[0,n-2tr]$ or $[n-2tr,0]$ (depending on the sign of $n-2tr$). Note also that $F(\mu)$ has a vertical asymptote at $\mu=n/2-tr$ (at which its denominator vanishes).

According to (\ref{eqg1wh}) we need to discuss according to the value of $\mu(t)$ in (\ref{eqmut}). So we partition the range of values taken by $t$ into $\R=T_1\cup T_2 \cup T_3$, where we set
$$T_1:=\{t\in\R: tr-n/2\ge 0\},\quad T_2:= \{t\in \R: tr-n/2<0,\ t\ge 0\}, \quad T_3:=\{t\in \R: t<0\}.$$
Then, for $\ell\in\{1,2,3\}$ and for $t\in T_\ell$, set
$$F_\ell(t):= \min_\mu\{F(\mu): \mu+tr-n/2>0,\ \mu\ge \mu(t)\Big\},
$$
so that  we have 
\begin{equation}\label{eqg2}
 \widehat {g}(G)= \min_{\ell\in\{1,2,3\}} \min _{t\in T_\ell}F_\ell(t).
\end{equation}
We thus need to compute the value  of $\min_{t\in T_\ell} F_\ell(t)$ for each $\ell=1,2,3$. 
So we distinguish the three cases $\ell=1,2,3$.

\medskip\noindent
{\bf Case 1: $\ell=1$.} Assume $t\in T_1$. Then, $t>0$ and $\mu(t)= \max\{tr-n/2, t\lambda_2\}$.
Then we have $$F_1(t)=\min_\mu \{F(\mu): \mu\ge \mu(t)\}= F(\mu(t)),$$
where the last equality follows since  the function $F(\mu)$ is monotone nondecreasing on $[0,\infty)$. We have two cases.
\begin{itemize}
\item Either $tr-n/2\ge t\lambda_2$, which implies $\mu(t)= tr-n/2$ and thus
$F_1(t)= F(tr-n/2) ={n\over 4}(tr-n/2)$. Note that in this case we have $r>\lambda_2$. 
Then we obtain
\begin{equation}\label{val1a}
\min_{t\in T_1} \{F_1(t): tr-n/2\ge t\lambda_2\}=\min\Big\{{n\over 4}(tr-n/2): t\ge {n\over 2(r-\lambda_2)}\Big\}=
 {n^2\over 8} {\lambda_2\over r-\lambda_2}.
\end{equation}
\item
Or $tr-n/2\le t\lambda_2$, so that $t\le {n\over 2(r-\lambda_2)}$ if $\lambda_2<r$, and  $\mu(t)=t\lambda_2$. Then we have
\begin{equation*}
\min_{t\in T_1} \{F_1(t): tr-n/2\le t\lambda_2\}
= \min\Big\{ F(t\lambda_2)={n\over 2}{t^2\lambda_2^2\over t(\lambda_2+r)-n/2}: {n\over 2r}\le t\le {n\over 2(r-\lambda_2)}\Big\},
\end{equation*}
setting ${n\over 2(r-\lambda_2)}=\infty$ if $r=\lambda_2$.
Consider the function  $\psi(t):=F(t\lambda_2)$, whose derivative is 
$\psi'(t)={n\lambda_2^2\over 2}{t(t(\lambda_2+r)-n)\over (t(\lambda_2+r)-n/2)^2}$.
Note that ${n\over 2(\lambda_2+r)}\le {n\over 2r}\le {n\over \lambda_2+r}$, where ${n\over 2(\lambda_2+r)}$ is an aymptote of $\psi(t)$ (as it is a zero of its denominator).
We also need to compare the relative positions of ${n\over \lambda_2+r}$ (zero of $\psi'(t)$) and $n\over 2(r-\lambda_2)$ (upper bound of the range for $t$); note that  ${n\over \lambda_2+r}\le {n\over 2(r-\lambda_2)} $ if and only if $r\le 3\lambda_2$.
We can now compute the minimum value taken by the function  $\psi(t)$ for ${n\over 2r}\le t \le {n\over 2(r-\lambda_2)}$. When $r\le 3\lambda_2$ it is attained at 
${n\over \lambda_2+r}$ with value $\psi({n\over \lambda_2+r})= {n^2\lambda_2^2\over (\lambda_2+r)^2}$ and when 
$r\ge 3\lambda_2$ (so that $\lambda_2<r$) it is attained at $n\over 2(r-\lambda_2)$ with value 
$\psi({n\over 2(r-\lambda_2)}) = {n^2\lambda_2\over 8(r-\lambda_2)}$.
In summary we have shown that 
\begin{align}
\min_{t\in T_1} \{F_1(t): tr-n/2\le t\lambda_2\} & = {n^2\lambda_2^2\over (\lambda_2+r)^2} & \text{ if } r\le 3\lambda_2,\label{val1b1}\\
&= {n^2\lambda_2\over 8(r-\lambda_2)} & \text{ if } r\ge 3\lambda_2.\label{val1b2}
\end{align}
\end{itemize}
We can now compute $\min_{t\in T_1}F_1(t)$ by comparing (\ref{val1a}) and (\ref{val1b1}), (\ref{val1b2}).
We obtain:
\begin{align}
\min_{t\in T_1}F_1(t) & = \min\Big\{ {n^2\over 8} {\lambda_2\over r-\lambda_2},  {n^2\lambda_2^2\over (\lambda_2+r)^2}\Big\}= 
 {n^2\lambda_2^2\over (\lambda_2+r)^2} & \text{ if } r\le 3\lambda_2,\label{val1c1}\\
\min_{t\in T_1}F_1(t) & =  {n^2\over 8} {\lambda_2\over r-\lambda_2}
& \text{ if } r\ge 3\lambda_2.\label{val1c2}\end{align}

\medskip\noindent
{\bf Case 2: $\ell=2$.} Assume $t\in T_2$, then $tr-n/2\le 0$ and $t\ge 0$. In this case $\mu(t)=\max\{n/2-tr,t\lambda_2\}$.
We now have $0\le n/2-tr\le \mu(t)$. Moreover, one can verify that
\begin{align}
F_2(t)=\min_{\mu\ge \mu(t)} F(\mu) & = F(n-2tr) &\text{ if } \mu(t)\le n-2tr,\\
& = F(\mu(t))\ge F(n-2tr) & \text{ if } \mu(t) \ge n-2tr.
\end{align}
Hence, the minimum value of $F_2(t)$ for $t\in T_2$ is equal to $F(n-2tr)=n(n-2tr)$, which is obtained when $\mu(t)\le n-2tr$. 
We now proceed to compute the minimum value taken by  $F(n-2tr)$ for $t\in T_2$ and $\mu(t)\le n-2tr$. For this we distinguish two cases depending on the value of $\mu(t)$.
\begin{itemize}
\item Either $n/2-tr \ge t\lambda_2$, i.e., $t\le {n\over 2(r+\lambda_2)}$ and thus $\mu(t)=n/2-tr \le n-2tr$. Then we have
\begin{align}
\min_{t\in T_2} \Big\{F_2(t):  t\le  {n\over 2(r+\lambda_2)}\Big\}= \min\Big\{ n(n-2tr): 0\le t\le {n\over 2(r+\lambda_2)}\Big\}= {n^2\lambda_2\over r+\lambda_2}.\label{val2a}
\end{align}
\item Or $n/2-tr \le t\lambda_2$, i.e., $t\ge {n\over 2(r+\lambda_2)}$ and thus $\mu(t)=t\lambda_2$.
Then $\mu(t)=t\lambda_2 \le n-2tr$ is equivalent to $t\le {n\over \lambda_2+2r}$ ($\le {n\over 2r}$).
Then we have
\begin{align}
\min_{t\in T_2} \Big\{F_2(t):  {n\over 2(r+\lambda_2)}\le t\le {n\over \lambda_2+2r}\Big\} & =\min\Big\{n(n-2tr): {n\over 2(r+\lambda_2)}\le t\le {n\over \lambda_2+2r} \Big\}\\
&= {n^2\lambda_2\over \lambda_2+2r}.\label{val2b}
\end{align}
\end{itemize}
Comparing the values in (\ref{val2a}) and (\ref{val2b}) we obtain that 
\begin{align}
\min_{t\in T_2}F_2(t)= {n^2\lambda_2\over \lambda_2+2r}.\label{val2c}
\end{align}

\medskip
\noindent
{\bf Case 3: $\ell=3$.} Assume $t\in T_3$, i.e., $t<0$, and thus $tr-n/2<0$ and 
$\mu(t)=\max\{n/2-tr, -t\lambda_2\}$. Then $F_3(t)=\min_{\mu\ge \mu(t)} F(\mu)$ (since $\mu+tr-n/2>0$).
If $-t\lambda_2\le n/2-tr$ then $\mu(t)=n/2-tr$ and we find that $F_3(t)=F(n-2tr)$. Else, if  $-t\lambda_2\le n/2-tr$ 
then $\mu(t)=-t\lambda_2\le n-2tr$ and we have again $F_3(t)=F(n-2tr)$.
Hence $F_3(t)=F(n-2tr)=n(n-2tr)$ for all $t\in T_3$. 
Then we have
\begin{align}\min_{t\in T_3} F_3(t)=\min\{n(n-2tr): t<0\}= n^2.\label{val3}
\end{align} 

\medskip\noindent
We can now finally compute the value of $\widehat g(G)$ as defined in  (\ref{eqg2}) based on relations (\ref{val1c1})-(\ref{val1c2}),  (\ref{val2c}) and (\ref{val3}).
 Note that $n^2\ge {n^2 \lambda_2 \over \lambda_2+2r}$, ${n^2\lambda_2^2\over (\lambda_2+r)^2}\le {n^2\lambda_2\over \lambda_2+2r}$, and $\frac{n^2 \lambda_2}{\lambda_2+2r} \geq \frac{n^2}{8} \frac{\lambda_2}{r-\lambda_2}$ if $r \geq 3 \lambda_2$. Based on this we obtain
$$
\widehat g(G) 
=\begin{cases}
 \frac{n^2\lambda_2^2}{(\lambda_2+r)^2}  &  \text{ if $r\le 3\lambda_2$},\\
 \frac{n^2\lambda_2}{8(r-\lambda_2)}  &  \text{ if $r\ge 3\lambda_2$,}
\end{cases}
$$
which is the desired result. \hfill\qed

\section{Proof of Proposition \ref{prop:symbalhat}.}\label{appendixproofsym}

We give here the proof of Proposition \ref{prop:symbalhat}. For this let $P:= \lambda I -J +tA_G +v\Diag(f) +sff\T$ denote the matrix appearing in program (\ref{eq:thetabalhat}). 
We need to find the smallest $\lambda$ for which there exist $t,v,s\in\R$ such that $P\succeq 0$. Note that $A_G e = re$,  $(\Diag f)e = f$, $ff\T e =0$, $Je = 2ne$, $A_Gf=-rf$, $(\Diag f)f =e$, $ff\T f =2nf$, and $Jf=0$. Hence~$P$ leaves the subspaces $\langle e,f \rangle$ and $\langle e,f \rangle^\perp$ invariant. Let~$u$ be an eigenvector of~$P$ for eigenvalue~$\tau$, and write~$u=x+y$ with~$x \in \langle e,f\rangle$ and~$y \in \langle e,f \rangle^{\perp}$. Then $Px +Py = \tau x + \tau y$, so~$Px - \tau x = \tau y - Py$. The left-hand side is contained in~$\langle e,f\rangle$, while the right-hand side is contained in~$\langle e,f \rangle^{\perp}$, so both sides of the equality are~$0$. So~$\tau$ is an eigenvalue corresponding to~$x$ (if~$x\neq 0$) and also corresponding to~$y$ (if~$y\neq 0$). Hence
$$
P \succeq 0 \,\,\,\, \Longleftrightarrow \,\,\,\,  \begin{cases}
    x\T P x \geq 0   & \text{for all } x \in \langle e,f\rangle,  \\
    y\T P y \geq 0  & \text{for all }  y \in \langle e,f\rangle^\perp.
  \end{cases}
$$
We now characterize when $x\T P x \geq 0$ for all $x \in \langle e,f\rangle$,  and when $y\T P y \geq 0$ for all $y \in \langle e,f\rangle^\perp$.
\begin{description}
\item[(i)] Let $x \in \langle e,f\rangle$ and write~$x = ae+bf$ with~$a,b \in \R$. Then 
\begin{align*}
Px &= a(\lambda e +tre +vf -2ne) +b(\lambda f - trf +ve +2nsf) 
\\ &= \left( a (\lambda +tr-2n) +bv\right) e + \left(av+b(\lambda-tr+2ns)\right) f,
\end{align*}
so $x\T P x =  2na\left( a (\lambda +tr-2n) +bv\right) e + 2nb\left(av+b(\lambda-tr+2ns)\right) f $. Hence
\begin{align}
  x\T P x \geq 0  \,\,\, \forall \, x \in \langle e,f\rangle &\Longleftrightarrow   a^2 (\lambda +tr-2n) +2abv+b^2(\lambda-tr+2ns) \geq 0 \,\,\, \forall a,b \in \R \notag
\\  
&\Longleftrightarrow \begin{cases} v^2 \leq (\lambda+tr-2n)(\lambda-tr+2ns), & \\
  \lambda+tr-2n\geq 0.&  \end{cases}
\end{align}
Here the first equivalence follows by rearranging terms in $x\T Px$, and  the second one  by considering the expression in~$a$ and~$b$ as a quadratic equation in~$a$ and computing the discriminant. 

\item[(ii)]  Assume that~$y=( c \,\,d)\T \in \langle e,f \rangle^\perp$ is an eigenvector of~$P$ for eigenvalue~$\tau$, where $c,d\in\R^n$. Then, $c\T e=d\T e=0$. Using the block-form of $P$ we obtain
$$
Py = \begin{pmatrix} \lambda c + t M_G d +vc \\ \lambda d +tM_G\T c - vd \end{pmatrix}= \tau  \begin{pmatrix}c \\d\end{pmatrix}.
$$
So~$(\tau -\lambda-v)c = t M_G d$ and $(\tau-\lambda+v)d= t M_G\T c$. It follows that 
$t^2 M_G\T M_G d =  (\tau -\lambda-v)tM_G\T c =(\tau-\lambda-v)(\tau-\lambda+v)d = ((\tau-\lambda)^2-v^2)d$. Similarly,~$t^2 M_G M_G\T c = ((\tau-\lambda)^2-v^2)c$. As $c \neq 0$ or $d \neq 0$, we have that $\frac{(\tau-\lambda)^2-v^2}{t^2}$ is an eigenvalue of~$M_G\T M_G$ (if $t\ne 0$), which is distinct from its eigenvalue~$r^2$ for eigenvector~$e$, as~$e\T c=e\T d=0$. So
\begin{align*}
&(\tau-\lambda)^2-v^2 = t^2 \lambda_i(M_G\T M_G) \,\,\,\,\, (i \geq 2),\\
\text{ and thus }\,\,\,\,&\tau = \lambda \pm \sqrt{v^2+t^2\lambda_i (M_G\T M_G)} \,\,\,\,\, (i \geq 2).\end{align*}
We need to ensure~$\tau \geq 0$. Hence we obtain the condition
\begin{align*}
\lambda \geq \sqrt{v^2 + t^2 \lambda_2 (M_G\T M_G)}= \sqrt{v^2 + t^2 \lambda_2^2}.
\end{align*}
Note this also holds if $t=0$.
\end{description}
Summarizing, we obtain that~$\thetabalhat(G)$ is the smallest~$\lambda$ such that there exist~$t,s,v \in \R$ satisfying  
$$
\begin{cases} v^2 \leq (\lambda+tr-2n)(\lambda-tr+2ns), & \\
  \lambda+tr-2n\geq 0, & 
  \\  \lambda \geq \sqrt{v^2 + t^2 \lambda_2^2}. \end{cases}
$$
Without loss of generality we may set~$v=0$, since if~$(\lambda,t,s,v)$ is feasible, then also~$(\lambda,t,s,v=0)$ is feasible. 
Hence $\thetabalhat(G)$ is the minimum~$\lambda$ such that there exist~$t,s \in \R$ satisfying  
$$
\begin{cases}  \lambda+tr-2n\geq 0, & \\
  \lambda-tr+2ns\geq 0, & \\  \lambda \geq |t| \lambda_2. \end{cases}
$$ 
Now we may eliminate the second equation, as we can choose~$s$ such that~$\lambda-tr+2ns=0$. So~$\thetabalhat(G)$ is the minimum~$\lambda$ such that there exists~$t \in \R$ satisfying 
$$
\begin{cases}  \lambda+tr-2n\geq 0, &  \\
  \lambda \geq |t| \lambda_2. \end{cases}
$$ 
This implies $\lambda \geq 2n-tr$ and $\lambda \geq t \lambda_2$. Hence $\lambda$ is above the point of intersection, where~$2n-tr=t\lambda_2$, i.e.,  $t={2n\over \lambda_2+r}$, which implies  $\lambda \geq t\lambda_2 = \frac{2n\lambda_2}{\lambda_2 +r}$. Setting $t={2n\over \lambda_2+r}$ and~$\lambda=\lambda_2 t$ is feasible, so the optimum $\lambda$ is $\frac{2n\lambda_2}{\lambda_2+r}$, which completes the proof of Proposition \ref{prop:symbalhat}. \hfill\qed

\medskip 
Let us point out that it follows from the above proof  that one may set  $v=0$ in the program (\ref{eq:thetabalhat}) defining $\thetabal(G)$; this observation was mentioned just after Proposition \ref{prop:symbalhat}.

\section*{Acknowledgments.}
We thank Frank Vallentin for communicating his eigenvalue bound on the parameter $h(\cdot)$ and its  application  to bounding product-free sets in finite groups. We are grateful to Aida Abiad for discussions and bringing the paper \cite{CK03} to our attention, and to Hajime Tanaka for letting us know about the papers \cite{ST2014,STT2017}.  This research was carried out while S.\ Polak and L.F.\ Vargas were with Centrum Wiskunde~$\&$ Informatica, Amsterdam. This work is supported by the European Union's Framework Programme for Research and Innovation Horizon
2020 under the Marie Skłodowska-Curie Actions Grant Agreement No. 813211  (POEMA).


\end{document}